\pdfoutput=1
\documentclass[12pt]{article}
\usepackage[utf8]{inputenc}
\usepackage{graphicx}
\graphicspath{{figures/}}
\usepackage[percent]{overpic}
\usepackage{xcolor}
\usepackage{natbib}
\usepackage[toc,page]{appendix}
\usepackage{geometry}
\usepackage{microtype}
\usepackage{setspace}
\usepackage{fullpage}
\usepackage{tabularx}
\usepackage{hyperref}%

\usepackage{amsmath, amssymb, amsthm}
\theoremstyle{plain}
\newtheorem{theorem}{Theorem}[section]
\newtheorem{lemma}[theorem]{Lemma}
\newtheorem{proposition}[theorem]{Proposition}
\newtheorem{corollary}[theorem]{Corollary}
\theoremstyle{definition}

\newtheorem{assumption}{Assumption}
\newtheorem{example}[theorem]{Example}

\newtheorem{manualtheoreminner}{Assumption}
\newenvironment{manualtheorem}[1]{%
    \IfBlankTF{#1}
    {\renewcommand{\themanualtheoreminner}{\unskip}}
    {\renewcommand\themanualtheoreminner{#1}}%
  \manualtheoreminner
}{\endmanualtheoreminner}

\theoremstyle{remark}
\newtheorem*{remark}{Remark}

\newcommand{\E}{\mathsf{E}}

\newcommand{\pr}{\mathsf{P}}

\newcommand{\tsp}{\mathrm{\scriptscriptstyle T}}
\newcommand{\rN}{\mathrm{N}}

\newcommand{\emax}{\lambda_{\mathrm{max}}}
\newcommand{\emin}{\lambda_{\mathrm{min}}}

\newcommand{\ncm}{\newcommand}
\ncm{\bfm}[1]{\mbox{\boldmath $#1$}}
\ncm{\sbfm}[1]{\mbox{\scriptsize\boldmath $#1$}}
\ncm{\scr}[1]{\mbox{\scriptsize #1}}
\ncm{\bfmscr}[1]{\mbox{\scriptsize{\boldmath $ #1$}}}

\ncm{\R}{{\mathbb{R}}}
\ncm{\Z}{{\mathbb{Z}}}

\ncm{\cA}{{\cal A}}
\ncm{\cB}{{\cal B}}
\ncm{\cC}{{\cal C}}
\ncm{\calF}{{\cal F}}
\ncm{\cD}{{\cal D}}
\ncm{\cG}{{\cal G}}
\ncm{\cN}{{\cal N}}
\ncm{\cI}{{\cal I}}
\ncm{\cJ}{{\cal J}}
\ncm{\cL}{{\cal L}}
\ncm{\cU}{{\cal U}}
\ncm{\cV}{{\cal V}}
\ncm{\cW}{{\cal W}}
\ncm{\cT}{{\cal T}}
\ncm{\cX}{{\cal X}}
\ncm{\cY}{{\cal Y}}
\ncm{\cZ}{{\cal Z}}
\ncm{\cQ}{{\cal Q}}
\ncm{\cR}{{\cal R}}
\ncm{\cP}{{\cal P}}
\ncm{\cO}{{\cal O}}
\ncm{\cPzer}{{\cal P}_0}
\ncm{\cPone}{{\cal P}_1}
\ncm{\cPk}{{\cP_{\mbox{\scr{known}}}}}
\ncm{\cF}{{\cal F}}
\ncm{\cE}{{\cal E}}
\ncm{\cMD}{{\cal MD}}
\ncm{\tcV}{\tilde{\cal V}}
\ncm{\cCobs}{{\cal C}_{\scr{obs}}}

\ncm{\Om}{\Omega}
\ncm{\om}{\omega}
\ncm{\va}{\varepsilon}
\ncm{\de}{\delta}
\ncm{\De}{\Delta}
\ncm{\ga}{\gamma}
\ncm{\Ga}{\Gamma}
\ncm{\la}{\lambda}
\ncm{\ka}{\kappa}
\ncm{\Si}{\Sigma}
\ncm{\La}{\Lambda}
\ncm{\eps}{\epsilon}

\ncm{\bY}{\bfm{Y}}
\ncm{\bI}{\bfm{I}}
\ncm{\bZ}{\bfm{Z}}
\ncm{\bG}{\bfm{G}}
\ncm{\bT}{\bfm{T}}
\ncm{\bM}{\bfm{M}}
\ncm{\bv}{\bfm{v}}
\ncm{\bw}{\bfm{w}}
\ncm{\bwpr}{\bfm{w}^\prime}
\ncm{\bhp}{\bfm{h}^\prime}
\ncm{\bm}{\bfm{m}}
\ncm{\bh}{\bfm{h}}
\ncm{\bg}{\bfm{g}}
\ncm{\be}{\bfm{e}}
\ncm{\bxi}{\bfm{\xi}}
\ncm{\bom}{\bfm{\om}}

\ncm{\sbM}{\sbfm{M}}
\ncm{\sbw}{\sbfm{w}}

\ncm{\heta}{\hat{\eta}}
\ncm{\hde}{\hat{\delta}}
\ncm{\hth}{\hat{\theta}}
\ncm{\hs}{\hat{s}}
\ncm{\hI}{\hat{I}}
\ncm{\hw}{\hat{w}}
\ncm{\htau}{\hat{\tau}}
\ncm{\hxi}{\hat{\xi}}

\ncm{\mast}{m^\ast}
\ncm{\cast}{c^\ast}
\ncm{\fast}{f^\ast}
\ncm{\siast}{\si^\ast}
\ncm{\psiast}{\psi^\ast}
\ncm{\tsiast}{\tilde{\si}^\ast}
\ncm{\alfast}{\alpha^\ast}
\ncm{\tkaast}{\tilde{\kappa}^\ast}

\ncm{\Depr}{\Delta^\prime}
\ncm{\wpr}{w^\prime}
\ncm{\jp}{j^\prime}
\ncm{\ip}{i^\prime}
\ncm{\Ip}{I^\prime}
\ncm{\kp}{k^\prime}
\ncm{\tpr}{t^\prime}
\ncm{\hp}{h^\prime}
\ncm{\lp}{l^\prime}
\ncm{\np}{n^\prime}
\ncm{\Gp}{G^\prime}
\ncm{\Sp}{S^\prime}
\ncm{\Rp}{R^\prime}
\ncm{\dep}{\delta^\prime}
\ncm{\Mp}{M^\prime}
\ncm{\Cp}{C^\prime}
\ncm{\ap}{a^\prime}
\ncm{\vp}{v^\prime}
\ncm{\up}{u^\prime}
\ncm{\npr}{n^\prime}
\ncm{\xp}{x^\prime}
\ncm{\yp}{y^\prime}
\ncm{\zp}{z^\prime}
\ncm{\phpr}{\phi^\prime}

\ncm{\wbis}{w^{\prime\prime}}

\ncm{\tih}{\tilde{h}}
\ncm{\tZ}{\tilde{Z}}
\ncm{\tY}{\tilde{Y}}
\ncm{\tX}{\tilde{X}}
\ncm{\tz}{\tilde{z}}
\ncm{\tI}{\tilde{I}}
\ncm{\tJ}{\tilde{J}}
\ncm{\tmu}{\tilde{\mu}}
\ncm{\tOm}{\tilde{\Omega}}
\ncm{\tnu}{\tilde{\nu}}
\ncm{\tsi}{\tilde{\sigma}}
\ncm{\tth}{\tilde{\theta}}
\ncm{\tal}{\tilde{\alpha}}
\ncm{\tbeta}{\tilde{\beta}}
\ncm{\tde}{\tilde{\delta}}
\ncm{\txi}{\tilde{\xi}}
\ncm{\tmathV}{\tilde{\V}}
\ncm{\tV}{\tilde{V}}
\ncm{\tW}{\tilde{W}}
\ncm{\ta}{\tilde{a}}
\ncm{\tv}{\tilde{v}}
\ncm{\tw}{\tilde{w}}
\ncm{\tu}{\tilde{u}}
\ncm{\twpr}{\tilde{w}^\prime}
\ncm{\td}{\tilde{d}}
\ncm{\tp}{\tilde{p}}
\ncm{\tf}{\tilde{f}}
\ncm{\tn}{\tilde{n}}
\ncm{\tR}{\tilde{R}}
\ncm{\tS}{\tilde{S}}
\ncm{\tC}{\tilde{C}}
\ncm{\tL}{\tilde{L}}
\ncm{\tl}{\tilde{l}}
\ncm{\tP}{\tilde{P}}
\ncm{\tSmath}{\tilde{\mathbb{S}}}
\ncm{\tT}{\tilde{T}}
\ncm{\tK}{\tilde{K}}
\ncm{\tM}{\tilde{M}}
\ncm{\tka}{\tilde{\ka}}
\ncm{\tva}{\tilde{\va}}
\ncm{\tla}{\tilde{\la}}
\ncm{\tLa}{\tilde{\La}}
\ncm{\tpi}{\tilde{\pi}}
\ncm{\tpsi}{\tilde{\psi}}
\ncm{\tbe}{\tilde{\bfm{e}}}
\ncm{\tbom}{\tilde{\bfm{\om}}}
\ncm{\tbxi}{\tilde{\bfm{\xi}}}

\ncm{\baA}{\bar{A}}
\ncm{\baU}{\bar{U}}
\ncm{\baV}{\bar{V}}
\ncm{\baW}{\bar{W}}
\ncm{\bacW}{\bar{\cW}}
\ncm{\bacV}{\bar{\cV}}
\ncm{\bau}{\bar{u}}
\ncm{\bav}{\bar{v}}
\ncm{\baf}{\bar{f}}
\ncm{\baw}{\bar{w}}
\ncm{\baZ}{\bar{Z}}
\ncm{\baY}{\bar{Y}}
\ncm{\baS}{\bar{S}}
\ncm{\baH}{\bar{H}}
\ncm{\baI}{\bar{I}}
\ncm{\baD}{\bar{D}}
\ncm{\baC}{\bar{C}}
\ncm{\bah}{\bar{h}}
\ncm{\bal}{\bar{l}}
\ncm{\bam}{\bar{m}}
\ncm{\bae}{\bar{e}}
\ncm{\bacR}{\bar{{\cal R}}}
\ncm{\bacP}{\bar{{\cal P}}}
\ncm{\baka}{\bar{\kappa}}
\ncm{\bamu}{\bar{\mu}}
\ncm{\banu}{\bar{\nu}}
\ncm{\bala}{\bar{\la}}
\ncm{\bachi}{\bar{\chi}}
\ncm{\baga}{\bar{\ga}}
\ncm{\bapsi}{\bar{\psi}}

\ncm{\chnu}{\check{\nu}}
\ncm{\chla}{\check{\la}}
\ncm{\chLa}{\check{\La}}
\ncm{\chZ}{\check{Z}}
\ncm{\chY}{\check{Y}}
\ncm{\chX}{\check{X}}

\ncm{\rmp}{\mbox{p}}
\ncm{\rmq}{\mbox{q}}

\ncm{\scrp}{\scr{p}}
\ncm{\scrq}{\scr{q}}

\ncm{\Leq}{\, \stackrel{\cal L} =}
\ncm{\Lto}{\, \stackrel{\cal L} \longrightarrow}
\ncm{\pto}{\, \stackrel{p} \longrightarrow}
\ncm{\asto}{\, \stackrel{\rm a.s.} \longrightarrow}
\ncm{\Cov}{\mbox{Cov}}
\ncm{\Var}{\mbox{Var}}
\ncm{\sameord}{\stackrel{\cup}{{\scriptstyle \cap}}}

\ncm{\ith}{i^{\scr{th}}}
\ncm{\jth}{j^{\scr{th}}}
\ncm{\kth}{k^{\scr{th}}}
\ncm{\lth}{l^{\scr{th}}}
\ncm{\bth}{b^{\scr{th}}}
\ncm{\Id}{\mbox{Id}}
\ncm{\hId}{\widehat{\mbox{Id}}}
\ncm{\IBD}{\mbox{IBD}}
\ncm{\kamin}{\kappa_{\scr{min}}}
\ncm{\kamax}{\kappa_{\scr{max}}}
\ncm{\Ts}{\T_s}
\ncm{\Spairs}{S_{\scr{pairs}}}
\ncm{\Sgprs}{S_{\scr{g-prs}}}

\ncm{\beq}{\begin{equation}}
\ncm{\eeq}{\end{equation}}
\ncm{\beqr}{\begin{eqnarray}}
\ncm{\eeqr}{\end{eqnarray}}
\ncm{\beqrn}{\begin{eqnarray*}}
\ncm{\eeqrn}{\end{eqnarray*}}
\ncm\rthm[1]{\ref{#1}}
\ncm\lb[1]{\label{#1}}
\ncm\re[1]{(\ref{#1})}
\ncm{\slut}{
  {\unskip\nobreak\hfill\penalty100\hskip1em\vadjust{}\nobreak
  \hfill\mbox{$\Box$}\parfillskip=0pt\finalhyphendemerits=0}}

\title{Asymptotics for likelihood ratio tests of boundary points with singular information and unidentifiable nuisance parameters}
\author{Karl Oskar Ekvall$^{\dagger}$\thanks{Corresponding author. Department of Statistics, University of Florida, Gainesville, Florida, USA. E-mail: k.ekvall@ufl.edu}, Ola H\"ossjer$^{\star}$, Matteo Bottai$^{\ddagger}$, J.M. Patrik Albin$^{\mathsection}$ \\
{\normalsize $^\dagger$Department of Statistics, University of Florida} \\
{\normalsize $^\star$Department of Mathematics, Stockholm University} \\
{\normalsize $^\ddagger$Division of Biostatistics, Institute of Environmental
Medicine, Karolinska Institutet}\\
{\normalsize $^\mathsection$Department of Mathematical Sciences, Chalmers University of Technology}}
\date{}

\begin{document}

\maketitle

\begin{abstract}
We establish the asymptotic distribution of likelihood ratio tests when
nuisance parameters are unidentifiable under the null hypothesis, parameters
of interest lie on the boundary, and the information matrix of the
identifiable parameters may be singular. Mixture models, change point models
and genetic linkage analysis exhibit all three features simultaneously. Under
the null hypothesis the limit is the supremum of a chi-bar-squared process,
whose marginal distributions are mixtures of chi-squared distributions with
weights depending on the nuisance parameter; under local alternatives it is
the supremum of a noncentral chi-bar-squared process. Unlike prior work on
singular information, where singularity stems from the parameter of interest
and changes the limit's form, here it occurs only at isolated values of the
nuisance parameter, and under additional conditions the limit keeps its
nonsingular form. Classical boundary results and application-specific results
for mixtures, change points and linkage analysis are recovered and extended.
\end{abstract}


\section{Introduction}\lb{Sec:Intro}

Under regularity conditions, the likelihood ratio test (LRT) statistic for
testing the value of a structural parameter $\xi$ has an asymptotic
$\chi_p^2$-distribution under the null hypothesis, with the degrees of freedom
$p$ determined by the number of restrictions imposed by the null
\citep{wilks1938largesample}. In many applications, however, such regularity
conditions do not hold and, consequently, inference based on the
$\chi_p^2$-distribution is unreliable. We consider settings in which three
nonregularities can be present simultaneously: (i) some nuisance parameters
$t$, possibly infinite-dimensional, are present only under the alternative
hypothesis, and hence unidentifiable under the null hypothesis;
(ii) the remaining parameters $\xi$ (structural parameters of interest and
possibly some additional nuisance parameters) lie on the boundary of the
parameter space under the null hypothesis; and (iii) the information matrix for
$\xi$ may be singular at some values of $t$, under the null hypothesis. Existing theory treats these features separately, in particular models,
or only for independent and identically distributed observations under the
null hypothesis.

Important examples include LRTs for mixture models \citep{daga1999,
chen2001large, chen2003tests}, change point models \citep{anpl1995}, and genetic
linkage analysis \citep{sham1998statistics}. In mixture models, parameters of a
component distribution are unidentifiable under a null hypothesis assigning zero
weight to that component. Moreover, the mixing distribution may only be known to
belong to an infinite-dimensional family. For change point models, the location
of the change point of a stochastic process cannot be identified under the null
hypothesis that distribution of the process is the same before and after the
change point. In genetic linkage analysis, the null hypothesis is that no
disease gene is present along a chromosome of length $T$; the alternative is
that a gene is located at some position $\tau \in [0,T]$. The gene location
$\tau$ is then unidentifiable under the null. The parameter of interest is the
genetic variance attributable to the disease gene: it is nonnegative and equals
zero under the null, placing it on the boundary of the parameter space, and
the Fisher information can be singular. Singular information arises when part of the parameter of interest has no
effect on the distribution, as when the mixture components coincide in Examples
\ref{ex:gauss_mix_1}--\ref{ex:gauss_mix_2} or the change point is at the endpoint
in Example \ref{Ex:CP}. Together, these examples exhibit all three
nonregularities. Further
details on these examples, including references to previous work, are in
Sections \ref{sec:mix_ex} and \ref{sec:linkage}. 

Other problems for which the LRT statistic has nonstandard asymptotics
include tests for the number of components in segmented regression
\citep{fe1975, da1987}, for stochastic coefficients in regression
\citep{watson1985}, for the relevance of explanatory variables in nonlinear
regression and for threshold effects in autoregression \citep{anpl1995}, for a
discrete frequency component in a time series \citep{da1987}, for the order of
an ARMA process \citep{daga1999}, for conditional heteroskedasticity
\citep{bera1992} and for covariate effects in frailty regression for survival
data \citep{song2009}. In the last, the baseline hazard function is an
infinite-dimensional nuisance parameter, unidentifiable under the null
hypothesis.

Our main results give conditions under which the likelihood ratio statistic
converges in law to the supremum of a stochastic process, and we give a geometric characterization of
the limiting process. Under the null hypothesis and appropriate regularity
conditions, this process is a $\bar{\chi}^2$-process, with marginal
distributions that are mixtures of $\chi^2$-distributions and weights determined
by the local geometry of the parameter space. Under contiguous (local)
alternatives, the process is a noncentral $\bar{\chi}^2$-process, with marginal
distributions that are mixtures of truncated noncentral $\chi^2$-distributions,
with weights not only determined by the local geometry of the parameter space,
but also by the sequence of local alternatives. For composite null hypotheses whose null set has a linear tangent space, the
contribution of the identifiable nuisance parameters in $\xi$ cancels and the limit is a
supremum of squared projections onto cones orthogonal to that space.
Section \ref{sec:polar_ex} shows that writing a regular problem in polar
coordinates, which makes the angle unidentifiable under the null hypothesis,
leaves our limit unchanged.

The closest existing results are those of \citet{andrews2001testing}, for
boundary parameters with unidentifiable nuisance parameters, uniformly
nonsingular information and null hypotheses that fix some components of
$\xi$. We allow the information to be singular at some values of the nuisance
parameter $t$, and null hypotheses that restrict $\xi$ to a general subset.
\citet{liu2003asymptotics} allow singular information, in a framework indexed
by generalized score functions rather than by the parameter, for independent
and identically distributed observations under the null hypothesis. Without
that assumption, we cover the triangular array of Example \ref{Ex:CP} and the
several family types of Section \ref{Sec:ManyFT}, and we treat contiguous
alternatives. Where both frameworks apply, the limits agree, as shown in
Section \ref{sec:singular}. Because our limit process is indexed by the parameter $t$, its
covariance function, mixture weights and, under contiguous alternatives, mean
function follow from the score covariance across values of the nuisance
parameter $t$ and the tangent cone of the parameter space for $\xi$.

Our findings also recover, as special cases, results obtained by removing some
of the complicating features. This includes classical results for boundary
inference without nuisance parameters
\citep{chernoff1954distribution,self1987asymptotic,geyer1994asymptotics}. In
addition, our framework incorporates results for testing whether a single
parameter of interest is zero or positive, in the presence of nuisance
parameters, when either all \citep{da1977, da1987} or some \citep{risk2005} of
these nuisance parameters are unidentifiable under the null hypothesis. For linkage analysis with additive and dominance genetic variance components,
we obtain the asymptotic null distribution of the likelihood ratio statistic as
the supremum of a $\bar{\chi}^2$-process along the chromosome, and compute its
weights for common pedigree types. This extends the Gaussian and
$\chi^2$-process limits of \citet{lander1989mapping},
\citet{feingold1993gaussian} and \citet{dupuis1995statistical}.

The source of singularity here differs from that in
\citet{rotnitzky2000likelihoodbased}, \citet{ekvall2022confidence}, \citet{bottai2003confidence},
and \citet{guedon2024bootstrap}. There, the rank
of the information matrix is determined by the parameter of interest, whereas
here it is a function of the nuisance parameter that is unidentifiable under the null hypothesis. In
\citet{rotnitzky2000likelihoodbased}, singularity fundamentally changes the form
of the asymptotic distribution, the likelihood is bimodal near the singular
point, and only information matrices of rank one less than full are treated.
Here, the rank at singular points is arbitrary and, under certain smoothness
assumptions, the limit distribution has the same form as in the nonsingular
case. Under these assumptions the singularity complicates the analysis but not the conclusion.


\section{Quadratic Approximation and \texorpdfstring{$\bar{\chi}^2$}{chi-bar-square} Process}\lb{Sec:QA}

In this section we introduce basic notation and explain the asymptotic framework for generalized likelihood ratio testing that will be employed in subsequent sections. The random variable $\Lambda_n$ of interest (of which the LRT statistic is a special case) is the supremum of a stochastic process, for which we define an asymptotic quadratic approximation. This representation of $\Lambda_n$ will then be used in Section \ref{sec:asy_dist} to prove weak convergence of $\Lambda_n$ when $n$ (typically the sample size) grows. 

We assume an underlying probability space $(\Omega, \mathcal{S}, \pr)$ and call
any subset of $\Omega$ an event, whether it is measurable or not. Similarly, we
call any function mapping elements of $\Omega$ to some set $\cD$ a random element of
that set. In particular, if $\cD=\cF(\Theta)$ is the space of real-valued functions whose
domain is the parameter space $\Theta$, i.e.\ the set of parameters $\theta$ of a statistical model, we let $\ell_n: \Omega \to \cF(\Theta)$ be a random element on that set. That is, for every $\omega \in \Omega$, $\ell_n(\cdot) = \ell_n(\cdot; \omega)$ is a real-valued function on $\Theta$. Assume $\theta \in\Theta$ satisfies $\theta = (\xi, t)$, where $t\in \cT$ for some arbitrary set
$\mathcal{T}$ and $\xi \in \Xi(t) \subseteq \R^{p}$, corresponding to a representation $\Theta = \cup_{t\in\cT} \Xi(t)\times \{t\}$ of the parameter space. For now $\ell_n$ can be an
arbitrary objective function of $\theta$, but later we assume it is a log-likelihood, with $\xi$
and $t$ those parameters that are identifiable and not identifiable under a null hypothesis, when the value of $\theta$ is tested. 

A key object in our theory is the stochastic process $X_n$ indexed by
$\mathcal{T}$ and defined by
\begin{equation} \label{eq:X_n-process}
    X_n(t) = \sup_{\xi \in \Xi(t)}2\{\ell_n(\xi, t) - \ell_n(\xi_0, t)\},
\end{equation}
where $\xi_0$ is a particular value of $\xi$. Later, $\xi_0$ will be the null
value of $\xi$. This corresponds to hypothesis testing with $H_0: \xi = \xi_0$ versus $H_1: \xi \in \Xi \setminus \{\xi_0\}$, where $\Xi = \cup_{t\in\cT}\Xi(t)$ is the parameter space for $\xi$. Throughout the article, we will assume that $t$ is a nuisance parameter that is unidentifiable under the null, meaning that $\ell_n(\xi_0, t)$ does not depend
on $t$. When $\ell_n$ is the log likelihood, the likelihood ratio test statistic for testing $H_0$ against $H_1$ is
\begin{align} \label{eq:lrt_stat}
    \Lambda_n = \sup_{t \in \mathcal{T}} X_n(t) .
\end{align}
As we will see, splitting $\theta$ into two parts $\xi$ and $t$ that correspond to identifiable and non-identifiable components under $H_0$, will facilitate a geometric interpretation of the limit process of $X_n$ as $n\to\infty$, and hence also of the limit distribution of the LRT statistic $\Lambda_n$. In order to see this, we will discuss properties of the process $X_n$ in the more general framework when $\ell_n$ is not necessarily
the log-likelihood and $\Lambda_n$ not necessarily a LRT statistic. For each $t
\in \mathcal{T}$ and $n \in \{1, 2,\dots\}$, let $U_n(t)$ denote a random vector
and let $I(t) \in \R^{p\times p}$ be deterministic and positive definite.
Consider a quadratic expansion 
\begin{align} \label{eq:ln_expand}
    \ell_n(\xi, t) - \ell_n(\xi_0,t) &=  \sqrt{n}(\xi - \xi_0)^\tsp  U_n(t) - \frac{n}{2} (\xi - \xi_0)^\tsp I(t) (\xi - \xi_0) + n\Vert I(t)^{1/2} (\xi - \xi_0)\Vert_2^2 R_n(\xi, t),
\end{align}
where ${}^\tsp$ denotes transpose, $\Vert \xi\Vert_2^2 = \xi^\tsp \xi $ is the
squared Euclidean norm in $\R^p$, and $R_n(\xi, t)$ is a remainder term to be
specified. We call $U_n$ and $I$ a generalized score and information matrix,
respectively. When $\ell_n$ is a log-likelihood, they can often be the usual
score and Fisher information matrix for $\xi$, evaluated at $\xi_0$. In that
case, for a fixed $t$, \eqref{eq:ln_expand} is the usual quadratic expansion
used to derive the asymptotic distribution of the likelihood ratio test
\citep[see, for example,][]{serfling2002approximation}. Here, by contrast, it
will be important to control the behavior of $U_n$, $I$, and $R_n$ as elements
of appropriate spaces of functions on $\mathcal{T}$. Moreover, we will later
consider cases where $I(t)$ is singular for some $t$.

Put $\delta = n^{1/2}A(t)^\tsp(\xi - \xi_0)$ and $\tilde{\ell}_n(\delta, t) =
\ell_n(\xi_0 + n^{-1/2}A(t)^{-\tsp}\delta, t)$, where $A(t)A(t)^\tsp = I(t)$. 
For example, $A(t)$ can be a Cholesky root or the symmetric matrix with the same
eigenvectors as $I(t)$ and square-root eigenvalues. It follows that the set of
possible values of $\delta$ is $\Delta_n(t) = n^{1/2}A(t)^\tsp(\Xi(t)-\xi_0)$.
Let also $Z_n(t) = A(t)^{-1}U_n(t)$ and $Q_n(\delta, t) = 2\delta^\tsp Z_n(t) -
\Vert \delta\Vert_2^2$. Then, by \eqref{eq:ln_expand}, the function
$\tilde{\ell}_n$ defined by $\tilde{\ell}_n(\delta, t) = \ell_n(\xi_0 +
n^{-1/2}A(t)^{-\tsp}\delta, t)$ satisfies
\begin{equation} \label{eq:ln_tilde_expand}
    2\{\tilde{\ell}_n(\delta, t) - \tilde{\ell}_n(0, t)\} = Q_n(\delta, t) + 2\Vert \delta\Vert_2^2 \tilde{R}_n(\delta, t),
\end{equation}
where $\tilde{R}_n(\delta, t) = R_n(\xi_0 + n^{-1/2}A(t)^{-\tsp}\delta, t)$. Thus,
\begin{equation} \label{eq:X_n-process-delta}
X_n(t) =  \sup_{\delta \in \Delta_n(t)} \left\{ Q_n(\delta, t) + 2\Vert \delta\Vert_2^2 \tilde{R}_n(\delta, t)\right\}.
\end{equation}
Under conditions in the next section, $Z_n \Lto Z$ and $\Delta_n(t) \to
\Delta(t)$ for all $t\in\cT$, from which $X_n \Lto X$ follows. The limit set $\Delta(t)$ encodes
the boundary constraints on $\xi$ around $\xi_0$ at $t$: if $\xi_0$ is in the interior of $\Xi(t)$,
then $\Delta(t) = \R^p$, whereas if $\xi_0$ lies on the boundary, $\Delta(t)$ is
typically a proper convex cone, i.e.\ a convex cone with a nonempty interior such that $\{z,-z\}\subset \Delta(t)$ implies $z=0$. For each $t$, the limit process satisfies
\begin{equation}\label{eq:X-process}
    X(t) = \sup_{\delta \in \Delta(t)} Q(\delta, t);~~  Q(\delta, t) = 2\delta^\tsp Z(t) - \Vert \delta\Vert_2^2.
\end{equation}
Intuitively, then, we expect $\Lambda_n\Lto\Lambda$ as $n\to\infty$, where 
\beq
\Lambda = \sup_{t \in \mathcal{T}} X(t),
\lb{La}
\eeq
if $Z_n$, $\Delta_n$, and $\tilde{R}_n(\delta, \cdot)$ converge to $Z$, $\Delta$,
and $0$, respectively, as functions on $\mathcal{T}$. Under certain conditions, we have
\begin{align}\lb{XtProj}
X(t) = \Vert Z(t)\Vert_2^2 - \inf_{\delta \in \Delta(t)}\Vert Z(t) - 
\delta\Vert_2^2 = \Vert Z(t)\Vert_2^2 - 
\Vert Z(t) - P_{\Delta(t)} Z(t)\Vert_2^2 = \Vert P_{\Delta(t)} Z(t)\Vert_2^2.
\end{align}
Specifically, the first equality is always true, while the second holds if the
infimum is attained by some $\delta \in \Delta(t)$, in which case it defines the
projection. If $\Delta(t)$ is closed, the projection $P_{\Delta(t)}Z(t)$ of $Z(t)$ exists and
is unique Lebesgue-almost everywhere \citep[Proposition
4.2]{geyer1994asymptotics}. The third equality holds if $\Delta(t)$ is a closed,
convex cone, with $\Vert P_{\Delta(t)}z\Vert_2^2=\sup_{u\in\Delta(t);\Vert u \Vert_2=1} \max(u^\tsp z,0)^2$. 

Suppose in addition that $Z$ is a centered Gaussian process whose marginal
covariance matrix is the identity matrix $\Id_p$ of order $p$ for all $t\in\cT$.
This will typically be the case when $\ell_n(\xi,t)$ is a log-likelihood and the
null hypothesis $\xi=\xi_0$ holds. Then, for each fixed $t$, $X(t)$ has a
$\bar{\chi}^2$-distribution \citep{shapiro1985asymptotic,shapiro1988unified}.
That is, for some nonnegative weights $w_0(t), \dots, w_p(t)$ summing to one
(whose values depend on $\Delta(t)$)
\beq
X(t) \sim \sum_{j=0}^p w_j(t) \chi_j^2.
\lb{chibar}
\eeq
Then, we say $X$ is a $\bar{\chi}^2$-process indexed by $\mathcal{T}$.
Now suppose $\ell_n(\xi,t)$ is a log-likelihood and consider a sequence of
contiguous (or local) alternatives $\theta_n=(\xi_n,t_0)$ with $t_0\in\cT$ and $\Xi \ni \xi_n\to\xi_0$. Then, for an appropriate rate of convergence of $\xi_n$ towards $\xi_0$, $X_n$ converges weakly to a
process $X$ whose marginal distribution $X(t)$ is a mixture of truncated,
noncentral $\chi_j^2$-distributions, and we refer to $X$ as a noncentral
$\bar{\chi}^2$-process indexed by $\mathcal{T}$. The challenge here is not
merely to show $X_n(t) \Lto X(t)$ for fixed $t \in \mathcal{T}$, but rather to
prove weak convergence $X_n \Lto X$ on an appropriate space of functions on $\cT$.

\section{Asymptotic Distributions} \label{sec:asy_dist}

In this section we state the main results of the article, that $\Lambda_n$ (the generalized LRT statistic) under suitable assumptions converges weakly to the supremum $\Lambda$ of a $\bachi^2$-process as $n\to\infty$.

\subsection{Asymptotics with uniformly positive definite information} \label{sec:pd_inf}

We first consider the case where $I(t)$ is uniformly positive definite over
$\mathcal{T}$, meaning its smallest eigenvalue is bounded away from zero
uniformly in $t$. This simplifies the discussion and forms the basis for
Section~\ref{sec:singular}, where $I(t)$ may be singular.

Let $F^p(\mathcal{T})$ be the set of bounded $\R^{p}$-valued functions on
$\mathcal{T}$ (with $F^1 = F$); that is, $z\in F^p(\mathcal{T})$ if and only if
$\Vert z\Vert_{\mathcal{T}} = \sup_{t \in \mathcal{T}}\Vert z(t)\Vert_2 <
\infty$, where $\Vert \cdot \Vert_2$ is the Euclidean norm and $\Vert \cdot
\Vert_{\mathcal{T}}$ is defined by the first equality. We equip
$F^{p}(\mathcal{T})$ with the corresponding metric and Borel $\sigma$-algebra.
Because some random elements we shall consider need not be measurable, some
results concern outer probabilities and expectations, denoted by $\pr^*$ and
$\E^*$, respectively; we refer the reader to, for example,
\citet{vandervaart2023weak} for definitions and properties. We denote generic
constants by $c_i, C_i$, with $c_i$ small and $C_i$ large, $i \in \{1, 2,
\dots\}$.

\begin{assumption} \label{assm:Xn_bdd}
    For every $n$, $X_n$ is a random element of $F(\mathcal{T})$.
\end{assumption}

\begin{assumption} \label{assm:equicont}
    Equation \eqref{eq:ln_expand} holds with a remainder term that is
    stochastically, uniformly equicontinuous at $\xi_0$ in the sense that, for every $c_1,
    c_2 > 0$, there exists a $c_3 > 0$ such that
    \begin{align} \label{eq:equicont}
    \limsup_{n\to \infty} \pr^*\left(\sup_{t\in\mathcal{T}} 
    \sup_{\xi \in \bar{B}_{c_3}(\xi_0; t)} |R_n(\xi, t)| > c_1\right) < c_2,
    \end{align}
    where $\bar{B}_{c_3}(\xi_0; t) = \{\xi \in \Xi(t) : \Vert \xi - \xi_0\Vert_2
    \leq c_3\}$.
    
\end{assumption}

\begin{assumption} \label{assm:consistency}
    Approximate partial maximizers are uniformly consistent; that is, if for every $n$ and $t \in
    \mathcal{T}$, $\hat{\xi}_n(t)$ satisfies $\ell_n(\hat{\xi}_n(t), t) \geq
    \sup_{\xi \in \Xi(t)}\ell_n(\xi, t) - O_{p^*}(1)$, with the $O_{p^*}(1)$
    term not depending on $t$, then $\Vert \hat{\xi}_n -
    \xi_0\Vert_{\mathcal{T}} = o_{p^*}(1)$. 
\end{assumption}

\begin{assumption} \label{assm:inf_bound}
    For some constants $\underline{\kappa}$ and $\bar{\kappa}$,
    \[
       0 <  \underline{\kappa} \leq \inf_{t \in \mathcal{T}} \emin\{I(t)\} \leq 
       \sup_{t \in \mathcal{T}} \emax\{I(t)\} \leq \bar{\kappa} < \infty,
    \]
    where $\emin(\cdot)$ and $\emax(\cdot)$ are, respectively, the smallest and
    largest eigenvalues of the argument matrix. 
\end{assumption}

\begin{assumption} \label{assm:zn_conv}
    There is a tight, Borel-measurable $Z \in F^p(\mathcal{T})$ such
    that $Z_n \Lto Z$ in $F^p(\mathcal{T})$ as $n\to \infty$.
\end{assumption}

\begin{assumption} \label{assm:closed}
    For every $t\in \mathcal{T}$, $\Xi(t)$ is closed. 
\end{assumption}

\begin{assumption} \label{assm:cone}
    For every $t\in \mathcal{T}$, there is a closed cone $C(t) \subseteq \R^{p}$ such that
    \beq
        \lim_{s \downarrow 0} \sup_{t\in\mathcal{T}} d\left[s^{-1}\{\Xi(t) - \xi_0\}, C(t)\right] = 0,
        \lb{CtDef}
    \eeq
    where $d(\cdot, \cdot)$ is the distance for sets defined in Section
    \ref{app:proofs_sub}.
\end{assumption}

Assumption \ref{assm:Xn_bdd} ensures $X_n$ and its limit $X$ are elements of the
same metric space. It can be relaxed, for instance by working with a truncated version
$\min(X_n, n) \in F(\mathcal{T})$; assuming $\Lambda_n = O_{p^*}(1)$, we have
$\pr^*(\Lambda_n \leq n) \to 1$ and hence the truncation has no effect on the
asymptotic distribution. Neither $X_n$ nor $Z_n$ need be measurable.

In settings with independent and identically distributed (i.i.d.) observations,
Assumption \ref{assm:equicont} can often be verified with $U_n$ the score
function and $I$ the Fisher information matrix for one observation. This is formalized in
Proposition \ref{prop:assm_2_sufficient}.

If $\mathcal{T}$ is a singleton, which essentially corresponds to a setting
without nuisance parameters, Assumptions \ref{assm:equicont}--\ref{assm:cone}
are similar to common ones \citep[see, for example,][]{geyer1994asymptotics}. We
are mostly interested in cases where $\mathcal{T}$ is a more complicated set.
The results in this section make few explicit assumptions about $\mathcal{T}$,
though verifying Assumptions \ref{assm:equicont}--\ref{assm:cone} typically
requires some structure. For example, in some settings $I(\cdot)$ is a
continuous function from $\R^{q}$ to $\R^{p\times p}$, and then Assumption
\ref{assm:inf_bound} may require $\mathcal{T}$ to be a compact subset of $\R^q$.
Similarly, verifying Assumption \ref{assm:zn_conv} often amounts to showing a
class of functions indexed by $\mathcal{T}$ is Donsker, which in general
requires controlling the complexity of that class \citep[see, for
example,][Section 2]{vandervaart2023weak}. More generally, $Z$ need not be a
Gaussian process, but it will typically be so when $\ell_n$ is a log-likelihood.
In settings with i.i.d.\ observations, for example, $Z$ is often a Gaussian process,
which is centered under the null hypothesis $H_0: \xi=\xi_0$ and non-centered
under contiguous alternatives (see Section \ref{sec:ind}). 

Assumption \ref{assm:consistency} typically requires case-by-case verification,
as sufficient conditions depend on the specific model and parameter space.
Assumption \ref{assm:inf_bound} is common in the literature but will be relaxed
in Section \ref{sec:singular}, at the expense of making stronger assumptions
about $\mathcal{T}$, among other things. Assumption \ref{assm:closed} ensures
$\Delta_n(t)$ is closed, so that projections onto $\Delta_n(t)$ and $\Delta(t)$
exist. It can be relaxed if this can be ensured by other means. Assumption
\ref{assm:cone} is essentially a uniform version of Chernoff regularity
\citep{chernoff1954distribution}, which is common in the literature. It can only
hold if $\xi_0 \in \Xi(t)$ for all $t \in \mathcal{T}$. With $C(t)$ the closed cone given by
Assumption~\ref{assm:cone}, let $\Delta(t) = A(t)^\tsp C(t)$. This transformed closed cone $\Delta(t)$ is the set that appears in  the definition of $X(t)$ in \eqref{eq:X-process}. 

We are ready to state the main result of the section, along with two
illustrative corollaries. The proof invokes several lemmas, which we state and
discuss shortly. With the lemmas in hand, the proof is a straightforward
application of a continuous mapping theorem and Slutsky's lemma. Therefore, it
is deferred to the Appendix, along with the proofs of other
results stated without proof in the main text.

\begin{theorem} \label{thm:main_nonsing} Under Assumptions
\ref{assm:Xn_bdd}-\ref{assm:cone}, $X$ is a Borel-measurable random element of
$F(\mathcal{T})$ and $X_n \Lto X$ in $F(\mathcal{T})$ as $n\to\infty$.
\end{theorem}

\begin{corollary} \label{cor:lrt_nonsing} Under Assumptions
\ref{assm:Xn_bdd}-\ref{assm:cone}, and assuming $t \mapsto \ell_n(\xi_0, t)$ is
constant, $\Lambda = \Vert X\Vert_{\mathcal{T}}$ in \re{La} is Borel-measurable.
Moreover, for $\Lambda_n = \Vert X_n \Vert_{\mathcal{T}}$ in \re{eq:lrt_stat},
$\Lambda_n \Lto \Lambda$ as $n\to\infty$.
\end{corollary}

Corollary \ref{cor:lrt_nonsing} is an immediate consequence of Theorem
\ref{thm:main_nonsing} and the Continuous Mapping Theorem. Unless we also assume
$C(t)$ is convex for every $t \in \mathcal{T}$, $X$ need not be a
$\bar{\chi}^2$-process, as discussed following \eqref{eq:X-process}.

\begin{remark}
Our principal application of Theorem \ref{thm:main_nonsing} and
Corollary \ref{cor:lrt_nonsing} is when $\ell_n(\xi,t)$ is a log-likelihood and
$\ell_n(\xi_0,t)$ does not depend on $t\in\cT$. Then $\Lambda_n$ is the
likelihood ratio test (LRT) statistic for testing the simple null hypothesis
$H_0:\xi=\xi_0$ against $H_1:\xi\in\Xi\setminus\{\xi_0\}$. If $\xi=\xi_0$, the
convergences in Theorem \ref{thm:main_nonsing} and
Corollary \ref{cor:lrt_nonsing} provide weak convergence of $X_n$ and the LRT
statistic $\Lambda_n$ under $H_0$, respectively. In particular, the
significance level $\alpha_n(c) = \pr_{(\xi_0,t)}(\La_n\ge c)$ of the LRT with
threshold $c>0$ is independent of $t$ and
\beq
\alpha_n(c) \to \alpha(c) = \pr_{(\xi_0,t)}(\La \ge c)\mbox{ as }n\to\infty
\lb{alphaConv}
\eeq
at all continuity points $c$ of the distribution function $F_\La$ of $\La$
under $H_0$. Theorem \ref{thm:main_nonsing} can also be applied to weak
convergence of $X_n$ under contiguous alternatives $\theta_n=(\xi_n,t_0)$. Then, if
$\xi_n \to \xi_0$ at an appropriate rate,
$\alpha_n(c;\xi_n,t_0)=\pr_{(\xi_n,t_0)}(\La_n\ge c)$ converges to a limiting
power function. It is also possible to consider limiting averaged local power
functions \citep{anpl1995, an1996, song2009}, where a limiting averaged power of the LRT
is evaluated over a set of local alternatives that correspond to the same
asymptotic noncentrality parameter.
\hfill\slut
\end{remark}

The following corollary shows how classical results for boundary points are
obtained as special cases by taking, for example, $\mathcal{T} = \{0\}$; the
choice is arbitrary as the value of $t$ does not affect the result.
\begin{corollary} \label{cor:no_nuisance} Under the conditions of
    Corollary~\ref{cor:lrt_nonsing}, suppose there are no nuisance parameters,
    i.e.\ $\mathcal{T} = \{0\}$, and additionally that $\Delta = \Delta(0)$ is a
    closed convex cone and $Z(0)\sim N(\mu,\Id_p)$ for some $\mu\in\R^p$. If
    $\mu=0$, then $\Lambda_n \Lto \Lambda = X(0) = \|P_\Delta Z(0)\|_2^2 \sim
    \bar{\chi}^2_p$ as $n\to\infty$, where the limiting
    $\bachi^2_p$-distribution in \re{chibar} has weights $w_0,w_1,\ldots,w_p$
    determined by $\De$. If additionally $\xi_0$ is in the interior of $\Xi$,
    then $\Delta = \R^p$ and $X(0)\sim \chi^2_p$. If $\mu\ne 0$, then $\La =
    \|P_\Delta [\mu + \va(0)]\|_2^2 $, with $\va(0)\sim N(0,\Id_p)$, will be a
    mixture of truncated, noncentral $\chi^2$-distributions, with weights
    depending on $\mu$.  If additionally $\xi_0$ is in the interior of $\Xi$,
    then $X(0)\sim\chi_{p,\|\mu\|_2^2}^2$ has a noncentral $\chi^2$-distribution
    with noncentrality parameter $\|\mu\|_2^2$.
\end{corollary}

As we will find in Section \ref{sec:ind} for i.i.d.\ data, the case $\mu=0$ of
Corollary \ref{cor:no_nuisance} corresponds to asymptotic limits under the null
hypothesis $H_0$, whereas $\mu\ne 0$ corresponds to limits for a sequence of
contiguous alternatives. 

Now, to discuss the components of the proof of Theorem \ref{thm:main_nonsing}, define, for every $t\in \mathcal{T}$,
\begin{align} \label{eq:X_tilde_check}
\tilde{X}_n(t) = \sup_{\delta \in \Delta_n(t)} Q_n(\delta, t);~~ \check{X}_n(t) = \sup_{\delta \in \Delta(t)}
Q_n(\delta, t).
\end{align}
The proof strategy is to show that $X_n$ is close to $\tilde{X}_n$, which in turn is close to $\check{X}_n$, which converges in law to $X$. We shall need the following rate of convergence result.

\begin{lemma} \label{lem:rate}
    Under Assumptions \ref{assm:equicont}--\ref{assm:zn_conv}, every $\hat{\xi}_n$ in Assumption \ref{assm:consistency} satisfies
    \[
        \Vert \hat{\xi}_n - \xi_0\Vert_{\mathcal{T}} = O_{\pr^*}(n^{-1/2}).
    \]
\end{lemma}

\begin{lemma} \label{lem:Xtilde_Xn}
    Under Assumptions \ref{assm:equicont}--\ref{assm:zn_conv}, $\tilde{X}_n$ is a random element of $F(\mathcal{T})$ for all $n$, and $\Vert \tilde{X}_n - X_n\Vert_{\mathcal{T}}  = o_{p^*}(1)$ as $n\to \infty$.
\end{lemma}

\begin{proof}
We first show $\tilde{X}_n$ is a random element of $F(\mathcal{T})$ by an argument we will use variations on several times: since $0 \in \Delta_n(t)$ and $Q_n(0, t) = 0$, the supremum of $\delta \mapsto Q_n(\delta, t)$ over $\Delta_n(t)$ is non-negative. But for any $\delta$ with $\Vert \delta\Vert_2 > 2\Vert Z_n(t)\Vert_2$, by Cauchy--Schwarz,
\begin{align} \label{eq:large_delta_negative}
   Q_n(\delta, t) \leq 2\Vert \delta\Vert_2 \Vert Z_n(t)\Vert_2 - \Vert \delta\Vert_2^2 = \Vert \delta\Vert_2(2\Vert Z_n(t)\Vert_2 - \Vert \delta\Vert_2) < 0.
\end{align}
Thus, 
\[
    |\tilde{X}_n(t)| = \sup_{\delta \in \Delta_n(t), \Vert \delta\Vert_2 \leq 2\Vert Z_n(t)\Vert_2} Q_n(\delta, t) \leq 8\Vert Z_n(t)\Vert_2^2,
\]
which, as a function of $t$, is a random element of $F(\mathcal{T})$ since $Z_n$ is a random element of $F^p(\mathcal{T})$ by Assumption \ref{assm:zn_conv}.

Next, pick $\hat{\xi}_n(t)$ such that $\ell_n(\hat{\xi}_n(t), t) \geq \sup_{\xi
\in \Xi(t)}\ell_n(\xi, t) - 1/n$, say. Let $\hat{\delta}_n(t) =
n^{1/2}A(t)(\hat{\xi}_n(t) - \xi_0)$ and note $\Vert
\hat{\delta}_n\Vert_{\mathcal{T}} = O_{p^*}(1)$ by Lemma \ref{lem:rate} and
Assumption \ref{assm:inf_bound}. Similarly, pick a $\tilde{\delta}_n(t)$ such
that $Q_n(\tilde{\delta}_n(t), t) \geq \sup_{\delta \in \Delta_n(t)} Q_n(\delta, t) - 1/n$. That supremum is non-negative by
the arguments preceding \eqref{eq:large_delta_negative}, so
$\tilde{\delta}_n(t)$ must achieve at least the value $-1/n$. We will use this
to show $\Vert \tilde{\delta}_n\Vert_{\mathcal{T}} = O_{p^*}(1)$. To that end,
suppose $\Vert Z_n\Vert_{\mathcal{T}} < C_1$. Then by the inequalities in
\eqref{eq:large_delta_negative}, any $\Vert \delta\Vert_2 > 3C_1$ leads to a value less than $-3C_1$, which is less than $1/n$ for all large enough $n$.
Thus, for such $n$, $\Vert \tilde{\delta}_n\Vert_{\mathcal{T}} \leq 3C_1$
whenever $\Vert Z_n\Vert_{\mathcal{T}} \leq C_1$, and since $\Vert
Z_n\Vert_{\mathcal{T}} = O_{p^*}(1)$, we have $\Vert
\tilde{\delta}_n\Vert_{\mathcal{T}} = O_{p^*}(1)$.

Next, for any $c_1, C_1 > 0$, by countable sub-additivity of outer probabilities, $ \pr^*(\Vert X_n - \tilde{X}_n\Vert_{\mathcal{T}} > c_1)$ is upper bounded by
\begin{equation}\label{eq:part_a}
\begin{aligned}
 \pr^*(\Vert X_n - \tilde{X}_n\Vert_{\mathcal{T}} > c_1, \Vert \hat{\delta}_n\Vert_{\mathcal{T}}\leq C_1, \Vert \tilde{\delta}_n\Vert_{\mathcal{T}} \leq C_1)
 +\pr^*(\Vert \hat{\delta}_n\Vert_{\mathcal{T}} > C_1) + \pr^*(\Vert \tilde{\delta}_n\Vert_{\mathcal{T}} > C_1).
\end{aligned}
\end{equation}
The last two terms can be made arbitrarily small by choosing $C_1$ large enough, so let us deal with the first term. Define
\begin{equation*}
    \Upsilon_{C_1,n} = \{(\delta, t): t \in\mathcal{T}, \delta\in\Delta_n(t), \Vert\delta\Vert_2\leq C_1\}. 
\end{equation*}
Since $\hat{\delta}_n(t)$ and $\tilde{\delta}_n(t)$ are $1/n$-approximate maximizers satisfying $\Vert\cdot\Vert_2\leq C_1$, two applications of the approximate-maximizer property and $|\sup f - \sup g|\leq \sup|f-g|$ give
\begin{equation*}
    \Vert X_n - \tilde{X}_n\Vert_{\mathcal{T}} \leq 4/n +
    \sup_{(\delta,t)\in\Upsilon_{C_1,n}}
    \left|2\{\tilde{\ell}_n(\delta,t)-\tilde{\ell}_n(0,t)\}-Q_n(\delta,t)\right|.
\end{equation*}

The quantity in absolute value in the last display is, by
\eqref{eq:ln_tilde_expand} and the definition of $\tilde{R}_n$, $2\Vert \delta\Vert_2^2 |R_n(\xi_0 + n^{-1/2}A(t)^{-\tsp}\delta, t)|$. Since, for
$(\delta, t) \in \Upsilon_{C_1,n}$, $\vert n^{-1/2}A(t)^{-\tsp}\delta\vert \leq
n^{-1/2}\Vert A(t)^{-\tsp}\Vert \Vert \delta\Vert_2 \leq n^{-1/2}C_1
\underline{\kappa}^{-1/2}$, $\xi_0 + n^{-1/2}A(t)^{-\tsp}\delta \to \xi_0$,
\eqref{eq:equicont} gives that the last supremum in the last display is
$o_{p^*}(1)$, which finishes the proof.
\end{proof}

\begin{lemma} \label{lem:tildeX_checkX} Under Assumptions
    \ref{assm:equicont}--\ref{assm:cone}, $\tilde{X}_n$ and $\check{X}_n$ are
    random elements of $F(\mathcal{T})$ for all $n$, and $\Vert \tilde{X}_n -
    \check{X}_n\Vert_{\mathcal{T}}  = o_{p^*}(1)$ as $n\to \infty$.
\end{lemma}

\begin{proof}
That $\tilde{X}_n \in F(\mathcal{T})$ is already given by Lemma
\ref{lem:Xtilde_Xn}. That $\check{X}_n \in F(\mathcal{T})$ follows from an
almost identical argument, and will be established in the proof of Lemma
\ref{lem:checkX_X}.

Since $\Xi(t)$ is closed by Assumption \ref{assm:closed}, so is $\Delta_n(t)$,
and therefore the projection onto it exists. Similarly, by Assumption
\ref{assm:cone}, the projection onto $\Delta(t)$ exists. Thus, as in
\eqref{eq:X-process}, we get $\tilde{X}_n(t) = \Vert
Z_n(t)\Vert_2^2 - \Vert Z_n(t) - P_{\Delta_n(t)}Z_n(t)\Vert_2^2$ and $\check{X}_n(t)= \Vert Z_n(t)\Vert_2^2 - \Vert
Z_n(t) - P_{\Delta(t)}Z_n(t)\Vert_2^2$. Thus, for outcomes where $\Vert
Z_n\Vert_{\mathcal{T}} \leq C_1$ we get that  $\sup_{t \in \mathcal{T}}|\tilde{X}_n(t) - \check{X}_n(t)|$
equals
\begin{align*} 
    \sup_{t\in \mathcal{T}}|\Vert Z_n(t) - P_{\Delta(t)}Z_n(t)\Vert_2^2 - 
        \Vert Z_n(t) - P_{\Delta_n(t)}Z_n(t)\Vert_2^2| \leq \sup_{t\in\mathcal{T}, \Vert z\Vert_2 \leq C_1} \left|\Vert z - P_{\Delta(t)}z\Vert_2^2 - \Vert z - P_{\Delta_n(t)}z\Vert_2^2\right|,
\end{align*}
which tends to zero by Lemma \ref{lem:set_dist}. Since $\Vert
Z_n\Vert_{\mathcal{T}} \Lto \Vert Z\Vert_{\mathcal{T}}$ by continuous mapping
theorem and the limit is tight by Assumption \ref{assm:zn_conv}, $\Vert Z_n\Vert_{\mathcal{T}}$ is asymptotically
tight. Thus, the outer probability of
$\Vert Z_n\Vert_{\mathcal{T}} > C_1$ can be made arbitrarily small
asymptotically by choosing $C_1$ large enough, which completes the proof.
\end{proof}

\begin{lemma} \label{lem:checkX_X} Under Assumptions
    \ref{assm:inf_bound}--\ref{assm:cone}, $\check{X}_n$ is a random element of
    $F(\mathcal{T})$ for all $n$, $X$ is a Borel-measurable random element of
    $F(\mathcal{T})$, and $\check{X}_n \Lto X$ in $F(\mathcal{T})$ as
    $n\to\infty$.
\end{lemma}

\begin{proof}
   Consider the mapping $g:F^p(\mathcal{T}) \to F(\mathcal{T})$ defined for $z
   \in F^p(\mathcal{T})$ by $g(z)(t) = \sup_{\delta \in \Delta(t)}(2\delta^\tsp
   z(t) - \Vert \delta\Vert_2^2)$. Indeed, $g(z) \in F(\mathcal{T})$
   for any $z \in F^p(\mathcal{T})$ since, by the triangle inequality and $0\in
   \Delta(t)$, arguing as in \eqref{eq:X-process},
\[
   | g(z)(t)| \leq \sup_{t\in\mathcal{T}} \Vert z(t)\Vert_2^2 + 
    \sup_{t\in\mathcal{T}}\inf_{\delta \in \Delta(t)} \Vert z(t) - 
    \delta\Vert_2^2 \leq 2\Vert z\Vert_{\mathcal{T}}^2 < \infty.
\]
This also shows that $X = g(Z)$ and $\check{X}_n = g(Z_n)$ are random elements of
$F(\mathcal{T})$ since $Z$ and $Z_n$ are random elements of $F^p(\mathcal{T})$
by Assumption \ref{assm:zn_conv}.

To show $g$ is continuous, pick an arbitrary $z_{\infty} \in F^p(\mathcal{T})$
and $z_m \in F^p(\mathcal{T})$ such that $\Vert z_{\infty} -
z_m\Vert_{\mathcal{T}} \to 0$ as $m\to \infty$. Pick a finite $C_1 > \Vert
z_{\infty}\Vert_{\mathcal{T}}$ so that $\Vert z_m\Vert_{\mathcal{T}} < C_1$ for
all large enough $m$. Then, as argued in the proof of Lemma \ref{lem:Xtilde_Xn},
any $\delta$ with $\Vert \delta\Vert_2 > 3C_1$ leads to $\delta^\tsp z_m(t) -
\Vert \delta\Vert_2^2 < 0$, which since $0 \in \Delta(t)$ implies
\[
    g(z_m)(t) = \sup_{\delta \in \Delta(t) \cap \bar{B}_{3C_1}} (2\delta^\tsp z_m(t) 
    - \Vert \delta\Vert_2^2),
\]
also when $m = \infty$. Thus, using that the difference of suprema is no greater than the supremum of the difference, we get
\begin{align*}
   |g(z_m)(t) - g(z_{\infty})(t)| 
    \leq \sup_{\delta \in \Delta(t) \cap \bar{B}_{3C_1}} |2\delta^\tsp (z_m(t) - z_{\infty}(t))|
    \leq 6C_1\, \Vert z_m(t) - z_{\infty}(t)\Vert_2.
\end{align*}
Taking the supremum over $t\in \mathcal{T}$ on both sides and sending
$m\to\infty$ shows $\Vert g(z_m) - g(z_{\infty})\Vert_{\mathcal{T}} \to 0$, so
$g$ is continuous. Thus, $X$ is Borel-measurable since $Z$ is, and since $Z_n
\Lto Z$ by Assumption \ref{assm:zn_conv}, the continuous mapping theorem gives
$\check{X}_n = g(Z_n) \Lto g(Z) = X$, which completes the proof.
\end{proof}

With these lemmas in hand, the proof of Theorem \ref{thm:main_nonsing} is a
straightforward application of Slutsky's lemma and the continuous mapping
theorem given in the Appendix.

\subsection{Likelihood ratio test with composite null} \label{sec:composite}

In this section we will generalize to a composite null hypothesis, and the corresponding alternative hypothesis, of the form 
\beq
\begin{array}{ll}
H_0: & \xi \in \Xi_0,\\
H_1: & \xi \in \Xi \setminus \Xi_0,
\end{array}
\lb{H0H1}
\eeq
where $\Xi_0\subseteq \R^{p}$ is such that, for each $\xi \in \Xi_0$, there exists a $t \in
\mathcal{T}$ with $(\xi, t) \in \Theta$. For each $t \in \mathcal{T}$, let
$\Xi^0(t)$ be the set of $\xi \in \Xi_0$ such that $(\xi, t) \in \Theta$. Thus,
$\Xi_0 = \cup_{t \in \mathcal{T}} \Xi^0(t)$, and when $\ell_n$ is a log-likelihood the likelihood ratio test
statistic for testing $H_0$ against $H_1$ in \re{H0H1} is
\begin{align} \label{eq:Lambda_composite}
    \Lambda_n &= 2 \left\{\sup_{t \in \mathcal{T}}\sup_{\xi \in \Xi(t)}\ell_n(\xi, t) - \sup_{t \in \mathcal{T}}\sup_{\xi \in \Xi^0(t)}\ell_n(\xi, t)\right\}.
\end{align}
To obtain the asymptotic distribution for this case, define $X_n^0$ like $X_n$ but with $\Xi^0(t)$ in place of $\Xi(t)$; that is,
\begin{align}
    X_n^0(t) = 2\sup_{\xi \in \Xi^0(t)} \{\ell_n(\xi, t) - \ell_n(\xi_0, t)\}.
\end{align}
The joint convergence of $(X_n, X_n^0)$, given by the following result, will
lead to the asymptotic distribution of $\Lambda_n$ under the null. Assumption \ref{assm:cone} can only apply to $\Xi^0(t)$ if $\xi_0 \in \Xi^0(t)$
for all $t \in \mathcal{T}$. In a hypothesis testing context, this means the
null hypothesis must be correct regardless of the value of the nuisance
parameter.

\begin{lemma} \label{thm:joint_conv}
    Under Assumptions \ref{assm:Xn_bdd}--\ref{assm:cone}, with Assumptions \ref{assm:closed} and \ref{assm:cone} also holding for $\Xi^0(t)$ and a cone $C^0(t)$ in place of $\Xi(t)$ and $C(t)$, respectively, $(X_n, X_n^0) \Lto (X, X^0)$ in $F^2(\mathcal{T})$. Here $X^0=\{X^0(t);\, t\in\cT\}$, where $X^0(t)$ is defined as in \re{eq:X-process}, with $\De^0(t)=A(t)^\tsp C^0(t)$ in place of $\De(t)=A(t)^\tsp C(t)$. 
\end{lemma}
\begin{proof}
The proof follows the same approximation and continuous mapping steps as for
Theorem \ref{thm:main_nonsing}. Under the assumptions of Theorem
\ref{thm:main_nonsing}, and the corresponding assumptions with $\Xi^0(t)$ in place
of $\Xi(t)$, the lemmas leading up to Theorem \ref{thm:main_nonsing} yield
$\Vert X_n-\check{X}_n\Vert_{\mathcal{T}}=o_{\pr^*}(1)$ and
$\Vert X_n^0-\check{X}_n^0\Vert_{\mathcal{T}}=o_{\pr^*}(1)$, where $\check{X}_n^0(t)=\sup_{\delta\in\Delta^0(t)}(2\delta^\tsp Z_n(t)-\Vert\delta\Vert_2^2)$.
For $z\in F^p(\mathcal{T})$, define maps $g,g_0:F^p(\mathcal{T})\to F(\mathcal{T})$ by
\[
    g(z)(t)=\sup_{\delta\in\Delta(t)}(2\delta^\tsp z(t)-\Vert\delta\Vert_2^2),
    \qquad
    g_0(z)(t)=\sup_{\delta\in\Delta^0(t)}(2\delta^\tsp z(t)-\Vert\delta\Vert_2^2).
\]
Then $\check{X}_n=g(Z_n)$ and $\check{X}_n^0=g_0(Z_n)$.
By the argument in Lemma \ref{lem:checkX_X}, applied to $\Delta$ and, separately,
to $\Delta^0$, both $g$ and $g_0$ are continuous. Hence, the map
$z\mapsto (g(z),g_0(z))$ from $F^p(\mathcal{T})$ to $F^2(\mathcal{T})$ is
continuous, and since $Z_n\Lto Z$ in $F^p(\mathcal{T})$ by Assumption
\ref{assm:zn_conv}, the continuous mapping theorem gives
$(\check{X}_n,\check{X}_n^0)=(g(Z_n),g_0(Z_n))\Lto (g(Z),g_0(Z))=(X,X^0)$ in
$F^2(\mathcal{T})$. The $o_{\pr^*}(1)$ approximations then imply
$(X_n,X_n^0)\Lto (X,X^0)$ by Slutsky's lemma.
\end{proof}

\begin{theorem}\label{thm:composite}
    Under the conditions of Lemma~\ref{thm:joint_conv}, and assuming $t \mapsto
    \Xi^0(t)$ and $t \mapsto \ell_n(\xi, t)$, for all $\xi \in \Xi_0$, are
    constant, the random variable defined in \re{eq:Lambda_composite} converges weakly $\Lambda_n \Lto \Lambda = \Vert X - X^0\Vert_{\mathcal{T}}$ as $n\to\infty$.
\end{theorem}
\begin{proof}
    Since $\Xi^0(t)$ does not depend on $t$, it is equal to $\Xi_0$. Since
    also $\ell_n(\xi, t)$ is constant in $t$ for $\xi \in \Xi_0$, $X_n^0(t) = 2
    \{\sup_{\xi \in \Xi_0}\ell_n(\xi, t) - \ell_n(\xi_0, t)\}$ in fact does not
    depend on $t$. Thus, since $\xi_0 \in \Xi_0$ by Assumption \ref{assm:cone}, we can move the supremum over $t$ outside the brackets in \eqref{eq:Lambda_composite}, add and subtract $\ell_n(\xi_0)$, and get
    \[
        \Lambda_n = \sup_{t \in \mathcal{T}}\{X_n(t) - X_n^0(t)\} = \Vert X_n - X_n^0\Vert_{\mathcal{T}} = \Vert X_n\Vert_{\mathcal{T}} - \Vert X_n^0\Vert_{\mathcal{T}}.
    \]
    By Lemma~\ref{thm:joint_conv} and the continuous mapping theorem, $\Vert X_n\Vert_{\mathcal{T}} - \Vert X_n^0\Vert_{\mathcal{T}}$ converges in law to $\Vert X \Vert_{\mathcal{T}} - \Vert X^0 \Vert_{\mathcal{T}} = \Vert X - X^0\Vert_{\mathcal{T}}$.
\end{proof}

An interesting special case of the setting in Theorem \ref{thm:composite} is
when $C^0(t)= \lim_{s\downarrow 0} s^{-1}(\Xi_0(t) - \xi_0)$ is a linear subspace of $\R^p$.
This happens, for example, when $\xi=(\xi^{(1)}, \xi^{(2)})^\tsp$, where the first $p_1$ components $\xi^{(1)}$ of $\xi$ are restricted under the null hypothesis, whereas the other $p_2=p-p_1$ components $\xi^{(2)}$ are free. This can be formalized as $\xi \in \Xi(t) = \Xi^{(1)}(t) \times \Xi^{(2)}(t)$, where $\Xi^{(1)}(t)\subset \R^{p_1}$, $\Xi^{(2)}(t)\subset \R^{p_2}$, and a null hypothesis $H_0: \xi^{(1)}=\xi_0^{(1)}$, with $\xi_0=(\xi_0^{(1)},\xi_0^{(2)})^\tsp$, and $\Xi_0(t) = \{\xi_0^{(1)}\}\times \Xi^{(2)}(t)$. If additionally $\xi_0^{(2)}$ is an inner point of $\Xi^{(2)}(t)$, then $C(t)=C^0(t)\oplus C^1(t)$, where $C^0(t)=C^0 = \{0\}\times\R^{p_2}$ and $C^1(t)= C_{\scr{simple}}(t) \times \{0\}$, where $C_{\scr{simple}}(t)= \lim_{s\downarrow 0} s^{-1}(\Xi^{(1)}(t) - \xi_0^{(1)})$ is the convex cone of a model where $\xi^{(2)}$ is known. 
In such settings, the asymptotic distribution of the LRT statistic is similar (but not necessarily identical)
to when a simple null hypothesis of dimension $p_1$ is tested, as formalized by the following
result (compare to Corollary \ref{cor:lrt_nonsing}):

\begin{theorem}\label{thm:decomposition} 
    Assume the conditions of Theorem \ref{thm:composite} hold, and that for each $t\in\mathcal{T}$,
    $\Delta(t)$ is a closed convex cone and $\Delta^0(t) = A(t)^\tsp C^0(t)$ is
    a linear subspace. Then $K(t) = \Delta(t)\cap\{\Delta^0(t)\}^\perp$ is a
    closed convex cone, $\Delta(t) = \Delta^0(t)\oplus K(t)$, and, consequently, $\Lambda = \sup_{t\in\mathcal{T}}\|P_{K(t)}Z(t)\|_2^2$.
\end{theorem}
Since $K(t)$ is a closed convex cone, the same argument as in
\eqref{eq:X-process} shows that the distribution of $X(t)-X^0(t)=\|P_{K(t)}Z(t)\|_2^2$ is a
mixture of $\chi^2$-distributions when $Z(t)$ is a centered Gaussian process, with weights
determined by the geometry of $K(t)$. Then $\Lambda$ is the supremum of a
$\bar{\chi}^2$-process. If $K(t)$ is polyhedral, defined by $r(t)$ linear
inequality constraints, only the $\chi^2_{p_1-r(t)}, \ldots, \chi^2_{p_1}$ components receive
positive weights. 

The following corollary of Theorem \ref{thm:decomposition} gives conditions under which the asymptotic distribution of the LRT statistic, for a composite null hypothesis, is the same as for the corresponding simple null hypothesis for $\xi^{(1)}$, when $\xi^{(2)}$ is known and $\theta=(\xi^{(1)},t)$:

\begin{corollary}\lb{Cor:Decomp}
Suppose the conditions of Theorem \ref{thm:decomposition} hold, with a composite null hypothesis $H_0: \xi^{(1)}=\xi_0^{(1)}$ that is based on a decomposition $\xi=(\xi^{(1)},\xi^{(2)})^\tsp$ of the first part of $\theta=(\xi,t)$, as described above Theorem \ref{thm:decomposition}. Suppose additionally that the information matrix 
\beq
I(t) = \left( \begin{array}{cc}
I^{(1)}(t) & 0 \\
0 & I^{(2)}(t) \end{array}
\right)
\lb{IBDiag}
\eeq
is block diagonal for all $t\in\cT$, with the first and second blocks corresponding to $\xi^{(1)}$ and $\xi^{(2)}$ respectively. The distribution of $\Lambda$ is then the same as for a model with $\theta=(\xi^{(1)},t)$ where $\xi^{(2)}$ is known, and a simple null hypothesis $H_0: \xi^{(1)}=\xi_0^{(1)}$ is tested.  
\end{corollary}

\begin{proof}
Since $I(t)$ is block diagonal we can choose $A(t)=\mbox{BDiag}(A^{(1)}(t),A^{(2)}(t))$ to be block diagonal as well, with $A^{(i)}(t)A^{(i)}(t)^\tsp=I^{(i)}(t)$. From this it follows that the normalized score vector $Z_n(t)=(Z_n^{(1)}(t),Z_n^{(2)}(t))^\tsp$ decomposes into two components $Z_n^{(i)}(t)=A^{(i)}(t)^{-1}U_n^{(i)}(t)$ of dimension $p_i$ for $i=1,2$. Then $Z_n^{(1)}(t)$ is the standardized score of the model with $\xi^{(2)}$ known, so, by Assumption \ref{assm:zn_conv} and the continuous mapping theorem, its weak limit $Z^{(1)}$ is the limit process of that model. Since additionally
\beq
\begin{array}{rcl}
\Delta(t) &=& A(t)^\tsp C(t)\\
&=& A(t)^\tsp (C^0(t)\oplus C^1(t))\\
&=& [\{0\}\times A^{(2)}(t)^\tsp \R^{p_2}] \oplus  [A^{(1)}(t)^\tsp C_{\scr{simple}}(t) \times \{0\}]\\
&=&  [\{0\}\times \R^{p_2}] \oplus [\Delta_{\scr{simple}}(t) \times \{0\}]\\
&=: & \Delta_0(t) \oplus K(t),
 \end{array}
\lb{DeltaCompH0}
\eeq
with $\Delta_{\scr{simple}}(t)=A^{(1)}(t)^\tsp C_{\scr{simple}}(t) $ the transformed convex cone of the simple null hypothesis where $\xi^{(2)}$ is known. It follows from Theorem \ref{thm:decomposition} and  \re{DeltaCompH0} that 
\beq
\begin{array}{rcl}
X(t)-X^0(t) &=& \|P_{K(t)}Z(t)\|^2\\
&=& \|P_{\Delta_{\scr{simple}}(t)}Z^{(1)}(t) \times \{0\}\|^2\\ 
&=&   \|P_{\Delta_{\scr{simple}}(t)}Z^{(1)}(t) \|^2 = X_{\scr{simple}}(t),
\end{array}
\lb{XDiffXSimp}
\eeq
where $Z^{(1)}(t)$ and $X_{\scr{simple}}(t)$ are the weak limits of the normalized score vector and maximized LRT statistic respectively, under the simple null hypothesis, for each fixed $t\in\cT$. When maximizing over $t$, we deduce that $\Lambda = \sup_{t\in\cT} [X(t)-X^0(t)]$ equals $\Lambda_{\scr{simple}} = \sup_{t\in\cT} X_{\scr{simple}}(t)$, from which the corollary follows. 
\end{proof}

\begin{remark}
Suppose $\ell_n(\xi,t)$ is a log-likelihood and the conditions of
Theorem \ref{thm:composite} hold. If $\xi=\xi_0$ for some $\xi_0\in\Xi_0$,
Theorem \ref{thm:composite} provides weak convergence of the likelihood ratio
test (LRT) statistic $\La_n$ under an element $\xi_0$ of the composite null
hypothesis $H_0$ in \re{H0H1}. If $H_0$ is rejected when $\La_n\ge c$, the significance level satisfies
\beq
\alpha_n(c;\xi_0)  = \pr_{(\xi_0,t)}(\Lambda_n\ge c) \to \alpha(c;\xi_0) = \pr_{(\xi_0,t)}(\Lambda \ge c)
\lb{alphaConv2}
\eeq
as $n\to\infty$ at all continuity points $c$ of the distribution function
$F_\La$ of $\La$, independently of $t\in\cT$. Note that the distribution of the limiting LRT
statistic $\La$ in Theorem \ref{thm:composite} typically depends on
the choice of $\xi_0\in\Xi_0$, since $X$ and $X_0$ both depend on $\xi_0$. Therefore, the limiting
significance level $\alpha(c)=\alpha(c;\xi_0)$ typically depends on $\xi_0\in\Xi_0$ as well. However, under the conditions of Corollary \ref{Cor:Decomp}, $\La$ has the
distribution it would have with $\xi^{(2)}$ known, so $\alpha(c;\xi_0)$ depends
on $\xi_0$ only through the covariance function of $Z^{(1)}$ and the cone
$\Delta_{\scr{simple}}(t)$ in the proof of that corollary. 

Suppose instead we have a sequence of
contiguous alternatives $\theta_n=(\xi_n,t_0)$ such that $\xi_n \to \xi_0$ as
$n\to\infty$ for some $\xi_0\in\Xi_0$. Then
$\alpha_n(c;\xi_n,t_0)=\pr_{(\xi_n,t_0)}(\Lambda_n\ge c)$ corresponds to the
power of the LRT and if $\xi_n\to\xi_0$ at an appropriate rate it will converge
to a limiting power function. As mentioned in Section \ref{sec:pd_inf}, it is
possible then to consider limiting averaged local power
functions \citep{anpl1995, an1996, song2009}.
\hfill\slut
\end{remark}

\subsection{Singular information} \label{sec:singular}

In some cases the information matrix $I(t)$ is singular for some $t \in
\mathcal{T}$. We partition $\mathcal{T}$ into $\mathcal{T}_s = \{t\in
\mathcal{T}: \emin\{I(t)\} = 0\}$ and $\mathcal{T}_{ns} = \{t\in \mathcal{T}:
\emin\{I(t)\} > 0\}$. We will first give conditions for convergence in law of
$\Lambda_n^{(ns)} = \sup_{t \in \mathcal{T}_{ns}} X_n(t)$, and then extend that
to $\Lambda_n = \sup_{t \in \mathcal{T}} X_n(t)$.

Establishing the asymptotic distribution of $\Lambda_n^{(ns)}$ is more
complicated than applying the conditions from Section \ref{sec:asy_dist}
with $\mathcal{T}$ replaced by $\mathcal{T}_{ns}$. In particular, assuming
something like Assumption \ref{assm:inf_bound} on $\mathcal{T}_{ns}$ is
inappropriate as, in many settings of interest, one can find a sequence $(t_m)$
in $\mathcal{T}_{ns}$ tending to a $t_* \in \mathcal{T}_s$ and get
$\emin\{I(t_m)\} \to 0$ as $m\to \infty$. Our strategy will be to first work on
sets $\mathcal{T}_\epsilon = \{t\in \mathcal{T}: d(t, \mathcal{T}_s) \geq
\epsilon\}$ for $\epsilon > 0$. This requires a metric $d$ on $\mathcal{T}$ and,
for sets $A \subseteq \mathcal{T}$, defining $d(t, A) = \inf_{t'\in A} d(t, t')$.

\begin{assumption} \label{assm:T_metric} The set $\mathcal{T}$ is a metric space
    with metric $d(\cdot, \cdot)$,  $\mathcal{T}_s$ is closed in $\mathcal{T}$,
    and $\mathcal{T}_{ns}$ is dense in $\mathcal{T}$.
\end{assumption}

Closedness of $\mathcal{T}_s$ ensures $d(t, \mathcal{T}_s) > 0$ for every $t
\in \mathcal{T}_{ns}$, so that $\cup_{\epsilon > 0} \mathcal{T}_\epsilon =
\mathcal{T}_{ns}$. Density of $\mathcal{T}_{ns}$ is used when extending the
result from $\mathcal{T}_{ns}$ to all of $\mathcal{T}$ in Corollary
\ref{cor:main_singular_fullT}.

Assumption \ref{assm:equicont} does not control the remainder term
$\tilde{R}_n(\delta, t)$ uniformly over $\mathcal{T}_{ns}$ unless the
eigenvalues of $I(t)$ are bounded away from zero. Since this cannot be assumed
on $\mathcal{T}_{ns}$, we instead make the following assumption (which, under Assumption \ref{assm:inf_bound} is equivalent to Assumption \ref{assm:equicont}):

\begin{manualtheorem}{2$^\star$}{} \label{assm:Rtilde_bound} For every $t
\in \mathcal{T}_{ns}$, \eqref{eq:X_n-process-delta} holds with a remainder term
$\tilde{R}_n(\delta, t)$ that satisfies the following: for every $c_1, c_2 > 0$,
there exists a $c_3 > 0$ such that
    \[
        \limsup_{n\to \infty} \pr^*\left(\sup_{t \in \mathcal{T}_{ns}} \sup_{\delta \in \tilde{B}(n, t, c_3)} \left|\tilde{R}_n(\delta, t)\right| > c_1\right) < c_2,
    \]
    where $\tilde{B}(n, t, c_3) = \{\delta \in \Delta_n(t): \Vert \delta\Vert_2 \leq \sqrt{n}c_3\}$.
\end{manualtheorem}

Assumption \ref{assm:sing_inf} modifies Assumption \ref{assm:inf_bound} by
allowing $\underline{\kappa}$ to depend on $\epsilon$. The lower bound on
$\emin\{I(t)\}$, which is usually the challenging part, holds if $I(\cdot)$ is
continuous on a compact $\mathcal{T}_{\epsilon}$, for example.

\begin{manualtheorem}{4$^{\star}$}{} \label{assm:sing_inf}
   There exists a $\bar{\kappa} < \infty$ and, for every $\epsilon > 0$, a $\underline{\kappa}_\epsilon > 0$ such that, for every $t \in \mathcal{T}_\epsilon$,
    \[
       \underline{\kappa}_\epsilon \leq \emin\{I(t)\} \leq \emax\{I(t)\} \leq \bar{\kappa}.
    \]
\end{manualtheorem}

Assumption \ref{assm:sing_Zn} is Assumption \ref{assm:zn_conv} restricted to
$\mathcal{T}_{ns}$. In particular, $Z_n(t) = A(t)^{-1}U_n(t)$ is only defined
for $t \in \mathcal{T}_{ns}$. The spaces $F^p(\mathcal{T}_{ns})$, $p\geq 1$,
are defined as in Section \ref{sec:pd_inf} but with $\mathcal{T}_{ns}$ in place
of $\mathcal{T}$.

\begin{manualtheorem}{5$^\star$}{} \label{assm:sing_Zn}
    There is a tight and Borel-measurable $Z \in F^p(\mathcal{T}_{ns})$ such
    that $Z_n \Lto Z$ in $F^p(\mathcal{T}_{ns})$ as $n\to \infty$, with $Z_n$ defined
    as in \eqref{eq:ln_tilde_expand}.
\end{manualtheorem}

Assumption \ref{assm:monotonicity} ensures $\Delta_n(t) \subseteq \Delta(t)$
for all $n$, which gives $\tilde{X}_n(t) \leq \check{X}_n(t)$ and is key to
the bounding argument in the proof. It is satisfied whenever $\Xi(t)$ is convex
for every $t \in \mathcal{T}$.

\begin{assumption} \label{assm:monotonicity}
    For every $t\in \mathcal{T}$ and $0 < s_1 \leq s_2$,
    \[
        s_2^{-1}\{\Xi(t) - \xi_0\} \subseteq s_1^{-1}\{\Xi(t) - \xi_0\}.
    \]
\end{assumption}

\begin{theorem}\label{thm:main_singular}
    Under Assumptions \ref{assm:Rtilde_bound}, \ref{assm:consistency}, \ref{assm:sing_inf}, \ref{assm:sing_Zn}, \ref{assm:closed}, \ref{assm:cone}, \ref{assm:T_metric}, and \ref{assm:monotonicity},
    \[
        \Lambda_n^{(ns)} = \sup_{t \in \mathcal{T}_{ns}} X_n(t) \Lto   \sup_{t \in \mathcal{T}_{ns}} X(t) = \Lambda^{(ns)},
    \]
    where $X(t) = \sup_{\delta \in \Delta(t)} (2\delta^\tsp Z(t) - \|\delta\|_2^2)$ is defined as in \re{eq:X-process} and $Z$ is the limit process from Assumption~\ref{assm:sing_Zn}.
\end{theorem}

\begin{proof}
    Let $\tilde{\Lambda}_n^{(ns)} = \sup_{t \in \mathcal{T}_{ns}} \tilde{X}_n(t)$ and $\check{\Lambda}_n^{(ns)} = \sup_{t \in \mathcal{T}_{ns}} \check{X}_n(t)$. For any $\epsilon > 0$, let $\tilde{\Lambda}_n^{\epsilon} = \sup_{t \in \mathcal{T}_\epsilon} \tilde{X}_n(t)$ and $\check{\Lambda}_n^{\epsilon} = \sup_{t \in \mathcal{T}_\epsilon} \check{X}_n(t)$. Observe
    \[
        \Lambda_n^{(ns)} = (\Lambda_n^{(ns)} - \tilde{\Lambda}_n^{(ns)}) + (\tilde{\Lambda}_n^{(ns)} - \check{\Lambda}_n^{(ns)}) + \check{\Lambda}_n^{(ns)},
    \]
    with the subtracted quantities finite since $\tilde{\Lambda}_n^{(ns)}, \check{\Lambda}_n^{(ns)} \leq \|Z_n\|_{\mathcal{T}_{ns}}^2 = O_{\pr^*}(1)$ by Assumption \ref{assm:sing_Zn}. Lemma \ref{lem:Xtilde_Xn_singular} implies $\Lambda_n^{(ns)} - \tilde{\Lambda}_n^{(ns)} = o_{\pr^*}(1)$. Lemma \ref{lem:checkX_X_singular} implies $\check{\Lambda}_n^{(ns)} \Lto \Lambda^{(ns)}$, so it suffices to show $\check{\Lambda}_n^{(ns)} - \tilde{\Lambda}_n^{(ns)} = o_{\pr^*}(1)$.

    To that end, let $C_s(t) = s^{-1}\{\Xi(t) - \xi_0\}$, so that $\Delta_n(t) = A(t)^\tsp C_{n^{-1/2}}(t)$ and $\Delta(t) = A(t)^\tsp C(t)$. By Assumption \ref{assm:monotonicity}, if $\delta \in C_{n^{-1/2}}(t)$ then $\delta \in C_{m^{-1/2}}(t)$ for all $m\geq n$. Since $d\{C_{m^{-1/2}}(t), C(t)\} \to 0$ as $m\to\infty$ by Assumption \ref{assm:cone} and $C(t)$ is closed, this implies $C_{n^{-1/2}}(t) \subseteq C(t)$ and hence $\Delta_n(t) \subseteq \Delta(t)$.
    Therefore $\tilde{X}_n(t) \leq \check{X}_n(t)$ for all $t \in \mathcal{T}$. Consequently, for every $\epsilon > 0$,
    \[
        0 \leq \tilde{\Lambda}_n^{\epsilon} \leq \tilde{\Lambda}_n^{(ns)} \leq \check{\Lambda}_n^{(ns)} \quad \text{and} \quad 0 \leq \tilde{\Lambda}_n^{\epsilon} \leq \check{\Lambda}_n^{\epsilon} \leq \check{\Lambda}_n^{(ns)}.
    \]
    This gives the bound $0 \leq \check{\Lambda}_n^{(ns)} - \tilde{\Lambda}_n^{(ns)} \leq \check{\Lambda}_n^{(ns)} - \tilde{\Lambda}_n^{\epsilon} = (\check{\Lambda}_n^{(ns)} - \check{\Lambda}_n^{\epsilon}) + (\check{\Lambda}_n^{\epsilon} - \tilde{\Lambda}_n^{\epsilon})$. Thus, for any $\eta>0$,
    \begin{align*}
        \pr^*\left(|\check{\Lambda}_n^{(ns)} - \tilde{\Lambda}_n^{(ns)}| > \eta\right) &\leq \pr^*\left(\check{\Lambda}_n^{(ns)} - \check{\Lambda}_n^{\epsilon} \geq \eta/2\right) + \pr^*\left(\check{\Lambda}_n^{\epsilon} - \tilde{\Lambda}_n^{\epsilon} \geq \eta/2\right).
    \end{align*}
    For any fixed $\epsilon > 0$, the second term on the right-hand side tends to zero as $n\to\infty$ by Lemma \ref{lem:tildeX_checkX_singular}. Thus, we are done if we can show that, for any fixed $\alpha > 0$, there is an $\epsilon > 0$ such that the upper limit of the first term on the right-hand side is less than $\alpha$.

    To that end, note Lemma \ref{lem:checkX_X_singular} gives $\check{X}_n \Lto X$ in $F(\mathcal{T}_{ns})$. Since $\mathcal{T}_\epsilon \subseteq \mathcal{T}_{ns}$, the map $h: F(\mathcal{T}_{ns}) \to \R$ defined by
    \[
        h(f) = \sup_{t\in\mathcal{T}_{ns}} f(t) - \sup_{t\in\mathcal{T}_\epsilon} f(t)
    \]
    is continuous (in fact, Lipschitz). Therefore, by the continuous mapping theorem,
    \[
        \check{\Lambda}_n^{(ns)} - \check{\Lambda}_n^{\epsilon} = h(\check{X}_n) \Lto h(X) = \Lambda^{(ns)} - \Lambda^{\epsilon}.
    \]
    Moreover, by Assumption \ref{assm:T_metric}, for every $t\in\mathcal{T}_{ns}$ we have $d(t,\mathcal{T}_s) > 0$, so $t\in\mathcal{T}_\epsilon$ for all small enough $\epsilon$. Hence $\cup_{\epsilon>0}\mathcal{T}_\epsilon=\mathcal{T}_{ns}$ and $\Lambda^{\epsilon} \uparrow \Lambda^{(ns)}$ almost surely as $\epsilon \downarrow 0$. Thus, $\Lambda^{(ns)} - \Lambda^{\epsilon} \downarrow 0$ almost surely as $\epsilon\downarrow 0$, which implies $\pr(\Lambda^{(ns)} - \Lambda^{\epsilon} \geq \eta/2) \to 0$ as $\epsilon \downarrow 0$. We can therefore pick an $\epsilon$ such that $\pr(\Lambda^{(ns)} - \Lambda^{\epsilon} \geq \eta/2)\leq \alpha$. Then, by the Portmanteau theorem \citep[Theorem 1.3.4]{vandervaart2023weak},
    \[
        \limsup_{n\to\infty} \pr^*\bigl(\check{\Lambda}_n^{(ns)} - \check{\Lambda}_n^{\epsilon} \geq \eta/2\bigr) \leq \pr\bigl(\Lambda^{(ns)} - \Lambda^{\epsilon} \geq \eta/2\bigr) \leq \alpha,
    \]
which finishes the proof.
\end{proof}

\begin{remark}
For a closed convex cone $\Delta(t)$,
\[
    \sup_{\delta \in \Delta(t)}\{2\delta^\tsp Z(t) - \Vert \delta\Vert_2^2\}
    = \sup\{\max(u^\tsp Z(t), 0)^2 : u \in \Delta(t),\, \Vert u\Vert_2 = 1\}.
\]
The right side has the form of the limit of \citet{liu2003asymptotics}, with
the unit vectors $u$ in $\Delta(t)$ in the role of their limiting generalized
scores, so the limit in Theorem \ref{thm:main_singular} agrees with theirs
wherever both apply. Examples \ref{ex:gauss_mix_1} and \ref{ex:gauss_mix_2comp}
are instances of their Example 3.1, and their Theorem 3.3 corresponds to our
Theorem \ref{thm:decomposition}. In place of independent and identically
distributed observations, Theorem \ref{thm:main_singular} assumes that the
quadratic approximation holds uniformly on $\mathcal{T}_{ns}$, so it covers the
triangular array of Example \ref{Ex:CP}.
\end{remark}

Theorem \ref{thm:main_singular} can be used to obtain the asymptotic
distribution of $\Lambda_n$ in different ways. For example, if $\mathcal{T}_s$
is finite and $X_n$ is c\`adl\`ag, that is, right-continuous with left limits,
then $\Lambda_n = \sup_{t\in\mathcal{T}} X_n(t) = \sup_{t\in\mathcal{T}_{ns}}
X_n(t)$ for every $n$. This follows because $\mathcal{T}_{ns}$ is dense and
right-continuity forces the values at points in $\mathcal{T}_s$ to agree with
right limits in $\mathcal{T}_{ns}$. Hence, the limit distribution of $\Lambda_n$
is the same as that of $\Lambda_n^{(ns)}$.

The following corollary gives a general result in this direction. We say
$t\mapsto\Xi(t)$ is inner semicontinuous if for every $t\in\mathcal{T}$ and
sequence $(t_m)$ tending to $t$, there exist $\xi_m \in \Xi(t_m)$ such that
$\xi_m \to \xi$ for some $\xi \in \Xi(t)$. When $\mathcal{T}_{ns}$ is dense and
$t\mapsto \Xi(t)$ is inner semicontinuous on $\mathcal{T}_{ns}$, it can always
be extended to an inner semicontinuous map on $\mathcal{T}$ by defining
\begin{align} \label{eq:extend}
\Xi(t) = \{\xi \in \R^p : \xi = \lim_{m\to\infty} \xi_m,
\text{ for some } t_m \to t, \xi_m \in \Xi(t_m)\},~~t\in\mathcal{T}_s.
\end{align}

\begin{corollary}\label{cor:main_singular_fullT}
    Under the assumptions of Theorem \ref{thm:main_singular}, suppose
    additionally that (i) for every $n$, $\ell_n:\Theta\to\R$ is lower
    semicontinuous, (ii) $t\mapsto \Xi(t)$ and $t\mapsto \Delta(t)$ are both
    inner semicontinuous on $\mathcal{T}_{ns}$, and (iii) $Z$ can be extended to an almost surely
    continuous process on $\mathcal{T}$. Then $\Lambda_n \Lto \Lambda =
    \sup_{t\in\mathcal{T}} X(t)$, where $X$ is defined as in \eqref{eq:X-process}
    with $\Xi$ and $\Delta$ extended to $\mathcal{T}$ using \eqref{eq:extend}.
\end{corollary}

\section{Independent observations}\label{sec:ind}

In this section we treat the special case of LRT statistics for independent and identically distributed (i.i.d.) data. Let $\cY$ be the sample space of data and suppose $Y_i: \Om \to \cY$ are i.i.d.\ for $i=1,\ldots,n$ with density $f(y;\theta)$, with log-likelihood $\ell_n(\theta;\omega) = \sum_{i=1}^n \log f(Y_i(\omega);\theta)$ and normalized score $U_n(t) = n^{-1/2}\sum_{i=1}^n \nabla_\xi \log f(Y_i;\xi_0,t)$, so that $Z_n(t) = A(t)^{-1}U_n(t)$. Define the per-observation cross-information
\beq
I(t,\tpr) = \E_{\xi_0}\{\nabla_\xi \log f(Y_1;\xi_0,t)\nabla_\xi \log f(Y_1;\xi_0,\tpr)^\tsp\},
\lb{Ittpr}
\eeq
so that $I(t) = I(t,t)$. In order to verify the results of Section \ref{sec:pd_inf} for i.i.d.\ observations it is possible to derive general conditions under which Assumption~\ref{assm:equicont} and Assumption~\ref{assm:zn_conv} hold, whereas the other assumptions of Theorem~\ref{thm:main_nonsing} are more problem specific. Sufficient conditions for Assumption~\ref{assm:equicont} are in Proposition~\ref{prop:assm_2_sufficient}.

\begin{proposition} \label{prop:assm_2_sufficient}
   Suppose (i) $Y_1, Y_2, \dots$ are i.i.d.; (ii) the average observed information
   $\hat{I}_n(t) = -n^{-1}\nabla_\xi^2\ell_n(\xi_0,t)$ satisfies $\|\hat{I}_n -
   I\|_{\mathcal{T}} = o_{p^*}(1)$ with $I(t) = \E_{\xi_0}\{\hat{I}_n(t)\}$ satisfying Assumption~\ref{assm:inf_bound};
   and (iii) all third-order partial derivatives of $\log f(y;\theta)$
   with respect to $\xi$ are bounded in a neighborhood of $\xi_0$ by an
   integrable $M(y)$, uniformly in $\theta$. Then, with $U_n(t) = n^{-1/2}\sum_{i=1}^n \nabla_\xi \log f(Y_i;\xi_0,t)$, Assumption~\ref{assm:equicont}
   holds with a remainder term $R_n(\xi,t)$ defined in \re{eq:ln_expand}.
\end{proposition}

Assumption~\ref{assm:zn_conv}  corresponds to an invariance principle for $Z_n$; i.e.\ that $Z_n$ converges weakly towards a Gaussian process $Z$. The process $Z_n:\cT\to \R^p$ above \re{Ittpr} is a standardized score vector
\beq
Z_n(t) = A(t)^{-1} \cdot \frac{1}{\sqrt{n}}\sum_{i=1}^n \nabla_\xi \log f(Y_i;\xi_0,t),
\lb{ZntInd}
\eeq
for i.i.d.\ data. Under the conditions of Proposition~\ref{prop:assm_2_sufficient}, $Z_n$ is measurable and, by the Bartlett identities, it has mean zero and covariance function $\rho$ obtained from \re{Ittpr} and \re{ZntInd} as 
\beq
\rho(t,\tpr) = \E_{\xi_0}\{Z_n(t) Z_n(t')^\tsp\} = A(t)^{-1}I(t,\tpr)A(\tpr)^{-\tsp},
\lb{rho}
\eeq
with marginal covariance $\rho(t,t) = \Id_p$. Verification of Assumption~\ref{assm:zn_conv} reduces to a central limit theorem for $Z_n$ together with asymptotic tightness in $F^p(\mathcal{T})$, the latter typically established by showing the score functions form a Donsker class \citep[Chapter~2]{vandervaart2023weak}.

Under $H_0$, the limit process $Z$ is a centered Gaussian process with covariance $\rho$. Under a sequence
\beq
\theta_n = (\xi_0 + n^{-1/2}v, t_0)
\lb{thn}
\eeq
of contiguous alternatives with $v \ne 0$, $Z$ is non-centered with the same covariance $\rho$ and a mean function \beq
\mu(t;t_0,v) := \rho(t,t_0)A(t_0)^\tsp v
\lb{mucontig}
\eeq
obtained by a Taylor expansion of $f(y;\xi_0+n^{-1/2}v,t_0)$ and $A(t)^{-1}I(t,t_0)=\rho(t,t_0)A(t_0)^\tsp$. This extends a result of \citet{da1977} from $p=1$ to general $p\ge 1$.

\section{Examples}\label{sec:examples}

In this section we illustrate key aspects of our theory with examples. For most of these examples we follow the framework of Section \ref{sec:ind}, so that data are a sequence of $n$ i.i.d.\ random variables $Y_i\in\cY$, with log-likelihood $\ell_n(\theta)$. Moreover, the index set $\cT \subset \R^q$ of $Z_n$ and $Z$ will be a subset of $q$-dimensional Euclidean space for some positive integer $q$. We further assume in this section that the set $C(t)$, introduced in \re{CtDef}, is a closed, convex cone of the following form: There is an integer $0\le r \le p$ (independent of $t$) such that for each $t\in\cT$, the corresponding $C(t)$ is determined by the intersection of $r$ half planes, that possibly depend on $t$. In more detail, for each $t\in\cT$ we associate an $r\times p$ matrix $V(t)$ with linearly independent row vectors $v_1(t),\ldots,v_r(t)$, such that
\beq
C(t) = C_{V(t)} =: \{\xi\in\R^p; \, v_i(t)\xi \ge 0 \mbox{ for }i=1,\ldots,r\}.
\lb{Ct2}
\eeq
From this it follows that $\Delta(t) = A(t)^\tsp C(t) = C_{U(t)} = \{\delta\in\R^p;\, u_i(t)\delta \ge 0 \mbox{ for } i=1,\ldots,r\}$, where $U(t)=V(t)A(t)^{-\tsp}$ is a matrix with linearly independent row vectors $u_i(t) = v_i(t)A(t)^{-\tsp}$. Note that $v_i(t)$ and $u_i(t)$ can be renormalized by a positive constant without changing $C(t)$ and $\Delta(t)$ respectively.

Since $\Delta(t)$ is also a closed convex cone, we deduce from \re{XtProj} that $X(t)=\|P_{\Delta(t)} Z(t)\|_2^2$ is the squared norm of the projection of $Z(t)$ onto $\De(t)$. In order to find the marginal distribution of $X(t)$, assume first that $\xi=\xi_0$. It then follows from Section \ref{sec:ind} that $Z$ is a centered Gaussian process with marginal covariance  matrix $\Id_p$, and consequently the process $X$ is a $\bachi^2$-process with a marginal distribution of $X(t)$ given by \re{chibar}, where $w_i(t)$ is the probability that the projection $P_{\De(t)} Z(t)$ belongs
to a face of $\De(t)$ with dimension $i$, see \cite{sesi2002} and
references therein. From this we conclude that $X(t)$ is a mixture of $r$ $\chi^2$-distributions, corresponding to the $r$ nonzero weights $w_{p-r+1}(t),\ldots,w_p(t)$. In particular, when $r=0$ we have $C(t)=\De(t)=\R^p$, so that $w_p(t)=1$ in \re{chibar}, and
$X$ is a pure $\chi^2_p$-process. When $r=1$, $\De(t)$ is a half-plane
and $X(t)$ is a $0.5:0.5$-mixture of $\chi^2$-distributions with 
$p-1$ and $p$ degrees of freedom, i.e.\ $w_{p-1}(t)=w_p(t)=0.5$. When
$r=2$, $X(t)$ is a mixture of three types of $\chi^2$-distributions, with
weights $w_{p-2}(t)=0.5-w_p(t)$, $w_{p-1}(t)=0.5$ and
\beq
w_p(t) = \cos^{-1}(-u_1(t)u_2(t)^\tsp/\| u_1(t)\|_2 \|u_2(t)\|_2)/2\pi,
\lb{wDef}
\eeq
where $u_1(t)$ and $u_2(t)$ are the two row vectors of $U(t)$. 

On the other hand, if $\theta_n = (\xi_0+n^{-1/2}v,t_0)$, then $Z(t)\sim N(\mu(t),\Id_p)$, with $\mu(t)=\mu(t;t_0,v)$ given by \re{mucontig}. From this it follows that $X$ is a noncentral $\bachi^2$-process with marginal distribution $X(t)=\|P_{\Delta(t)} [\mu(t)+\va(t)]\|_2^2$ and $\va(t)\sim N(0,\Id_p)$. This distribution is a mixture of truncated, noncentral $\chi^2$-distributions, whose weights $w_{p-r+1}(t),\ldots,w_p(t)$ not only depend on $\De(t)$ but also on $\mu(t)$. 

Since $\ell_n$ is the i.i.d.\ log-likelihood, the random variable $\Lambda_n$ in \re{eq:Lambda_composite} is the log-likelihood ratio test statistic for testing the null hypothesis $H_0$ in \re{H0H1} against $H_1$. Let $0\le s \le p$ be the dimensionality of the null parameter set $\Xi_0$ in \re{H0H1}. In particular, a simple null hypothesis $\Xi_0=\{\xi_0\}$ corresponds to $s=0$.   

In the examples of Sections \ref{sec:mix_ex}--\ref{sec:polar_ex} we will write $\De(t)=\De$, $A(t)=A$, and $C(t)=C$ whenever these quantities are independent of $t\in\cT$. For each example we will also give the values of $q$, $p$, $r$ and $s$.

\subsection{Mixture and change point models}\label{sec:mix_ex}

In this section we give a number of mixture and change point models for which our results can be applied. Rather than a complete verification of all assumptions in each example, which would take substantial space, we at several points refer to existing results for that  particular setting. Asymptotic distribution results for the LRT statistic of other mixture models can be found in \cite{bich1993, chla1995, li1995, lepo1997, daga1999} and references therein. Further results and references on change point testing are provided in \citep{anpl1995, song2009}.

\begin{example}[Mixture distributions with $q=p=r=1$, $s=0$]\label{ex:gauss_mix_1}
Assume, independently for $i \in \{1, \ldots, n\}$,
\begin{equation} \label{eq:xi_def}
    Y_i \sim (1-\xi)N(0,1) + \xi N(t,1)
\end{equation}
with $\cY=\R$, $\cT=[-T,T]$ for some $T>0$ and $\Xi=[0,1]$.
The mixing weight $\xi$ is the parameter of interest, $t$ (the component mean) is
unidentifiable under the simple null hypothesis $H_0:\xi=0$, so that $s=0$ and $\xi_0=0$ is on the boundary of $\Xi$.
It follows that $p=1$, $C=\De=[0,\infty)$,
the per-observation score is $\nabla_\xi \log f(Y_i;\xi_0,t) = \exp(tY_i-t^2/2) - 1$,
the cross-information in \re{Ittpr} simplifies to $I(t,\tpr) = \exp(t\tpr)-1$ (so that $I(t) = I(t,t)=\exp(t^2)-1$ is singular at $\cT_s=\{0\}$). It follows from \re{ZntInd} that 
\[
Z_n(t) = \frac{1}{\sqrt{n\{\exp(t^2)-1\}}}\sum_{i=1}^n \left\{\exp(tY_i - t^2/2)-1 \right\}
\]
if $t\ne 0$. Although $Z_n$ is not defined at $t=0$, it still has right and left hand limits $Z_n(0+)=\sum_{i=1}^n Y_i/\sqrt{n}$ and $Z_n(0-)=-Z_n(0+)$ respectively. This implies that $Z_n$ as well as the limit process $Z$ are defined on $\cT_{ns}=[-T,0)\cup (0,T]$. The covariance function of $Z_n$ and $Z$ is
\beq
\rho(t,\tpr) = \{\exp(t\tpr)-1\}/[\{\exp(t^2)-1\}\{\exp((\tpr)^2)-1\}]^{1/2}.
\lb{rhoMix}
\eeq
We have that $\rho(t,\tpr)$ is positive if $t\tpr >0$ and negative if $t\tpr < 0$. The mean function of $Z$ is $\mu(t)=0$ under $H_0$, and it is $\mu(t)=\rho(t,t_0)\sqrt{I(t_0)}v$ under a sequence \re{thn} of contiguous alternatives. Since $\De=[0,\infty)$, we get $X(t) = \max(Z(t),0)^2$ for $t\in\cT_{ns}$. The definition of $X$ can be extended to all $t\in\cT$ in at least two ways. Either we make use of $X_n(0)=0$ (which follows from \re{eq:X_n-process}) and put $X(0)=0$. Alternatively, we first extend the definition of $Z$ to $\cT$ by requiring that $Z$ is cadlag with $Z(0)=Z(0+)$, and then put $X(0)=\max(Z(0),0)^2$. This leads to an extension of the covariance function of $Z$ to $\rho(t,t^\prime)=\lim_{u\to t+,u^\prime \to\tpr +}\rho(u,u^\prime)$ for $(t,\tpr)\in \cT\times\cT\setminus \cT_{ns}\times\cT_{ns}$. Regardless of which definition of $X$ we use,   
\beq
\Lambda_n\Lto \Lambda^{(ns)} = \sup_{t\in \cT_{ns}} X(t) = \sup_{t\in \cT} X(t) = \Lambda 
\lb{LambdaEx5p1}
\eeq
as $n\to\infty$. This been established under $H_0$ by \citet[Theorem 2]{chen2001large}.

The null hypothesis can be rephrased as $H_0: t=t_0=0$. This corresponds to an interchange of the roles of $\xi$ and $t$, as nuisance and structural parameters respectively. However, it turns out that Theorem \ref{thm:main_singular} does not give the correct asymptotic distribution of the LRT statistic $\Lambda_n$ for this reparametrization. In order to demonstrate this we use asterisks for all quantities for which $\xi$ and $t$ have been switched. Note first of all that $p^\ast=q^\ast=1$, $r^\ast=s^\ast=0$, and $\Delta^\ast=\R$. The per observation score is  $\nabla_t \log f(Y_i;\xi,t_0) = \xi Y_i$, and the cross-information is given by $I^\ast(\xi,\xi^\prime) = \xi\xi^\prime$, implying that $I^\ast(\xi) = I^\ast(\xi,\xi)=\xi^2$ is singular at $\xi_s=\{0\}$. The standardized score function
$Z_n^\ast(\xi) = \sum_{i=1}^n Y_i/\sqrt{n}=Z_n(+0)$ is independent of $\xi\in (0,1] = \Xi\setminus \{\xi_s\}=:\Xi_{ns}$. This gives rise to a constant limit processes $Z^\ast(\xi) = Z(0+)$ and $X^\ast(\xi)=Z^\ast(\xi)^2 = Z^2(0+)$ on $\Xi_{ns}$. The definition of $X^\ast$ can then be extended to $\Xi$ in at least two ways. Either we make use of $X^\ast_n(0)=0$ (which follows from \re{eq:X_n-process}) and put $X^\ast(0)=0$. Alternatively we extend $X^\ast$ by continuity and put $X^\ast(0)=X^\ast(0+)$. Regardless of which definition of $X^\ast$ on $\Xi$ we use,  
\beq
\Lambda^{\ast (ns)} = \sup_{0 < \xi \le 1} X^\ast(\xi) 
=  \sup_{0 \le \xi \le 1} X^\ast(\xi) = \Lambda^\ast = Z^2(0+) = \max[X(0-),X(0+)].
\lb{Lambdaastns}
\eeq
Consequently, the distribution of $\Lambda^\ast$ in \re{Lambdaastns} is different from that of $\Lambda$ in \re{LambdaEx5p1}, and the two distributions only agree in the limit $T\to 0+$. The reason why Theorem \ref{thm:main_singular} fails in \re{Lambdaastns} when $t$ is chosen as structural parameter is that Assumption \ref{assm:Rtilde_bound} is violated. That is, the remainder term of the quadratic approximation of the log likelihood is not uniformly small close to the point $\xi_s=0$ of singularity 
\hfill\slut 
\end{example}

\begin{example}[Mixture distributions with $q=2,p=r=1$, $s=0$]\label{ex:gauss_mix_2comp}
We generalize Example \ref{ex:gauss_mix_1} so that the mean and variance of the
second component in \eqref{eq:xi_def} are unknown, i.e.,
\[
Y_i \sim (1-\xi)N(0,1) + \xi N(t_1,t_2),
\]
with $\cY=\R$, $t=(t_1,t_2)$ and $\cT=[-T_1,T_1]\times [T_{21},T_{22}]$ for some $T_1>0$ and
$0 < T_{21} < T_{22} < 2$. In this case $q=2$, $p=1$, $C=\Delta=[0,\infty)$,
and the per-observation score at the simple null hypothesis $\xi_0 = 0$, corresponding to $s=0$, is
\[
\nabla_\xi \log f(Y_i;\xi_0,t) = \frac{1}{\sqrt{t_2}}\exp\left\{\frac{1}{2}(1-1 /t_2)Y_i^2
+ \frac{t_1}{t_2}Y_i - \frac{t_1^2}{2t_2}\right\} - 1.
\]
Since $I(t)$ vanishes at $t = (0,1)$, we have $\cT_s=\{(0,1)\}$, and the theory
of Section~\ref{sec:singular} is needed. The restriction $t_2<2$ ensures $I(t)$
is finite. As in Example~\ref{ex:gauss_mix_1}, $C=\De = [0,\infty)$ gives
\[
\Lambda = \sup_{t\in \cT} \,\,\max(0,Z(t))^2,
\]
where the covariance function $\rho(t,\tpr)$ of $Z$ is deduced from \re{Ittpr} and \re{rho} on $\cT_{ns}\times \cT_{ns}$ and then extended to all of $\cT\times \cT$. The mean function of $Z$ is $\mu(t)=0$ under $H_0$, and it equals $\rho(t,t_0)\sqrt{I(t_0)}v$ under a sequence \re{thn} of contiguous alternatives. 
\hfill\slut
\end{example}

\begin{example}[Mixture distributions with $q=r=1,p=2$, $s=0$ or 1]\label{ex:gauss_mix_2} A second generalization of Example
\ref{ex:gauss_mix_1} is
\beq
Y_i \sim (1-\xi_1)N(\xi_2,1) + \xi_1N(t,1),
\lb{YiMixNuisance}
\eeq
with $\cY=\R$, $\cT=[-T,T]$, $\xi=(\xi_1,\xi_2)^\tsp$ and $\Xi= [0,0.5]\times [-T,T]$.
The restriction $\xi_1\le 0.5$ is imposed for identifiability, since
$\theta=(\xi_1,\xi_2,t)$ and $\theta'=(1-\xi_1,t,\xi_2)$ give the same model.
Let us first consider a null parameter that corresponds to a simple null hypothesis $\xi_0=(0,0)^\tsp$ and $s=0$. The per-observation score at $\xi_0$ is
\[
\nabla_\xi \log f(Y_i;\xi_0,t) = [\exp(t Y_i-t^2/2)-1,Y_i]^\tsp,
\]
the cross information matrix is
\begin{equation}
I(t,\tpr) = \left( \begin{array}{cc} e^{t\tpr}-1 & t \\ \tpr &  1 \end{array} \right)
\lb{IMix2}
\end{equation}
and $\cT_s=\{0\}$.
We use a Cholesky decomposition of $I(t)=I(t,t)$, that is, choose a lower triangular $A(t)$ satisfying
$A(t)A(t)^\tsp = I(t)$, defined through its inverse
\begin{equation}
A(t)^{-1} = \left( \begin{array}{cc} 1/\sigma_t  &-t/\sigma_t  \\ 0  & 1 \end{array} \right),
\lb{AMix2}
\end{equation}
with $\sigma_t^2 = e^{t^2}-1-t^2$. From the definition $Z_n(t) =
A(t)^{-1}U_n(t)$ it follows that $Z_n(t)=[Z_{n1}(t),Z_{n2}(t)]^\tsp$, where
$Z_{n1}(t) = \sum_{i=1}^n
(\exp(tY_i-t^2/2)-1-tY_i)/(\sqrt{n}\sigma_t)$ and $Z_{n2}(t) =\sum_{i=1}^n Y_i/\sqrt{n}$ for $t\ne 0$. By taking the limit of $Z_{n1}(t)$ as $t\to 0$, $Z_{n1}(0) = \sum_{i=1}^n (Y_i^2-1)/\sqrt{2n}$. Hence we extend the definition of $Z_n$ and the limit process $Z$ from $\cT_{ns} = [-T,0)\cup (0,T]$ to $\cT=[-T,T]$. The covariance function of $Z_n$ and $Z$ is found by inserting \re{IMix2} and \re{AMix2} into \re{rho}. With $\sigma_{t\tpr} = e^{t\tpr}-1-t\tpr$, this gives
\beq
    \rho(t,\tpr) = \left( \begin{array}{cc} \sigma_{t\tpr}/(\sigma_t\sigma_{\tpr}) & 0 \\ 0 & 1 
    \end{array} \right),
\lb{rhoMix2}
\eeq
for $(t,\tpr)\in \cT_{ns}\times \cT_{ns}$, and then by continuous extension for $(t,\tpr)\in \cT\times \cT$. 
Hence, $Z(t)=[Z_1(t),Z_2]^\tsp$,
where $Z_2\sim \rN(\mu_2,1)$ and $Z_1$ is a real-valued Gaussian process, independent
of $Z_2$, with mean function $\mu_1(t)$ and covariance function $\sigma_{t\tpr}/(\sigma_t\sigma_{\tpr})$. Under $H_0$ we have $\mu_1(t)=\mu_2=0$, whereas it follows from \re{mucontig} that
$$
\mu(t) = \left( \begin{array}{c} \mu_1(t) \\ \mu_2 \end{array}\right) = 
 \left( \begin{array}{cc} \sigma_{tt_0}/(\sigma_t\sigma_{t_0}) & 0 \\ 0 & 1
    \end{array} \right) 
\left( \begin{array}{cc} \sigma_{t_0} & 0 \\ t_0 & 1\end{array} \right) 
\left( \begin{array}{c} v_1 \\ v_2 \end{array}\right)
=  \left( \begin{array}{c} (\sigma_{tt_0}/\sigma_{t})v_1 \\ t_0 v_1+v_2\end{array}\right)
$$
under a sequence \re{thn} of contiguous alternatives, with $v=(v_1,v_2)^\tsp$. Since $\xi_1 \geq 0$, the constraint on $\xi$ implies
$C = \{(c_1, c_2)^\tsp :\, c_1 \ge 0\}$,
the upper half plane. Applying $A(t)^\tsp$ using \re{AMix2}, the corresponding set of $\delta$ is
\[
\De = A(t)^\tsp C = \{(\delta_1,\delta_2)^\tsp :\, \delta_1 \ge 0\}
\]
for all $t\in\cT$. It follows that $X(t) = \Vert P_\De Z(t)\Vert_2^2 =\max(Z_1(t),0)^2  + Z_2^2$, and hence
\beq
\Lambda =  Z_2^2 + \sup_{-T\le t \le T} \max(Z_1(t),0)^2.
\lb{LambdaSimple}
\eeq
That $\Lambda_n\Lto\Lambda$ has been proved under $H_0$ by
\citet[Theorem 3]{chen2001large}.

To illustrate Theorem~\ref{thm:composite}, we consider another null hypothesis $H_0:\xi_1=0$, $\xi_2$ free, which tests whether the two-component mixture reduces to a single Gaussian with unknown mean. This corresponds to a composite null, more specifically a null parameter space $\Xi_0=\{0\}\times [-T,T]$ and $s=1$. Under $H_0$, the model is $Y_i\sim N(\xi_2,1)$ regardless of $t$, so $t$ is unidentifiable under $H_0$.
Suppose the true value of $\xi$ is $\xi_n = \xi_0 + n^{-1/2}v$, with $\xi_0=(0,0)^\tsp$. Optimizing over $\xi_2$ gives $X_n^0(t)=2\sup_{\xi_2\in[-T,T]}\{\ell_n(0,\xi_2,t)-\ell_n(0,0,t)\} = n\bar{Y}_n^2$,
which does not depend on $t$ and converges in law to $Z_2^2$. Thus
$X^0(t)=Z_2^2$ is a constant process, and $\|X_n^0\|_{\cT}\Lto Z_2^2$. By
Theorem~\ref{thm:composite} and the limit \re{LambdaSimple} from the simple null above,
\beq
\Lambda_n \Lto \La = \bigl[Z_2^2 + \sup_{-T\le t\le T}\max(Z_1(t),0)^2\bigr] - Z_2^2
= \sup_{-T\le t \le T} \max(Z_1(t),0)^2.
\lb{LambdaMixture2}
\eeq
The $Z_2^2$ contribution from the identifiable mean parameter $\xi_2$ cancels exactly, and the reason is that the composite $H_0$ satisfies the conditions of Theorem \ref{thm:decomposition}, with $K(t)=[0,\infty)\times \{0\}$ for all $t\in\cT$. Note also that the limiting distribution of $\La$ in \re{LambdaMixture2} holds both under the composite null hypothesis, as well as under a sequence \re{thn} of contiguous alternatives. However, the distribution of $\Lambda$ in \re{LambdaMixture2} is still different from the limit distribution of $\Lambda$ when $\xi_2$ is known (cf.\ Example~\ref{ex:gauss_mix_1}), in spite of the fact that $H_0$ satisfies the conditions of Theorem \ref{thm:decomposition}. Indeed, the covariance function $\sigma_{t\tpr}/(\sigma_t\sigma_{\tpr})$ of $Z_1$ in \re{LambdaMixture2} is different from the covariance function \re{rhoMix} of the limit process $Z$ of Example~\ref{ex:gauss_mix_1}. This illustrates that Theorem~\ref{thm:decomposition} does not give sufficient conditions for the asymptotic distribution of the LRT statistic $\Lambda_n$ not to depend on the nuisance parameters of $\xi$. The reason for the discrepancy between the simple and composite null distributions of $\Lambda$ is that \re{IBDiag} of Corollary \ref{Cor:Decomp} is not satisfied for model \re{YiMixNuisance}, since $I(t)=I(t,t)$ in \re{IMix2} is not diagonal.  

For asymptotics of the LRT statistics under the composite $H_0$, \re{LambdaMixture2} is the known-variance special case of \citet{chen2003tests}, who derive the same type of limit for the composite null in a normal mixture with unknown variance $\sigma^2$; the additional $\chi^2_1$ contribution from estimating $\sigma^2$
cancels by the same mechanism, covered by Theorem~\ref{thm:decomposition}.
\hfill\slut
\end{example}

\begin{example}[Change point detection with $p=2$, $r=0$ and $s=q=1$]\lb{Ex:CP}
Consider a sequence of independent random variables $Y_{ni}\sim N(m(d_{ni}),\sigma^2)$ for $i=1,\ldots,2n$, where $\sigma^2$ is known, $m(d)= m(d;\theta) = m(d;\xi_1,\xi_2,t)=\xi_2 + \xi_1(d-t)_+$ for $0\le d \le 1$, and $d_{ni}=\min((i-1)/n,1)$. This is a nonlinear regression model where $m:[0,1]\to\R$ has a change point at $t\in \cT = [0,1]$, so that $q=1$. The parameter space for $\xi=(\xi_1,\xi_2)^\tsp$ is $\Xi = \R^2$. Testing whether a change point exists corresponds to a composite null hypothesis $H_0: \xi_1=0$, so that $p=2$, $s=1$, and $\Xi_0=\{0\}\times \R$. There are no boundary constraints ($r=0$) since the alternative hypothesis is $H_1: \xi_1\ne 0$. Note also that $\{Y_{ni}; \, n\ge 1, 1\le i \le 2n\}_{n,i}$ is a triangular array of random variables. Differentiating the log likelihood with respect to $\xi$ gives rise to a normalized score vector $U_n(t)=\nabla_\xi \ell_n(\xi,t)/n^{1/2} = (U_{n1}(t),U_{n2}(t))^\tsp$ at $t\in\cT$ with components
$$
\begin{array}{rcl}
U_{n1}(t) &=& \sum_{i=1}^{2n} (Y_{ni}-m_{ni})(d_{ni}-t)_+/(\sigma^2n^{1/2}),\\
U_{n2}(t) &=&  \sum_{i=1}^{2n} (Y_{ni}-m_{ni}) /(\sigma^2n^{1/2})
\end{array}
$$  
and $m_{ni} = m(d_{ni})$. The asymptotic covariance matrix under $H_0$ between $U_n(t)$ and $U_n(\tpr)$, as $n\to\infty$, is 
\beq
I(t,\tpr) = \frac{1}{\sigma^2} \left( \begin{array}{rcl} 
h_{t\tpr} & g_t \\ g_{\tpr} & 2
\end{array}\right),
\lb{IttprCP}
\eeq
where $h_{t\tpr} = \int_0^1 (x-t)_+(x-\tpr)_+ dx + (1-t)(1-\tpr)$ and $g_t = \int_0^1 (x-t)_+ dx + 1-t = (1-t)^2/2+1-t$.  Since $I(t)=I(t,t)$ is singular at $t=1$ we have $\cT_s=\{1\}$. For $t\in\cT_{ns}=[0,1)$ we choose a lower triangular inverse $A(t)$ of $I(t)$, with inverse 
\beq
A(t)^{-1} = \sigma \left( \begin{array}{rcl} \frac{1}{\sigma_t} & -\frac{g_t}{2\sigma_t}\\ 0 & \frac{1}{\sqrt{2}}\end{array}\right),
\lb{AtCP}
\eeq
where $\sigma_t^2 = h_{tt} - g_t^2/2$. This gives rise to a standardized score vector $Z_n(t)=A(t)^{-1}U_n(t)=(Z_{n1}(t),Z_{n2}(t))^\tsp$ with components 
\beq
\begin{array}{rcl}
Z_{n1}(t) &=& \sum_{i=1}^{2n} [(Y_{ni}-m_{ni})(d_{ni}-t)_+ - (Y_{ni}-m_{ni})g(t)/2 ]/ (\sigma \sigma_t n^{1/2}) \\
Z_{n2}(t) &=& \sum_{i=1}^{2n} (Y_{ni}-m_{ni})/(\sqrt{2}\sigma n^{1/2})
\end{array}
\lb{ZnCP}
\eeq
for all $t\in\cT_{ns}$. It follows, by inserting \re{IttprCP} and \re{AtCP} into \re{rho}, that the covariance function $\rho(t,\tpr)$ of the asymptotic limit $Z$ of $Z_n$, is given by \re{rhoMix2}, with $\sigma_t^2$ defined below \re{AtCP} and $\sigma_{t\tpr} = h_{t\tpr} - g_t g_{\tpr}/2$ for $(t,\tpr)\in \cT_{ns}\times \cT_{ns}$. By extending this covariance function to $\cT\times\cT$ a limit process $Z(t)=(Z_1(t),Z_2)^\tsp$ is obtained for all $t\in \cT=[0,1]$. Since $C=\R^2$ and $C^0 = \{0\}\times \R$, we deduce $\Delta = A(t)^\tsp C = \R^2$ and $\Delta^0 = A(t)^\tsp C^0 = \{0\}\times \R$. It follows from Theorem \ref{thm:main_singular} that this gives rise to an asymptotic LRT statistic 
\beq
\Lambda = \sup_{0\le t \le 1} [X(t)-X^0(t)] =  \sup_{0\le t \le 1} [Z_1(t)^2+Z_2^2-Z_2^2]
 = \sup_{0\le t \le 1} Z_1(t)^2. 
\lb{LambdaCP}
\eeq
It is possible to reformulate the change point problem as testing $H_0: t=1$ against $H_1: 0\le t < 1$. For this parametrization, it turns out that Theorem \ref{thm:main_singular} no longer gives a valid expression for the asymptotic distribution of the LRT statistic. In order to explain this we use asterisks for all quantities where the value of $t$ rather than $\xi_1$ is tested. Since the structural parameter is $(t,\xi_2)$, and the non-identifiable parameter under $H_0$ is $\xi_1$, we see that the roles of $\xi_1$ and $t$ have been interchanged, with $p^\ast =2$, $r^\ast=q^\ast=s^\ast = 1$, $C^\ast = (-\infty,0]\times\R$, and $C^{0\ast}=\{0\}\times\R$. With this new parametrization, the normalized likelihood score vector $U_n^\ast(\xi_2) = \nabla_{t,\xi_1} \ell_n(\xi_1,\xi_2,t)/n^{1/2} = (U_{n1}^\ast(\xi_2),U_{n2}^\ast(\xi_2))^\tsp$ has components 
$$
\begin{array}{rcl}
U_{n1}^\ast(\xi_2) &=& -\xi_2\sum_{i=n+1}^{2n} (Y_{ni}-m_{ni})/(\sigma^2 n^{1/2}),\\
U_{n2}^\ast(\xi_2) &=& \sum_{i=1}^{2n} (Y_{ni}-m_{ni})/(\sigma^2 n^{1/2}).
\end{array}
$$
The covariance between $U_n^\ast(\xi_2)$ and $U_n^\ast(\xi_2^\ast)$ is 
\beq
I^\ast(\xi_2,\xi_2^\ast) = \frac{1}{\sigma^2} \left( \begin{array}{cc}
 \xi_2\xi_2^\prime & -\xi_2 \\ -\xi_2^\prime & 2.
  \end{array}\right),
\lb{Iast}
\eeq
Notice that $I^\ast(\xi_2)=I^\ast(\xi_2,\xi_2)$ is singular at $\xi_2=0$. We therefore define $A^\ast(\xi_2)=I^\ast(\xi_2)^{-1}$ for $\xi_2\ne 0$ through its inverse 
\beq
A^\ast(\xi_2)^{-1} = \sigma \left(\begin{array}{rcl}
- \frac{1}{\sigma_{\xi_2}} & - \frac{\xi_2}{2\sigma_{\xi_2}} \\ 0 & \frac{1}{\sqrt{2}} 
\end{array}\right),
\lb{Aast}
\eeq
with $\sigma_{\xi_2}^2 = \xi_2^2/2$. The standardized score vector process $Z_n^\ast(\xi_2) = A^\ast(\xi_2)^{-1}U_n^\ast(\xi_2) = (Z_{n1}^\ast (\xi_2),Z_{n2}^\ast (\xi_2))^\tsp = ( \mbox{sgn}(\xi_2) Z_{n1}^\ast , Z_{n2}^\ast )^\tsp$ is defined for $\xi_2\ne 0$ with $\mbox{sgn}(\xi_2)=\xi_2/|\xi_2|$ and  
\beq
\begin{array}{rcl}
Z_{n1}^\ast &=&  \left[ \sum_{i=n+1}^{2n} (Y_{ni}-m_{ni}) -\sum_{i=1}^n (Y_{ni}-m_{ni})\right]/(\sqrt{2}\sigma n^{1/2}),\\
Z_{n2}^\ast &=& \sum_{i=1}^{2n} (Y_{ni}-m_{ni})/(\sqrt{2}\sigma n^{1/2}).
\end{array}
\lb{Znast}
\eeq
It follows by inserting \re{Iast} and \re{Aast} into \re{rho}, or directly from \re{Znast}, that the covariance function of $Z_n^\ast$ and its limit process $Z^\ast(\xi_2)=(Z_1^\ast(\xi_2),Z_2^\ast(\xi_2))^\tsp = (\mbox{sgn}(\xi_2)Z_1^\ast,Z_2^\ast)^\tsp$, are given by
\beq
\rho(\xi_2,\xi_2^\prime) = \left(\begin{array}{cc}
\mbox{sgn}(\xi_2)\mbox{sgn}(\xi_2^\prime) & 0 \\ 0 & 1
\end{array}\right)
\lb{rhoast}
\eeq
for all $\xi_2\ne 0$ and $\xi_2^\prime\ne 0$. Since the transformed convex cones are $\Delta^\ast = A^\ast(\xi_2)^T C^\ast = [0,\infty)\times\R$ and $\Delta^{0\ast} = A^\ast(\xi_2)^T C^{0\ast} = \{0\}\times\R$, we find that 
\beq
\Lambda^\ast =  \sup_{\xi_2\ne 0} [X^\ast(\xi_2) - X^{0\ast}(\xi_2)] 
= \sup_{\xi_2\ne 0} [\max(0,Z_1^\ast(\xi_2))^2 + Z_2^\ast - Z_2^\ast] = (Z_1^\ast )^2.
\lb{LambdaastCP}
\eeq 
It can be shown that the first component of \re{ZnCP} satisfies $Z_{n1}(t)\to Z_{n1}^\ast$ as $t\to 1-$, and hence by taking the limit $n\to\infty$ we deduce $Z_1(t)\to Z_1^\ast$ as $t\to 1-$. From this it follows that the distribution of $\Lambda$ in \re{LambdaCP} differs from that of $\Lambda^\ast$ in \re{LambdaastCP}. We see that \re{LambdaastCP} only captures the maximum $Z_1^2(\cdot)$ locally at $t=1$, whereas  \re{LambdaCP} captures the maximum of $Z_1^2(\cdot)$  over all $t\in [0,1]$. The reason why $\Lambda^\ast$ does not give the correct asymptotic distribution of the LRT statistic is that  Assumption \ref{assm:Rtilde_bound} of Theorem \ref{thm:main_singular} fails.      
\end{example}

\subsection{Linkage analysis and MOD scores}\label{sec:linkage}

It has been noticed that some test statistics arising in genetic linkage
analysis can be formulated as the supremum of a one-dimensional Gaussian or
$\chi^2_1$-process along an interval, see, for instance, \citet{lander1989mapping} and
\citet{feingold1993gaussian}. \citet{dupuis1995statistical} define another test
statistic that asymptotically equals the supremum of a $\chi^2_2$-process along
an interval. In this section, we show that LR-tests in linkage analysis, usually
referred to as MOD scores (maximized over disease models), fit into our framework.

In order to present these asymptotic LR-test results, let us start by briefly summarizing the genetic linkage model; see  \citet{sham1998statistics} for more details. The objective is to test whether the position (or locus) $\tau$ of a disease susceptibility gene is along a given chromosome $[0,T]$ of genetic map length $T$:
\begin{equation}
\begin{array}{ll}
H_0: \tau\notin [0,T],\\
H_1: \tau\in [0,T].
\end{array}
\lb{H01Def3}
\end{equation}
At our disposal we have $n$ families, with information on how the disease (phenotypes) as well as DNA is segregated in each family.  

\subsubsection{One family type} \lb{Sec:OneFam} 

In this section we assume that all $n$ families are of the same type. By this we mean that not only the number $N$ of individuals is the same for each family, but also the pedigree structure
$\cP$ and the phenotype vector $\Phi = (\Phi_k,\, k\in\cP)$. The pedigree consists of two disjoint founder and non-founder sets $\cF$ and $\cN$, such that founders ($k\in\cF$) have no parents in the pedigree, whereas non-founders ($k\in\cN$) have the mother and father included. Transmission of genetic material at position $s\in [0,T]$ along the chromosome is described by a binary inheritance vector
$v(s) = (v_1(s),\ldots,v_m(s)) \in \{0,1\}^m$, where $m$ is the total number
of meioses and $v_l(s)\in\{0,1\}$ indicates whether grandpaternal or grandmaternal DNA was
transmitted during the $l$th meiosis at locus $s$, with switches between these two states corresponding to points of crossovers. There are two meioses (one from each parent) for each nonfounder, so that $m/2=|\cN|$ and $N-m/2=|\cF|$ is the number of nonfounders and founders of the pedigree, respectively. The complete marker data $Y = \{v(s);\, 0\le s \le T\}\in \cY$ is observed for each family. Here $\cY$ is the space of functions $y:[0,T]\to \{0,1\}^m$ with at most a finite number of jumps, corresponding to the total number of crossovers for all meioses of the pedigree.

We will restrict ourselves to binary phenotypes, so that an individual is either unaffected or affected by a disease. Consider a biallelic disease gene and assume that the allele that increases the risk of being affected has frequency $\rmp \in [0,1]$ in the population. Define penetrance parameters $\ga = (\ga_0,\ga_1,\ga_2)\in [0,1]^3 = \Ga$, where $\ga_j$ is the probability for an individual to be affected given $j$ copies of the disease allele (that is, $j$ corresponds to the number of parents that pass on the disease allele to this individual). With $\theta = (\tau,\rmp,\ga)$, the density of $Y$ conditional on $\Phi$ is
\begin{equation}
f(y;\theta) = \pr_{\scr{p},\ga}(v(\tau)|\Phi)\,\pr(y|v(\tau)).
\lb{fDefLink}
\end{equation}
The exact form of the first term $\pr_{\scr{p},\ga}(v(\tau)|\Phi)$ on the right-hand side of \re{fDefLink} is given in Section \ref{sec:link_details}. Assuming Mendelian inheritance, the null hypothesis in \re{H01Def3} corresponds to a uniform distribution 
\beq
\pr_{\scr{p},\ga}(v(\tau)|\Phi) = 2^{-m}
\lb{PUnif}
\eeq
of $v(\tau)$ over $\{0,1\}^m$. If additionally Haldane's map function for crossovers is assumed, all $\{v_l(s); \, 0\le s \le T\}$ are independent Markov processes on $\{0,1\}$ for $l=1,\ldots,m$, with intensity 1 for each process to switch between its two states (when genetic map length along the chromosome is measured in Morgans). This implies that $\pr(y|v(\tau)) = \exp(-mT)$, independently of the total number of crossovers for all $m$ meioses, and independently of where along the chromosomes these crossovers are located. Since $\pr(y|v(\tau))$ does not depend on $\tau$ it is a constant that can be dropped from the per-pedigree likelihood $f(y;\theta)$. 

To bring the model into the framework of the present article, we reparametrize the penetrance
parameters. For fixed $\rmp$, we introduce an inner product and an associated
orthonormal basis $e_0=(1,1,1), e_1, e_2$ for $\R^3$ depending on $\rmp$ (see Section \ref{sec:link_details} for the definition of the scalar product of $\R^3$, $e_1$ and $e_2$). Write
\begin{equation}
\ga = Ke_0 + \sqrt{\xi_1}\,e_1 + \sqrt{\xi_2}\,e_2.
\lb{zetaExp}
\end{equation}
It is shown in Section \ref{sec:link_details} that $K = \E(\Phi_k)$ is the disease prevalence, and $\xi_1 = \va_1^2$,
$\xi_2 = \va_2^2$ the additive and dominance genetic variances of the disease. We set
\begin{equation}
t = (\tau,\rmp,K), \quad \xi = (\xi_1,\xi_2)^\tsp, 
\lb{tDef}
\end{equation}
with $\cT = [0,T]\times[0,1]^2$ and $\Xi(t) = \{\xi \in [0,\infty)^2 : \ga(\xi;K,\rmp) \in \Ga\}$, corresponding to $q=3$ and $p=2$ respectively. Since the null hypothesis in \re{H01Def3} is mathematically equivalent to \re{PUnif}, under certain conditions it can also be formulated in terms of the two genetic variance components $\xi=(\xi_1,\xi_2)^\tsp$. In order to demonstrate this, we introduce for each $\rmp$ the set $\Ga_0(\rmp)\subset\Ga$ of penetrance vectors $\ga$ for which \re{PUnif} holds. We will assume that the pedigree type satisfies 
\beq
\Ga_0(\rmp) =  \{\ga:\ga_0=\ga_1=\ga_2\} \mbox{ for all }\rmp.
\lb{Garmp}
\eeq
In general $\Ga_0(\rmp)$ is at least as large as the right-hand side of \re{Garmp}. For instance, if the phenotype of at most one individual of the pedigree is known, then $\Ga_0(\rmp)=\Ga$ for all $\rmp$ and \re{Garmp} fails. Condition \re{Garmp} will however hold for most pedigree types of interest, and in view of \re{zetaExp} it implies that the null hypothesis condition \re{PUnif} is equivalent to a simple null hypothesis $\xi_0 = (0,0)^\tsp$ for the genetic variance vector $\xi$, corresponding to a simple null hypothesis and $s=0$. This is to say that testing whether the disease gene of interest is located along the chromosome $[0,T]$ or not is mathematically equivalent to testing whether $\xi$ equals $\xi_0$ or not. 

The above discussion implies that $f(y;\xi,t)=\pr_\xi(v(\tau)|\Phi)$ can be assumed after reparametrization. The per-observation score $\nabla_\xi \log f(y;\xi_0,t) = S(v(\tau);K)$ was derived in \citet{hossjer2005conditional}, based on methods developed in \citet{hossjer2003determining}. Note in particular that $S$ depends on $t$ only through $\tau$ and $K$, but not $\rmp$. Since $p=2$, the per-observation score vector $S = (S_1,S_2)^\tsp$ has two components given by
\begin{equation}
\begin{aligned}
S_1(v) = \displaystyle\sum_{1\le k<l\le N}\om_{kl}\,\IBD_{kl}(v)/2 - C_1;~~
S_2(v)  = \displaystyle\sum_{1\le k<l\le N}\om_{kl}\,\mathbf{1}_{\{\scr{IBD}_{kl}(v)=2\}} - C_2,
\end{aligned}
\lb{S1S2}
\end{equation}
where $\IBD_{kl}(v)\in\{0,1,2\}$ is the number of alleles from the founders of the family that are shared identical-by-descent by $k$ and $l$ from the same founder alleles, $C_i$ is a centering constant assuring that the null hypothesis expectation $E_{\xi_0}\{S_i(v)|\Phi\}=\E\{S_i(v)\}=0$, whereas $\om_{kl}$ is a weight assigned to each pair $(k,l)$ of individuals, based their phenotypes (see  Section~\ref{sec:link_details}). For binary phenotypes we assign $\Phi_k=0$, 1 or ?, depending on whether $k$ is affected (1), unaffected (0) or has unknown phenotype (?). The weights of pair $(k,l)$ then simplify to
\begin{equation}
\om_{kl} = \begin{cases}
K^{-2}, & \Phi_k=\Phi_l=1,\\
-K^{-1}(1-K)^{-1}, & \Phi_k\ne\Phi_l,\;\Phi_k,\Phi_l\in\{0,1\},\\
(1-K)^{-2}, & \Phi_k=\Phi_l=0,\\
0, & \text{otherwise.}
\end{cases}
\lb{omDefBin}
\end{equation}
Since we restrict ourselves to binary phenotypes, we assume that \re{omDefBin} holds. 
It follows from \re{Ittpr} that the cross information matrix is 
\begin{equation}
\begin{array}{rcl}
I(t,\tpr) &=& \E_{\xi_0}\bigl[S(v(\tau);K) S(v(\tau^\prime);K^\prime)^\tsp |\Phi \bigr]\\
&=&  \sum_{v,\vp} S(v;K) S(\vp;K^\prime)^\tsp \, \pr(v(\tau)=v,\,v(\tau^\prime)=\vp)
\end{array}
\lb{IDefLink}
\end{equation}
between $t$ and $\tpr = (\tau^\prime,\rmp^\prime,K^\prime)$. The last sum ranges over all $4^m$ pairs of inheritance vectors $v,\vp$, under the null hypothesis probabilities $\pr_{\xi_0}(v(\tau)=v,\,v(\tau^\prime)=\vp|\Phi)=\pr(v(\tau)=v,\,v(\tau^\prime)=\vp)$ that all components of $Y=\{v_i(s); \, 0\le s \le T\}_{i=1}^m$ are independent and stationary Markov processes on $\{0,1\}$ with intensity 1 of jumping between their two states, independently of $\Phi$. It follows from \re{S1S2} and \re{omDefBin} that the per-observation score only depends on $t$ through $\tau$ and $K$, not $\rmp$. Hence we deduce from \re{IDefLink} that the cross information matrix $I(t,\tpr)$ is only a function of $(\tau,K)$ and $(\tau^\prime,K^\prime)$.  Provided the pedigree type satisfies \re{Garmp}, and if $\eps \le K \le 1-\eps$ for some $\eps > 0$, it follows that the per individual Fisher information matrix $I(t)=I(t,t)$ is nonsingular for all $t \in \cT$, i.e. $\cT_{ns}=\cT$. 

Although $\log f(y;\theta)$ is not twice differentiable in $\xi$ at $\xi_0$
owing to the reparametrization \eqref{zetaExp}, the required quadratic expansion
of the log-likelihood still holds as the number $n$ of families grows:

\begin{theorem}[Quadratic log-likelihood expansion for linkage.]\label{thm:link_quad}
Assume marker data $Y_i$ of families $i=1,\ldots,n$ are independent. The log-likelihood, with $f(Y_i;\theta)$ as in \eqref{fDefLink}, then admits a quadratic expansion of the form \eqref{eq:ln_expand},
with $U_n(t)$ as in Section~\ref{sec:ind}, $I(t) = I(t,t)$ as in
\eqref{IDefLink}, $A(t)$ satisfying $A(t)A(t)^\tsp = I(t)$, and a remainder
$R_n$ satisfying Assumption~\ref{assm:equicont}.
\end{theorem}

In order to find the asymptotic distribution $\La$ of the likelihood ratio test statistic under the simple null hypothesis $H_0$ we use Theorem~\ref{thm:main_nonsing}. Since $\xi_1\ge 0$ and $\xi_2\ge 0$, there are $r=2$ boundary restrictions with $C(t)=C=[0,\infty)^2$ a convex, closed cone, and therefore $\De(t)=A(t)^\tsp C$ is a closed, convex cone as well. Consequently, $X(t) = \|P_{\De(t)}Z(t)\|_2^2$, with $Z$ a centered Gaussian process with covariance function $\rho(t,\tpr)$ given by \re{rho}. Since $I(t,\tpr)$ only depends on $t$ and $\tpr$ through $(\tau,K)$ and $(\tau^\prime,K^\prime)$, the same is true for $\rho(t,\tpr)$, and therefore the limit process $X(t)$ will not depend on $\rmp$. Hence we regard $X(\tau,K)$ as a function of $(\tau,K)$ only, and the asymptotic LRT statistic reduces to $\La = \sup_{\tau \in [0,T],\, K \in [\eps,1-\eps]} X(\tau,K)$. There are $r=2$ constraints that define $\Delta(t)$. From this and \re{chibar} it follows that 
\begin{equation}
X(\tau,K) \sim \bigl(0.5 - w_2(\tau,K)\bigr)\chi^2_0 + 0.5\,\chi^2_1
+ w_2(\tau,K)\chi^2_2,
\lb{LinkWeights}
\end{equation}
has a $\bachi^2$-distribution under $H_0$ for each $(\tau,K)$, with
\beq
w_2(\tau,K) = \cos^{-1}(I_{12}(\tau,K)/\sqrt{I_{11}(\tau,K)I_{22}(\tau,K)})/(2\pi).
\lb{w2DefLink}
\eeq
The weight \re{w2DefLink} of the $\chi_2^2$ component in \re{LinkWeights} can be deduced from \re{wDef}. Indeed, since $C=[0,\infty)^2$, we may take $V(t)=\Id_2$ in \re{Ct2}, giving $u_i(t)=[A(t)^{-\tsp}]_i$ (the $i$-th row of $A(t)^{-\tsp}$). Then $u_i(t)u_j(t)^{\tsp}=[I(t)^{-1}]_{ij}=I^{-1}_{ij}(\tau,K)$, and consequently
$$
\frac{u_1(t)u_2^\tsp(t)}{\|u_1(t)\|_2 \|u_2(t)\|_2} = \frac{I_{12}^{-1}(\tau,K)}{\sqrt{I_{11}^{-1}(\tau,K)I_{22}^{-1}(\tau,K)}}
= -  \frac{I_{12}(\tau,K)}{\sqrt{I_{11}(\tau,K)I_{22}(\tau,K)}}.
$$

\begin{example}[MOD scores with only affecteds.]\label{ex:aff_only}
When all pedigree members with known phenotype are affected, \re{S1S2}-\re{omDefBin} simplify to 
\begin{equation}
S(v) = K^{-2}\bigl(\Spairs(v) - \E_0\{\Spairs\},\;
\Sgprs(v) - \E_0\{\Sgprs\}\bigr)^\tsp,
\lb{Saff}
\end{equation}
where $\Spairs = \sum_{k<l}\IBD_{kl}/2$ \citep{whittemore1994class}, and
$\Sgprs = \sum_{k<l}\mathbf{1}_{\{\scr{IBD}_{kl}=2\}}$ \citep{mcpeek1999optimal},
with sums over all affected pairs $(k,l)$. In this case the limit process $Z(t) = Z(\tau)$ under $H_0$ is a
function of $\tau$ only. It is a two-dimensional stationary Gaussian process on $[0,T]$ with 
\beq
\rho(\tau,\tau^\prime) = \rho(\tau^\prime-\tau).
\lb{rhostat}
\eeq
It is shown in Section \ref{sec:link_details} that the  covariance function has the form 
\beq
\rho(s) = \sum_{l=1}^m \ka_l \exp(-2l|s|), 
\lb{rhos}
\eeq
corresponding to a stationary diffusion; where $\ka_l$ are $2\times 2$ matrices given in Table~\ref{Tab2} for the pedigree types of Figure~\ref{Pedigrees}. For affected sibling pedigrees (types 1--4), the two components $Z_1$ and $Z_2$ of $Z$ are independent Ornstein-Uhlenbeck processes along $[-T,T]$. The asymptotic LR-statistic then simplifies to
\begin{equation}
\La = \sup_{0 \le \tau \le T} X(\tau),
\lb{laLink2}
\end{equation}
with $X(\tau) = \|P_\De Z(\tau)\|_2^2$ a stationary $\bachi^2$-process that is a mixture \re{LinkWeights} between $\chi^2$-distributions with 0, 1 and 2 degrees of freedom.  Since $K^{-2}$ enters as a multiplicative constant in \eqref{Saff}, it follows that $I(t)=I(t,t)=K^{-4}I_0$, for some $2\times 2$ matrix $I_0$ independent of $t$. Hence we deduce from \re{w2DefLink} that the weight $w_2$ of the $\chi_2^2$-distribution is a constant, independent of $t$. Figure~\ref{Pedigrees} shows several pedigree types with only affecteds; the
corresponding $I(t)$, $\De = C_U$, and $w_2$ are in Table~\ref{Tab1} of Section \ref{sec:link_details}.
\hfill\slut
\end{example}

\begin{example}[Unilineal pairs]\lb{Exa:UP}
An affected unilineal pair $(k,l)$ is a pedigree type in which both $k$ and $l$
are affected, all other members have unknown phenotype, and $(k,l)$ can share at
most one allele IBD (for instance, affected first cousins or uncle-nephew pairs, but not
affected sib pairs). For such pedigrees, the above parametrization $\theta=(\xi,t)$ with two genetic variance components in $\xi$ does not work, since \re{Garmp} is not satisfied. For this reason $I(t)$ is singular for {\sl all} values of $t$, and therefore the framework of Section~\ref{sec:singular} does not apply. Instead another parametrization $\theta=(t,\xi)$ is needed with only one genetic variance component in $\xi$. This means that $\cT \subset \R^4$, $\Xi(t) \subset \R$, $p = 1$, and $\De(t) = [0,\infty)$. It can be shown that $X(t)$ is only a function of $t$ through $\tau\in [-T,T]$. Therefore, the asymptotic LRT statistic satisfies \re{laLink2}, which under $H_0$ is the supremum of a $\bachi^2$-process with marginal distribution $X(\tau) \sim 0.5\,\chi^2_0 + 0.5\,\chi^2_1$. 
Formally, this corresponds to a weight $w_2 = 0$ for the $\chi^2_2$-component of $X(\tau)$. 
\hfill\slut
\end{example}

\begin{example}[Affected sib pairs and MLS scores]
For affected sib pairs (pedigree type 1 in Figure~\ref{Pedigrees}), there is an
alternative parametrization $\theta = (\tau,\xi)$ with $\xi = (z_0,z_1)^\tsp$, where
$z_i$ is the probability that the sib pair shares $i$ alleles IBD at $\tau$
\citep{suarez1978generalized}. The resulting test, the MLS score, is
asymptotically equivalent to the corresponding MOD score; a detailed
verification is in Section~\ref{sec:link_details}.
\hfill\slut
\end{example}

\subsubsection{Several family types}\lb{Sec:ManyFT}

Let us generalize the model of Section \ref{Sec:OneFam} and assume that the $n=n_1+\ldots+n_J$ families are of $J > 1$ types $(\cP_1,\Phi_1),\ldots,(\cP_J,\Phi_J)$ with $n_j$ families of type $j$, with a density $f_{j}(y;\theta)$ of marker data for pedigree type $(\cP_j,\Phi_j)$. The combined log-likelihood $\ell_n(\theta;\omega) = \sum_{j=1}^J\sum_{i=1}^{n_j} \log f_{j}(Y_{ji}(\omega);\theta)$ for the whole data set with all $J$ family types generalizes the single-type case, with $Y_{ji}\in\cY_j$ marker data for family $i$ of type $j$. This log-likelihood admits the same type of asymptotic expansion as in Theorem \ref{thm:link_quad}, with asymptotic per-observation cross information matrix $I(t,\tpr) = \sum_{j=1}^J \beta_j I_j(t,\tpr)$, where $\beta_j = \lim_{n\to\infty} n_j/n$ are the asymptotic family type proportions and $I_j(t,\tpr)$ is defined as in \eqref{IDefLink} for pedigree type $(\cP_j,\Phi_j)$.

\begin{example}[MOD score with affected sib pairs and first cousins.]
Put $J=2$ with asymptotic proportions $\beta_1=\beta$ of affected sib pairs and
$\beta_2=1-\beta$ of affected first cousins. Then
\[
I(t) = K^{-4}\left(\begin{array}{cc} 0.125\beta + 0.0469(1-\beta) & 0.125\beta\\
0.125\beta & 0.1875\beta \end{array}\right).
\]
The $\chi^2_2$-weight $w_2$ from \eqref{w2DefLink} is plotted as a function
of $\beta$ in Figure~\ref{Fig2} of Section \ref{sec:link_details}. It is notable that $w_2 \to 0.25$ as
$\beta \to 0^+$, whereas $\beta = 0$ corresponds to the degenerate case of only having unilineal first cousin pairs (Example \ref{Exa:UP}) with $w_2 = 0$. This discontinuity suggests that very large sample sizes are needed
when $\beta$ is positive but small in order for $\Lambda_n$ to be distributed approximately as \eqref{laLink2}.
\hfill\slut
\end{example}

\subsection{Polar-coordinates example}\label{sec:polar_ex}

In this section we illustrate that for models with Cartesian parameters $\eta$, the weak limit of the LRT statistic can be expressed as the supremum of a (noncentral) chibar process, also when the components of $\eta$ are identifiable under the null and alternative hypotheses. This is achieved by switching to polar coordinates. A transformation to polar coordinates can also be used in settings where the original testing problem has non-identifiable nuisance parameters under the null hypothesis \citep{king1993}. 

For simplicity we confine ourselves to a two-dimensional parameter $\eta$, although the results hold more generally. That is, we assume that the common density of the i.i.d.\ random variables $Y_1, \ldots, Y_n$ is $f(y;\eta)$ for some unknown parameter $\eta = (\eta_1,\eta_2)^\tsp$ that belongs to $\R^2$ or a subset thereof. We wish to test a simple null hypothesis $H_0:\eta=\eta_0$ against $H_1:\eta\ne\eta_0$. Consider a sequence of parameter vectors
\beq
\eta_n = \eta_0 + n^{-1/2} v [\cos(t_0),\sin(t_0)]^T =: \eta_0+ n^{-1/2}v J(t_0),
\lb{etanCart}
\eeq
for some fixed $v\ge 0$ and $t_0\in [0,2\pi]$. Note that $\eta_n=\eta_0$ if $v=0$, whereas $\eta_n$ corresponds to a contiguous sequence of alternatives when $v>0$. We will use bars to denote quantities for this model with Cartesian parameters. Thus $\bar{\ell}_n(\eta)$ is the i.i.d.\ log-likelihood and $\bar{U}_n=n^{-1/2}\nabla_\eta \bar{\ell}_n(\eta)$ is the normalized score for $\eta$ at $\eta_0$.  Let $\baI=(\baI_{ij})_{i,j=1}^2=E(\bar{U}_n\bar{U}_n^\tsp)$ be the $2\times 2$ Fisher information matrix for $\eta$ at $\eta=\eta_0$, and 
\beq
\baZ_n = \baA^{-1}\bar{U}_n
\lb{baZnPolar}
\eeq
the standardized score, with $\baA= \baI^{1/2}$ the symmetric square root of $\baI$. 

\subsubsection{Nonsingular case}\lb{Sec:NS}

When $\baI$ is nonsingular, we interpret $\baA^{-1}$ as a proper inverse of $\baA$ in \re{baZnPolar}. Under the conditions of Corollary \ref{cor:no_nuisance}, $\baZ_n\Lto \baZ = (\baZ_1,\baZ_2)^\tsp \sim N(\bamu,\Id_2)$, where $\bamu = v\baA J(t_0)$. 

We will first assume that $\eta_0=(\eta_{01},\eta_{02})^\tsp$ is an inner point of parameter space $\R^2$. This corresponds to $q=s=r=0$ and $p=2$, so that $\bar{C}=\bar{\Delta}=\R^2$. Consequently, under the conditions of Corollary \ref{cor:no_nuisance} we have that
\beq
\Lambda_n\Lto \La = \|\baZ \|_2^2 \sim \chi_{2,\|\bamu\|_2^2}^2
\lb{LambdaCart}
\eeq
as $n\to\infty$, meaning that $\La$ has a noncentral chisquare distribution when $v>0$, with noncentrality parameter $\|\bamu \|_2^2 = v^2 J(t_0)^\tsp \baI J(t_0)$. In polar coordinates we reformulate the model as
    \begin{equation}
        \eta = \eta_0 + \xi[\cos(t),\sin(t)]^\tsp = \eta_0 + \xi J(t)
        \lb{PolarDef}
    \end{equation}
for $t\in\cT = [0,2\pi]$, $\xi\in \Xi=[0,\infty)$ and $\xi_0= 0$. After this reparametrization the null hypothesis is still simple ($s=0$), but now there is a nuisance parameter $t$ of dimension $q = 1$, the parameter $\xi$ of interest is of dimension $p = 1$, and on the boundary of the parameter space ($r=1$).  In polar coordinates, the sequence of parameter values in \re{etanCart} corresponds to $t=t_0$ and $\xi_n = \xi_0 + n^{-1/2}v$. In order to verify that the distribution of $\La$ in \re{LambdaCart} is still the same for polar coordinates, we remove the bars for quantities that are evaluated under the polar coordinates parametrization. The Fisher information for $\xi$ at $\xi_0$ is given by $I(t)=J(t)^\tsp \baI J(t)$, so that $\cT_s=\emptyset$ when $\baI$ is nonsingular. Since $p=1$ and $I(t)$ is scalar, $A(t)=I(t)^{1/2}$ is scalar as well. The standardized score $Z_n(t) = A(t)^{-1}U_n(t)$ for polar coordinates can be written as
$$
 Z_n(t) = I(t)^{-1/2}U_n(t) = I(t)^{-1/2}J(t)^\tsp\bar{U}_n = I(t)^{-1/2}J(t)^\tsp\baA \baZ_n =: K(t)^\tsp \baZ_n.
$$
It is easy to see that $K(t)^\tsp K(t) = \|K(t)\|_2^2=1$, and for this reason we can write $K(t)=I(t)^{-1/2}\baA J(t)=J(g(t))$, were  $g: [0,2\pi]\to [0,2\pi]$ is a transformed angle. Hence the weak limit of $Z_n$ is a process $Z \in F([0, 2\pi])$ satisfying
$$
Z(t) = K(t)^\tsp \baZ = J(g(t))^\tsp \baZ,
$$
a time-transformed sinusoidal Gaussian process with covariance function $\rho(t,\tpr) = \cos [g(t)-g(\tpr)]$ and mean function
$$
\mu(t) = K(t)^\tsp \bamu = v I(t_0)^{1/2} K(t)^\tsp K(t_0) = v I(t_0)^{1/2} \cos [g(t)-g(t_0)],
$$
in agreement with $\mu(t;t_0,v)=\rho(t,t_0)A(t_0)^\tsp v$. Note in particular that $Z$ is a non-centered (centered) Gaussian process when $v>0$ ($v = 0$) respectively. Since $C=\De=[0,\infty)$, we get
$$
X(t) = \max(Z(t),0)^2 = \max[\baZ_1\cos(g(t)) + \baZ_2\sin(g(t)),0]^2.
$$
Thus, if the regularity conditions of Theorem \ref{thm:main_nonsing} hold for the polar coordinates representation, and since $g$ is a bijection on $[0,2\pi]$, we deduce 
\beq
\Lambda_n \Lto \Lambda = \sup_{0\le t \le 2\pi} X(t) = \baZ_1^2 + \baZ_2^2 \sim \chi^2_{2,\|\bamu \|_2^2},
\lb{LambdaPolar}
\eeq
in accordance with the Cartesian coordinates result \re{LambdaCart}. The last equality follows by picking $g(t)$ so that the Cauchy--Schwarz inequality holds with equality ($\vert K(t)^\tsp \baZ\vert \leq \Vert K(t)\Vert_2 \Vert \baZ\Vert_2$, with equality if $K(t) = \baZ / \Vert \baZ\Vert_2$, which is achievable since $g$ is a bijection).

Suppose instead that the Cartesian parameter space is $\{\eta; \, \eta_i \ge \eta_{i0} \mbox{ for }i=1,2\}$. The null parameter $\eta_0$ is then at the boundary of this parameter space, giving rise to a testing problem with $p=r=2$ and $q=s=0$. The neighbourhood of $\eta_0$ is a closed convex cone $\baC=[0,\infty)^2$ and the standardized convex cone is $\bar{\Delta}=\baA\baC$. The transformed, polar parameter space, however, is almost the same as before, except that the parameter set of the nuisance parameter changes to $\cT=[0,\pi/2]$. It follows, analogously to \re{LambdaCart} and \re{LambdaPolar}, that
\beq
 \Lambda = \| P_{\bar{\Delta}} \baZ \|_2^2 = \sup_{0\le t \le \pi/2} X(t).
\lb{LambdaPolar2}
\eeq
The distribution of $\Lambda$ is a mixture of truncated, noncentral $\chi^2$-distributions when $v>0$. In order to characterize this distribution we divide $\R^2$ into four closed convex cones. The first one is $\bar{\Delta}$, the closure of all non-zero vectors in $\R^2$ with polar coordinate angles between $g(0)$ and $g(\pi/2)$. The other three cones, $\bar{\Delta}^\prime$, $\bar{\Delta}^{\prime\prime}$ and  $\bar{\Delta}^{\prime\prime\prime}$, are defined analogously as the closures of the set of nonzero vectors in $\R^2$ with angles between $g(\pi/2)$ and $g(\pi/2)+\pi/2$, $g(0)-\pi/2$ and $g(0)$, and $g(\pi/2)+\pi/2$ and $g(0)-\pi/2$ respectively. Then 
\beq
\begin{array}{rcl}
\Lambda &\sim& P(\baZ\in \bar{\Delta}^{\prime\prime\prime})\chi_0^2 +
P(\baZ\in \bar{\Delta}^\prime) \cL( \max(J(g(\pi/2))^\tsp\baZ ,0)^2 | \baZ\in \bar{\Delta}^\prime) \\
&+& P(\baZ\in \bar{\Delta}^{\prime\prime}) \cL(\max(J(g(0))^\tsp\baZ,0)^2| \baZ\in \bar{\Delta}^{\prime\prime}) + P(\baZ\in\bar{\Delta})\cL(\|\baZ\|^2|\baZ\in \bar{\Delta}),
\end{array}
\lb{LambdaDistrPol}
\eeq
where $\cL(X)$ refers to the distribution of $X$. Under the null hypothesis $v=0$, \re{LambdaDistrPol} simplifies to a mixture
\beq
\Lambda \stackrel{v = 0}{\sim} (0.5-w_2)\chi^2_0 + 0.5\chi^2_1 + w_2\chi^2_2
\eeq
of three chisquare distributions. The weight $w_2$ of the $\chi^2_2$-distribution can be obtained from \re{wDef}. Indeed, since $\baC=[0,\infty)^2$ we can choose $\baV=\Id_2$ and $\baU=\baV\baA^{-1}=\baI^{-1/2}$. Letting $\bau_1$ and $\bau_2$ be the two rows of $\baU$ we find from \re{wDef} that
\beq
w_2 = \frac{\cos^{-1}(-\bau_1 \bau_2^\tsp/\| \bau_1\|_2 \|\bau_2\|_2)}{2\pi}
=  \frac{\cos^{-1}(-\baI_{12}^{-1}/\sqrt{\baI_{11}^{-1}\baI_{22}^{-1}})}{2\pi}
=  \frac{\cos^{-1}(\baI_{12}/\sqrt{\baI_{11}\baI_{22}})}{2\pi}.
\lb{w2Polar}
\eeq
Note that $2\pi w_2$ is the angle of formed the two boundary line segments of $\bar{\Delta}$, i.e.\ the difference in angle between $g(0)$ and $g(\pi/2)$. 

\subsubsection{Singular case}\lb{Sec:S}

Assume that the distribution of the data depends on $\eta$ only through $J(t_1)^\tsp(\eta - \eta_0)$ for $\eta$ near $\eta_0$, so that $\baI$ is singular with rank 1. If so,  $\baI$ has eigenvectors $J(t_1)$ and $J(t_2)$ with eigenvalues $I(t_1)$ and 0 respectively, for some $t_1$ and $t_2 = t_1 + \pi/2$ (assuming that addition of angles is modulo $2\pi$). This implies $\baI = I(t_1)J(t_1) J(t_1)^\tsp$ and $\baA = I(t_1)^{1/2}J(t_1) J(t_1)^\tsp$. Interpreting $\baA^{-1} = I(t_1)^{-1/2}J(t_1)J(t_1)^\tsp$ as a pseudo-inverse in \re{baZnPolar}, it follows that $\baZ_n\Lto \baZ = (\baZ_1,\baZ_2)^\tsp \sim N(\mu(t_1)J(t_1),J(t_1)J(t_1)^\tsp)$ as $n\to\infty$, where $\mu(t_1) = vI(t_1)^{1/2}\cos (t_1-t_0)$. Since $\eta_0$ is a point of singularity Corollary \ref{cor:no_nuisance} is violated and the covariance matrix of $\baZ$ is different from $\mbox{Id}_2$. In spite of this we will find that meaningful results are obtained after a switch to polar coordinates. 

If the parameter space for the Cartesian parameter $\eta$ is $\R^2$, under conditions analogous to those of Corollary \ref{cor:no_nuisance} we have 
\beq
\Lambda_n\Lto \La = \|\baZ \|_2^2 \sim \chi_{1,\mu^2(t_1)}^2
\lb{LambdaCartSing}
\eeq
as $n\to\infty$. In order to verify that the asymptotic distribution of LRT statistic is the same for polar coordinates \re{PolarDef}, we first notice that 
$$
Z_n(t) = I(t)^{-1/2} U_n(t) = \frac{\cos (t-t_0)}{|\cos(t-t_0)|} Z_n(t_1)
$$
for $t\in \cT_{ns} = [0,2\pi] \setminus \cT_s$, where $\cT_s = \{t_2,t_2+\pi\}$. We have $Z_n\Lto Z$ as $n\to\infty$, where
$$
Z(t) = \frac{\cos (t-t_1)}{|\cos(t-t_1)|} Z(t_1)
$$
for $\cT_{ns}$, and $Z(t_1)\sim N(\mu(t_1),1)$. This process has mean function 
$$
\mu(t) = E[Z(t)] = \frac{\cos(t-t_1)}{|\cos(t-t_1)|} \mu(t_1). 
$$ 
and covariance function 
$$
\rho(t,\tpr) = \Cov[Z(t),Z(\tpr)] = \frac{\cos(t-t_1)\cos(\tpr-t_1)}{|\cos(t-t_1)||\cos(\tpr-t_1)|}.
$$
It is possible to define $Z$ on $\cT$ as the unique right continuous extension from $\cT_{ns}$. Since $\Xi=C=\Delta=[0,\infty)$, it follows that $X(t)= (P_{\Delta} Z(t))^2= \max(Z(t),0)^2$. If the regularity conditions of Theorem \ref{thm:main_singular} hold for the polar coordinates representation, we deduce 
$$
\Lambda_n \Lto \Lambda = \sup_{t\in\cT} \max(Z(t),0)^2 = Z^2(t_1) \sim \chi_{1,\mu^2(t_1)}^2,
$$
in agreement with \re{LambdaCartSing}. 

Suppose instead that the original, untransformed parameter space is $\{\eta; \, \eta_i \ge \eta_{i0} \mbox{ for }i=1,2\}$. For polar coordinates we then have $\Xi=[0,\infty)$ and $\cT_{s} = [0,\pi/2]\cap \{t_2,t_2+\pi\}$. Since  $C=\Delta=[0,\infty)$, we conclude that the asymptotic LRT statistic is
$$
\Lambda = \sup_{t\in [0,\pi/2]} \max(Z(t),0)^2 = 
\left\{\begin{array}{ll}
\max [Z(t_1,0),0]^2; & 0\le t_1\le\pi/2,\\
\max [-Z(t_1,0),0]^2; & \pi\le t_1\le 3\pi/2,\\
Z^2(t_1); & \pi/2<t_1<\pi\mbox{ or }3\pi/2<t_1<2\pi.
\end{array}\right.
$$
This implies that the distribution of $\Lambda$ is given by
$$
\Lambda \sim
\left\{\begin{array}{ll}
P(Z(t_1)<0)\chi_0^2 + P(Z(t_1)>0)\cL (Z^2(t_1)|Z(t_1)>0) \stackrel{v=0}{=} 0.5\chi_0^2 + 0.5\chi_1^2; & 0\le t_1\le\pi/2,\\
P(Z(t_1)>0)\chi_0^2 + P(Z(t_1)<0)\cL (Z^2(t_1)|Z(t_1)<0) \stackrel{v=0}{=} 0.5\chi_0^2 + 0.5\chi_1^2; & \pi\le t_1\le 3\pi/2,\\
 \chi_{1,\mu^2(t_1)}^2  \stackrel{v=0}{=} \chi_1^2; & \mbox{otherwise}.
\end{array}\right.
$$

\section{Discussion}\lb{Sec:Disc}

The limiting distribution of the LRT statistic when a simple null hypothesis $H_0:\xi=\xi_0$ (or a
sequence of contiguous alternatives) of the first component of a parameter vector $\theta=(\xi,t)$ is tested, is the supremum $\La = \sup_{t\in\mathcal{T}} X(t) = \|X\|_{\mathcal{T}}$ of a (noncentral)
$\bachi^2$-process over $\cT$, the set of possible values of the nuisance
parameter $t$ that is unidentifiable under the null hypothesis, with $X(t) = \|P_{\Delta(t)}Z(t)\|_2^2$ the squared Euclidean norm of the projection $P_{\Delta(t)}Z(t)$ of a Gaussian process $Z$ down to a closed convex cone $\Delta(t)$ determined by the boundary constraints on $\xi$ under the null hypothesis. Classical results for boundary
inference without nuisance parameters
\citep{chernoff1954distribution,self1987asymptotic,geyer1994asymptotics},
testing whether a single parameter is zero or positive in the presence of
nuisance parameters \citep{da1977, da1987, risk2005}, and testing whether a
density is a mixture of two components or not \citep{chen2001large} are obtained
as special cases. Under certain conditions the limiting distribution has the same form whether or not the
information is singular at some isolated values of the nuisance parameter $t$.

For composite null hypotheses $H_0:\xi\in\Xi_0$, Theorem~\ref{thm:composite}
shows that the limiting distribution of the LRT statistic is
$\Lambda=\|X-X^0\|_{\mathcal{T}}$, where $X$ and $X^0$ are the (noncentral)
$\bachi^2$-processes corresponding to the full model and null hypothesis,
respectively. This implies that the identifiable components of the null hypothesis
cancel exactly. When the tangent cone of $\Xi_0$ is a linear subspace,
Theorem~\ref{thm:decomposition} further identifies the limiting distribution of
the LRT statistic as $\La=\sup_{t\in\mathcal{T}}\|P_{K(t)}Z(t)\|_2^2$, a
supremum of (noncentral) $\bar{\chi}^2$ variables $X(t) = \|P_{K(t)}Z(t)\|_2^2$
determined by a Gaussian process $Z$ projected down to convex closed cones $K(t)$, that are defined from the active boundary constraints. This is illustrated for the mixture problem of Example~\ref{ex:gauss_mix_2} (where it recovers the known-variance case of \citet{chen2003tests}) and the change point problem of Example \ref{Ex:CP}. 

In settings where a nuisance parameter is unidentifiable
under the null in one parametrization but not another, our theory implies that the limit distribution $\Lambda$ is the same regardless: Section \ref{sec:polar_ex} illustrates this in a setting where
reparameterizing in terms of polar coordinates introduces an unidentifiable
nuisance parameter, yet the limiting distribution of the LRT statistic is
unchanged. On the other hand, our asymptotic theory for LRTs is not valid for both parametrizations of the mixture model of Example \ref{ex:gauss_mix_1} and the change point model of Example \ref{Ex:CP}. In each of these two examples the Fisher information is singular for isolated values of the non-identifiable null parameter $t$, under both parametrizations, but the theory of Section \ref{sec:singular} only covers one of these two parametrizations. It is therefore of interest to extend the asymptotic LRT statistic theory to more scenarios with singular information matrices at isolated values of the non-identifiable null parameter $t$.    

Our framework includes the number $p$ of parameters to be tested, the number $r$
of boundary constraints on these parameters, as well as the dimensionalities $q$
and $s$ of the nuisance parameters that are not identifiable and identifiable
under the null hypothesis, respectively. Related results involving maxima of
(noncentral) $\bachi^2$-processes but without a general choice of $p$, $r$, $q$ and $s$ include
\citet{risk2005} ($p=r=1$, arbitrary $q,s$), \citet{anpl1995} ($r=0$, general
$p,q,s$), and \citet{an1996} ($q=0$, general $p,r,s$). The set $\cT$ of
possible values of the nuisance parameter vector that is unidentifiable under
the null hypothesis could be a functional space ($q=\infty$), although $\cT$
is a subset of a finite-dimensional Euclidean space ($q<\infty$) in all
examples of the present article. Our examples further have $r\le p\leq 2$; cases with $p\ge r >2$, such as
mixtures with more than two components or linkage analysis with two disease
genes, fit our framework but involve weights of the (noncentral)
$\bar{\chi}^2$-process that are often harder to characterize explicitly.

All of our examples are based on likelihoods for independent data. A majority of these examples are based on 
i.i.d.\ data, but we also included triangular arrays (Example \ref{Ex:CP}) and stratified sampling from different pedigree types (Section \ref{Sec:ManyFT}). However, our framework is much more general than
independent data. The most important requirement for our asymptotic results to hold is weak
convergence of the LRT statistic towards the supremum of a stochastic process,
regardless of whether data exhibit serial correlations or not. 

As mentioned in Section \ref{Sec:Intro}, there are many other applications of
the framework of the present article than mixture models, change point models and genetic linkage
analysis. For several of these applications, it would be interesting to
characterize the limiting (noncentral) $\bachi^2$-process explicitly.

A systematic treatment, including verification of the assumptions of Section
\ref{sec:asy_dist} for triangular-array models and characterization of the
(noncentral) $\bar{\chi}^2$-process limit when $\xi_0$ lies on the boundary of
$\Xi$, is another promising direction for future work. In a related preprint,
\citet{cox2022generalized} develops a generalized argmax theorem that handles
structural break estimation and boundary inference using Painlev\'e--Kuratowski
set convergence. These results may provide useful tools for the program of the
present article.

\section*{Acknowledgments}
Ola H\"ossjer was supported by the Swedish Research Council, contract number 621-2005-2810.
The authors wish to thank Annica Dominicus and Tobias Ryden for providing valuable references.

\bibliographystyle{plainnat}
\bibliography{lrt_nuis_bdry_2}

\newpage
\begin{appendix}
\section*{Appendix}
\label{app:proofs}

\setcounter{equation}{0}
\renewcommand{\thesection}{A}
\renewcommand{\theequation}{A.\arabic{equation}}
\renewcommand{\theHequation}{A.\arabic{equation}}

\subsection{Proofs}\label{app:proofs_sub}

\begin{proof}[Proof of Theorem \ref{thm:main_nonsing}]
    By Lemmas \ref{lem:Xtilde_Xn}, \ref{lem:tildeX_checkX}, and
    \ref{lem:checkX_X}, $\Vert  X_n - \tilde{X}_n\Vert_{\mathcal{T}} =
    o_{p^*}(1)$, $\Vert \tilde{X}_n - \check{X}_n\Vert_{\mathcal{T}} =
    o_{p^*}(1)$, and $\check{X}_n \Lto X$, respectively. Thus, by Slutsky's lemma
    \citep[Example 1.4.7]{vandervaart2023weak}, $(X_n - \tilde{X}_n, \tilde{X}_n
    - \check{X}_n, \check{X}_n) \Lto (0, 0, X)$, and hence by continuous mapping
    theorem \citep[Theorem 1.3.6]{vandervaart2023weak}, $X_n = (X_n -
    \tilde{X}_n) + (\tilde{X}_n - \check{X}_n) + \check{X}_n \Lto 0 + 0 + X =
    X$. Note the requirement that $X$ be separable holds because $X$ is tight and
    $F(\mathcal{T})$ a metric space \citep[p.15]{vandervaart2023weak}.
\end{proof}

For any closed $A, A' \subseteq \R^{p}$ and $\bar{B}_M = \{z \in \R^p: \Vert
z\Vert_2 \leq M\}$, the closed ball centered at the origin with radius $M > 0$,
define
\begin{equation*}
    d_M(A,A') = \max\left(\sup_{z\in A\cap \bar{B}_M} \Vert z- P_{A'}(z)\Vert_2,
\sup_{z' \in A'\cap \bar{B}_M} \Vert z' - P_{A}(z')\Vert_2\right)
\end{equation*}
and
\begin{equation}
d(A,A') = \sum_{M=1}^\infty 2^{-M}\min\{d_M(A,A'),1\}.
\lb{dDef}
\end{equation}
Then $d$ defines a metric between closed sets. See for example
\citet{geyer1994asymptotics} for the connection between Assumption
\ref{assm:cone} and Chernoff regularity \citep{chernoff1954distribution}.

\begin{lemma} \label{lem:set_dist}
Under Assumptions \ref{assm:inf_bound}, \ref{assm:closed}, and \ref{assm:cone},
\begin{equation*}
\lim_{n\to\infty} \sup_{t\in\mathcal{T}} d(\Delta_n(t),\Delta(t)) = 0.
\end{equation*}
Moreover, for any $M>0$,
\begin{equation*}
\lim_{n\to\infty}\sup_{t\in\mathcal{T},\Vert z\Vert_2\leq M}
\left| \Vert z -  P_{\Delta_n(t)}z\Vert_2^2 - \Vert z -  P_{\Delta(t)} z \Vert_2^2 \right| = 0
\end{equation*}
and
\begin{equation*}
\lim_{n\to\infty}\sup_{t\in\mathcal{T},\Vert z\Vert_2\leq M}
\left| \Vert P_{\Delta_n(t)}z\Vert_2^2 - \Vert P_{\Delta(t)} z \Vert_2^2 \right| = 0.
\end{equation*}

\end{lemma}
\begin{proof}
    For the first claim it suffices to prove that for any $M>0$,
    \begin{equation*}
        \lim_{n\to\infty} \sup_{t\in\mathcal{T}} d_M(\Delta_n(t),\Delta(t)) = 0.
    \end{equation*}
    Let $C_s(t) = \{\Xi(t) - \xi_0\}/s$ and pick $z \in \Delta_n(t)\cap \bar{B}_M$ and $z' \in \Delta(t)\cap \bar{B}_M$. Then, by definition of these sets, $z = A(t)^\tsp y$ for a $y \in C_{n^{-1/2}}(t)$ and $z' = A(t)^\tsp y'$ for a $y' \in C(t)$. Thus, $\Vert y\Vert_2, \Vert y'\Vert_2 \leq M/\underline{\kappa}^{1/2}$, where $\underline{\kappa} > 0$ is from Assumption \ref{assm:inf_bound}. Moreover,
    \begin{align*}
        \Vert z - z' \Vert_2^2 = (y - y')^\tsp A(t) A(t)^\tsp (y - y')
         = (y - y')^\tsp I(t) (y - y')
        \leq \bar{\kappa}\Vert y - y'\Vert_2^2,
    \end{align*}
    where $\bar{\kappa} < \infty$ is also from Assumption \ref{assm:inf_bound}.
    Thus, since $y \in C_{n^{-1/2}}(t)\cap B(0, M/\underline{\kappa}^{1/2})$
    and $y' \in C(t) \cap B(0, M/\underline{\kappa}^{1/2})$,
    \begin{align*}
        \sup_{t\in\mathcal{T}} d_M\{\Delta_n(t),\Delta(t)\} &\leq
            \sqrt{\bar{\kappa}} \sup_{t\in\mathcal{T}}
                d_{M/\underline{\kappa}^{1/2}}\{C_{n^{-1/2}}(t), C(t)\},
    \end{align*}
    which tends to zero by Assumption \ref{assm:cone}.

    To prove the second claim, pick arbitrary $z\in \bar{B}_M$ and $t\in\mathcal{T}$. We have $\Vert P_{\Delta_n(t)}z\Vert_2\leq M$ and
    $\Vert P_{\Delta(t)}z\Vert_2\leq M$, by definition of the projection since $0\in \Delta_n(t)$ and $0\in \Delta(t)$. Thus,
    \begin{align*}
        \left| \Vert z-P_{\Delta_n(t)}z\Vert_2^2 - \Vert z-P_{\Delta(t)}z\Vert_2^2 \right|
        &\leq 2M \left| \Vert z-P_{\Delta_n(t)}z\Vert_2 - \Vert z-P_{\Delta(t)}z\Vert_2 \right| \\
        &\leq 2M d_M\{\Delta_n(t),\Delta(t)\} \\
        &\leq 2M \sup_{t\in\mathcal{T}} d_M\{\Delta_n(t),\Delta(t)\},
    \end{align*}
    which tends to zero by the first claim.

    For the third claim, fix $z \in \bar{B}_M$ and $t\in\mathcal{T}$, and write
    $\epsilon = d_M\{\Delta_n(t),\Delta(t)\}$, $a = P_{\Delta_n(t)}z$, and
    $b = P_{\Delta(t)}z$. Because $\Delta_n(t)$ and $\Delta(t)$ are closed,
    convex, and contain the origin, the projections onto them are nonexpansive
    and map $0$ to $0$; hence $\Vert a\Vert_2 \leq \Vert z\Vert_2 \leq M$ and
    $\Vert b\Vert_2 \leq M$, so $a \in \Delta_n(t)\cap\bar{B}_M$ and
    $b \in \Delta(t)\cap \bar{B}_M$. Thus, by the definition of $d_M$, the points
    $a' = P_{\Delta_n(t)}b \in \Delta_n(t)$ and $b' = P_{\Delta(t)}a \in \Delta(t)$
    satisfy $\Vert b - a'\Vert_2 \leq \epsilon$ and $\Vert a - b'\Vert_2 \leq \epsilon$.
    By the variational characterization of the projection onto a closed convex
    set, $(z-a)^\tsp(x - a) \leq 0$ for every $x\in\Delta_n(t)$ and
    $(z-b)^\tsp(y-b)\leq 0$ for every $y \in \Delta(t)$. Taking $x = a'$ and
    $y = b'$, adding the two inequalities, and writing $a' = b + (a'-b)$ and
    $b' = a + (b'-a)$,
    \begin{align*}
        \Vert a - b\Vert_2^2 \leq (z-a)^\tsp(b - a') + (z-b)^\tsp(a - b')
        \leq \epsilon\left(\Vert z-a\Vert_2 + \Vert z-b\Vert_2\right) \leq 2M\epsilon,
    \end{align*}
    where the last inequality holds because $0 \in \Delta_n(t)$ and
    $0 \in \Delta(t)$ give $\Vert z - a\Vert_2 \leq \Vert z\Vert_2\leq M$ and
    $\Vert z - b\Vert_2 \leq M$. Consequently,
    \begin{align*}
        \left|\Vert P_{\Delta_n(t)}z\Vert_2^2 - \Vert P_{\Delta(t)}z\Vert_2^2\right|
        = \left|\Vert a\Vert_2 - \Vert b\Vert_2\right|\left(\Vert a\Vert_2 + \Vert b\Vert_2\right)
        \leq 2M \Vert a - b\Vert_2 \leq 2M(2M\epsilon)^{1/2}.
    \end{align*}
    Taking the supremum over $z \in \bar{B}_M$ and $t\in\mathcal{T}$, the right
    side tends to zero by the first claim.
\end{proof}

\begin{proof}[Proof of Theorem \ref{thm:decomposition}]
Since $K(t)$ is the intersection of the closed convex cone $\Delta(t)$ with
the closed linear subspace $\{\Delta^0(t)\}^\perp$, it is a closed convex cone
contained in $\{\Delta^0(t)\}^\perp$.

For the direct sum, take any $\delta\in\Delta(t)$ and write $\delta = \delta_0 +
\kappa$ with $\delta_0 = P_{\Delta^0(t)}\delta\in\Delta^0(t)$ and $\kappa =
P_{\{\Delta^0(t)\}^\perp}\delta$. Since $\Delta(t)$ is a convex cone containing
the subspace $\Delta^0(t)$, it is closed under addition of its elements, and
$-\delta_0\in\Delta^0(t)\subseteq\Delta(t)$, so $\kappa = \delta +
(-\delta_0)\in\Delta(t)$. Thus $\kappa\in K(t)$, giving
$\Delta(t)\subseteq\Delta^0(t)+K(t)$. The reverse inclusion follows because
$\Delta^0(t)\subseteq\Delta(t)$ and $K(t)\subseteq\Delta(t)$. Orthogonality and
uniqueness hold by construction, so $\Delta(t) = \Delta^0(t)\oplus K(t)$.

For the projection identity, write $z = z_0 + z_\perp$ with
$z_0 = P_{\Delta^0(t)}z$ and $z_\perp = P_{\{\Delta^0(t)\}^\perp}z$. Any
$\delta\in\Delta(t)$ decomposes as $\delta_0+\kappa$ with
$\delta_0\in\Delta^0(t)$, $\kappa\in K(t)$, $\delta_0\perp\kappa$, so
$\|z-\delta\|_2^2 = \|z_0-\delta_0\|_2^2 + \|z_\perp-\kappa\|_2^2$.
The two terms are minimized independently: the first at
$\delta_0 = z_0$ (since $\Delta^0(t)$ is a subspace) and the second at
$\kappa = P_{K(t)}z_\perp$, giving $P_{\Delta(t)}z = z_0 + P_{K(t)}z_\perp$.
By the Pythagorean theorem,
$\|P_{\Delta(t)}z\|_2^2 - \|P_{\Delta^0(t)}z\|_2^2 = \|P_{K(t)}z_\perp\|_2^2$.
Finally, since $K(t)\subseteq\{\Delta^0(t)\}^\perp$, for any $z = z_0 + z_\perp$
the component $z_0\in\Delta^0(t)$ is irrelevant to $P_{K(t)}$, so
$P_{K(t)}z_\perp = P_{K(t)}z$, giving the stated identity.
Substituting $z = Z(t)$ gives $X(t) - X^0(t) = \|P_{K(t)}Z(t)\|_2^2$ for each $t$, and taking the supremum yields the result.
\end{proof}

\begin{proof}[Proof of Lemma \ref{lem:rate}]
For an arbitrary $c_2 > 0$ and $c_1 = \underline{\kappa}/2$ pick a $c_3$ small
enough that \eqref{eq:equicont} holds. Let us denote $\bar{B} =
\bar{B}_{c_3}(\xi_0; t) = \{\xi \in \Xi(t) : \Vert \xi - \xi_0\Vert_2 \leq
c_3\}$ for simplicity. Then for all large enough $n$, with outer probability at least $1 - c_2$, for $\xi \in \bar{B}$,
\begin{align*}
    \ell_n(\xi, t) - \ell_n(\xi_0,t) &\leq 
        (\xi - \xi_0)^\tsp n^{1/2}A(t)Z_n(t) - \frac{n \underline{\kappa}}{4}\Vert \xi - \xi_0\Vert_2^2 \\
        &\leq n^{1/2}\bar{\kappa}^{1/2}\Vert \xi - \xi_0\Vert_2 \Vert Z_n\Vert_{\mathcal{T}} 
        - \frac{n \underline{\kappa}}{4}\Vert \xi - \xi_0\Vert_2^2.
\end{align*}
Here, $\underline{\kappa}$ and $\bar{\kappa}$ are from Assumption \ref{assm:inf_bound}. Let $T_n(\Vert \xi - \xi_0\Vert_2)$ denote the quadratic in the last right-hand
side. For any $C_1 > 0$, note $\Vert\hat{\xi}_n(t) - \xi_0\Vert_2 >
\frac{C_1}{\sqrt{n}}$ can only happen if at least one of the following holds:
(i) $\hat{\xi}_n(t) \notin \bar{B}$; (ii) the upper bound in the last display
fails for some $\xi \in \bar{B}$; or (iii) (i) and (ii) do not hold, in which
case taking $\xi = \hat{\xi}_n(t)$ in the upper bound and using that, by definition, $\ell_n(\hat{\xi}_n(t), t) \geq \ell_n(\xi_0,t) - O_{p^*}(1)$ yields $C_1/\sqrt{n} < \Vert \hat{\xi}_n(t) - \xi_0\Vert_2 \leq c_3$ and
\[
    -O_{p^*}(1) \leq \ell_n(\hat{\xi}_n(t), t) - \ell_n(\xi_0,t)  \leq 
    T_n(\Vert \hat{\xi}_n(t) - \xi_0\Vert_2) \leq \sup_{\Vert \xi - \xi_0\Vert_2 > 
        C_1/\sqrt{n}} T_n(\Vert \xi - \xi_0\Vert_2).
\]
Now (i) has outer probability tending to zero by Assumption
\ref{assm:consistency} and (ii) has outer probability at most $c_2$ for all
large enough $n$. Thus, it suffices to show that the outer probability of (iii)
can be made arbitrarily small asymptotically by choosing $C_1$ large enough.

To show this, note, for every $C_2 > 0$,
\begin{align*}
    \pr^*\left\{\sup_{\Vert \xi - \xi_0\Vert_2 > C_1/\sqrt{n}}T_n(\Vert \xi - \xi_0\Vert_2) \geq -O_{p^*}(1)\right\} &\leq 
    \pr^*\left\{\sup_{\Vert \xi - \xi_0\Vert_2 > C_1/\sqrt{n}}T_n(\Vert \xi - \xi_0\Vert_2) \geq -C_2/2\right\} \\
    &\quad
    + \pr^*\left\{O_{p^*}(1) \geq C_2/2\right\},
\end{align*}
and the last term can be made arbitrarily small asymptotically by choosing $C_2$
large enough. To control the other term, by the quadratic formula,
$T_n(\Vert \xi - \xi_0\Vert_2)\geq -C_2/2$ can only hold if
\[
    \sqrt{n}\Vert \xi - \xi_0\Vert_2 \leq \frac{\bar{\kappa}^{1/2} 
    \Vert Z_n\Vert_{\mathcal{T}} + \sqrt{\bar{\kappa}\Vert Z_n\Vert_{\mathcal{T}}^2 + C_2\underline{\kappa}/2}}{\underline{\kappa}/2}.
\]
Thus, $\sup_{\Vert \xi - \xi_0\Vert_2 > C_1 / \sqrt{n}} nT_n(\Vert \xi -
\xi_0\Vert_2)\geq -C_2/2$ can only happen if the right-hand side in the last
display is greater than $C_1$. But that right-hand-side is asymptotically tight by the continuous mapping theorem \citep[Theorem 1.3.6]{vandervaart2023weak}
since $\Vert Z_n\Vert_{\mathcal{T}}$ is by Assumption \ref{assm:zn_conv}
\citep[Definition 1.3.7 and Lemma 1.3.8]{vandervaart2023weak}. Thus, we can
pick $C_1$ large enough to make the outer probability of that event arbitrarily
small asymptotically. This completes the proof.
\end{proof}

\begin{lemma} \label{lem:rate_singular}
    Under Assumptions \ref{assm:Rtilde_bound} and \ref{assm:sing_Zn}, every $\hat{\delta}_n(t)$ such that $\sup_{t \in \mathcal{T}_{ns}} \Vert \hat{\delta}_n(t)\Vert_2 = o_{\pr^*}(\sqrt{n})$ and $Q_n(\hat{\delta}_n(t), t) + 2\Vert \hat{\delta}_n(t)\Vert_2^2\tilde{R}_n(\hat{\delta}_n(t), t) \geq \sup_{\delta \in \Delta_n(t)} \{Q_n(\delta, t) + 2\Vert \delta\Vert_2^2\tilde{R}_n(\delta, t)\} - O_{\pr^*}(1)$ satisfies
    \[
        \sup_{t \in \mathcal{T}_{ns}} \Vert \hat{\delta}_n(t)\Vert_2 = O_{\pr^*}(1).
    \]
\end{lemma}

\begin{proof}[Proof of Lemma \ref{lem:rate_singular}]
Fix an arbitrary $c_2 > 0$ and choose $c_1 \in (0, 1/2)$ (e.g., $c_1 = 1/4$). Pick a $c_3$ small enough that Assumption \ref{assm:Rtilde_bound} holds. Let us denote $\tilde{B} = \tilde{B}(n, t, c_3)$ for simplicity. Then for all large enough $n$, with outer probability at least $1 - c_2$, for $\delta \in \tilde{B}$,
\begin{align*}
    Q_n(\delta, t) + 2\Vert \delta\Vert_2^2\tilde{R}_n(\delta, t) &\leq Q_n(\delta, t) + 2c_1\Vert \delta\Vert_2^2 \\
        &\leq 2\Vert \delta\Vert_2 \Vert Z_n\Vert_{\mathcal{T}_{ns}} - (1-2c_1)\Vert \delta\Vert_2^2.
\end{align*}
Let $T_n(\Vert \delta\Vert_2)$ denote the quadratic in the last right-hand side. For any $C_1 > 0$, note $\sup_{t \in \mathcal{T}_{ns}} \Vert\hat{\delta}_n(t)\Vert_2 > C_1$ can only happen if either (i) there exists a $t \in \mathcal{T}_{ns}$ such that $\hat{\delta}_n(t) \notin \tilde{B}$; (ii) the upper bound in the last display does not hold on $\tilde{B}$; or (iii) the bound holds, but for some $t \in \mathcal{T}_{ns}$ there exists a $\delta$ such that $C_1 < \Vert \delta\Vert_2 \leq \sqrt{n}c_3 $ and
\[
    -O_{p^*}(1) \leq Q_n(\delta, t) + 2\Vert \delta\Vert_2^2\tilde{R}_n(\delta, t) \leq T_n(\Vert \delta\Vert_2) \leq \sup_{\Vert \delta\Vert_2 > C_1} T_n(\Vert \delta\Vert_2).
\]
The last claim uses the definition of $\hat{\delta}_n(t)$ and the upper bound in the preceding display. Now (i) has outer probability tending to zero by the assumption that $\sup_{t \in \mathcal{T}_{ns}} \Vert \hat{\delta}_n(t)\Vert_2 = o_{\pr^*}(\sqrt{n})$ and (ii) has outer probability at most $c_2$ for all large enough $n$. Thus, it suffices to show that the outer probability of (iii) can be made arbitrarily small asymptotically by choosing $C_1$ large enough.

To show this, note, for every $C_2 > 0$,
\begin{align*}
    \pr^*\left\{\sup_{\Vert \delta\Vert_2 > C_1}T_n(\Vert \delta\Vert_2) \geq -O_{p^*}(1)\right\} &\leq 
    \pr^*\left\{\sup_{\Vert \delta\Vert_2 > C_1}T_n(\Vert \delta\Vert_2) \geq -C_2/2\right\} \\
    &\quad
    + \pr^*\left\{O_{p^*}(1) \geq C_2/2\right\},
\end{align*}
and the last term can be made arbitrarily small asymptotically by choosing $C_2$ large enough. To control the other term, by the quadratic formula, $T_n(\Vert \delta\Vert_2)\geq -C_2/2$ can only hold if
\[
    \Vert \delta\Vert_2 \leq \Vert Z_n\Vert_{\mathcal{T}_{ns}} + \sqrt{\Vert Z_n\Vert_{\mathcal{T}_{ns}}^2 + 2c_1 + C_2/2}.
\]
Thus, $\sup_{\Vert \delta\Vert_2 > C_1} T_n(\Vert \delta\Vert_2)\geq -C_2/2$ can only happen if the right-hand side in the last display is greater than $C_1$. But that right-hand-side is asymptotically tight since $\Vert Z_n\Vert_{\mathcal{T}_{ns}} = O_{\pr^*}(1)$ by Assumption \ref{assm:sing_Zn}. Thus, we can pick $C_1$ large enough to make the outer probability of that event arbitrarily small asymptotically. This completes the proof.
\end{proof}

\begin{lemma} \label{lem:Xtilde_Xn_singular}
    Under Assumptions \ref{assm:Rtilde_bound} and \ref{assm:sing_Zn}, if $\sup_{t \in \mathcal{T}_{ns}} \Vert \hat{\delta}_n(t)\Vert_2 = o_{\pr^*}(\sqrt{n})$, then $\sup_{t \in \mathcal{T}_{ns}} |X_n(t) - \tilde{X}_n(t)| = o_{\pr^*}(1)$.
\end{lemma}

\begin{proof}[Proof of Lemma \ref{lem:Xtilde_Xn_singular}]
    Pick $\hat{\delta}_n(t)$ such that $Q_n(\hat{\delta}_n(t), t) + 2\Vert \hat{\delta}_n(t)\Vert_2^2\tilde{R}_n(\hat{\delta}_n(t), t) \geq \sup_{\delta \in \Delta_n(t)} \{Q_n(\delta, t) + 2\Vert \delta\Vert_2^2\tilde{R}_n(\delta, t)\} - 1/n$. Assuming $\sup_{t \in \mathcal{T}_{ns}} \Vert \hat{\delta}_n(t)\Vert_2 = o_{\pr^*}(\sqrt{n})$, Lemma \ref{lem:rate_singular} gives $\sup_{t \in \mathcal{T}_{ns}} \Vert \hat{\delta}_n(t)\Vert_2 = O_{\pr^*}(1)$. Pick a $\tilde{\delta}_n(t)$ such that $Q_n(\tilde{\delta}_n(t), t) \geq \sup_{\delta \in \Delta_n(t)} Q_n(\delta, t) - 1/n$. Since $0 \in \Delta_n(t)$ and $Q_n(0, t) = 0$, the supremum is non-negative, so $Q_n(\tilde{\delta}_n(t), t) \geq -1/n$. But for any $\delta$ with $\Vert \delta\Vert_2 > 2\Vert Z_n(t)\Vert_2$, by Cauchy--Schwarz,
    \begin{align} \label{eq:large_delta_negative_singular}
       Q_n(\delta, t) \leq 2\Vert \delta\Vert_2 \Vert Z_n(t)\Vert_2 - \Vert \delta\Vert_2^2 = \Vert \delta\Vert_2(2\Vert Z_n(t)\Vert_2 - \Vert \delta\Vert_2) < 0.
    \end{align}
    Thus, if $\Vert Z_n\Vert_{\mathcal{T}_{ns}} \leq C_1$, any $\Vert \delta\Vert_2 > 3C_1$ leads to a value less than $-3C_1$, which is less than $-1/n$ for all large enough $n$. Thus, for such $n$, $\sup_{t \in \mathcal{T}_{ns}} \Vert \tilde{\delta}_n(t)\Vert_2 \leq 3C_1$ whenever $\Vert Z_n\Vert_{\mathcal{T}_{ns}} \leq C_1$. Since $\Vert Z_n\Vert_{\mathcal{T}_{ns}} = O_{\pr^*}(1)$ by Assumption \ref{assm:sing_Zn}, we have $\sup_{t \in \mathcal{T}_{ns}} \Vert \tilde{\delta}_n(t)\Vert_2 = O_{\pr^*}(1)$.

    Next, for any $c_1, C_1 > 0$, by countable sub-additivity of outer probabilities,
    \begin{equation*}
    \begin{aligned}
        \pr^*\left(\sup_{t \in \mathcal{T}_{ns}} |X_n(t) - \tilde{X}_n(t)| > c_1\right) &\leq \pr^*\left(\sup_{t \in \mathcal{T}_{ns}} |X_n(t) - \tilde{X}_n(t)| > c_1, \sup_{t \in \mathcal{T}_{ns}} \Vert \hat{\delta}_n(t)\Vert_2 \leq C_1, \sup_{t \in \mathcal{T}_{ns}} \Vert \tilde{\delta}_n(t)\Vert_2 \leq C_1\right) \\
        &\quad +\pr^*\left(\sup_{t \in \mathcal{T}_{ns}} \Vert \hat{\delta}_n(t)\Vert_2 > C_1\right) + \pr^*\left(\sup_{t \in \mathcal{T}_{ns}} \Vert \tilde{\delta}_n(t)\Vert_2 > C_1\right).
    \end{aligned}
    \end{equation*}
    The last two terms can be made arbitrarily small by choosing $C_1$ large enough, so let us deal with the first term. Define $\tilde{\Upsilon}_{C_1,n} = \{(\delta, t): t \in\mathcal{T}_{ns}, \delta\in\Delta_n(t), \Vert\delta\Vert_2\leq C_1\}$. Then, using that the supremum of the difference is no greater than the difference of suprema, for outcomes where $\sup_{t \in \mathcal{T}_{ns}} \Vert\hat{\delta}_n(t)\Vert_2\leq C_1$ and $\sup_{t \in \mathcal{T}_{ns}} \Vert \tilde{\delta}_n(t)\Vert_2 \leq C_1$,
    \begin{align*}
        \sup_{t \in \mathcal{T}_{ns}} |X_n(t) - \tilde{X}_n(t)| & = \sup_{t\in\mathcal{T}_{ns}} \left| \sup_{\delta \in \Delta_n(t)} \{Q_n(\delta, t) + 2\Vert \delta\Vert_2^2\tilde{R}_n(\delta, t)\} - \sup_{\delta \in \Delta_n(t)} Q_n(\delta, t) \right| \\
        &\leq 2/n + \sup_{t\in\mathcal{T}_{ns}} \left| Q_n(\hat{\delta}_n(t), t) + 2\Vert \hat{\delta}_n(t)\Vert_2^2\tilde{R}_n(\hat{\delta}_n(t), t) - Q_n(\hat{\delta}_n(t), t) \right| \\
        &\leq 4/n + \sup_{t\in\mathcal{T}_{ns}} \left| \sup_{\delta \in \Delta_n(t), \Vert \delta\Vert_2 \leq C_1} \{Q_n(\delta, t) + 2\Vert \delta\Vert_2^2\tilde{R}_n(\delta, t)\} - \sup_{\delta \in \Delta_n(t), \Vert \delta\Vert_2 \leq C_1} Q_n(\delta, t) \right| \\
        &\leq  4/n + \sup_{(\delta, t) \in \tilde{\Upsilon}_{C_1,n} } 2\Vert \delta\Vert_2^2|\tilde{R}_n(\delta, t)|.
    \end{align*}
    By Assumption \ref{assm:Rtilde_bound}, for any $c_1 > 0$, the outer probability that $\sup_{(\delta, t) \in \tilde{\Upsilon}_{C_1,n}} |\tilde{R}_n(\delta, t)| > c_1$ can be made arbitrarily small for all large enough $n$ by choosing $c_3$ small enough, because for any fixed $C_1$, $C_1 \leq \sqrt{n}c_3$ for all large enough $n$, so the supremum is taken over a subset of $\tilde{B}(n, t, c_3)$. For outcomes where $\sup_{t\in\mathcal{T}_{ns}}\Vert\hat{\delta}_n(t)\Vert_2\leq C_1$ and $\sup_{t\in\mathcal{T}_{ns}}\Vert\tilde{\delta}_n(t)\Vert_2\leq C_1$, we have $\Vert\delta\Vert_2\leq C_1$ for $(\delta,t)\in\tilde{\Upsilon}_{C_1,n}$, so the last term is bounded by $2C_1^2\sup_{(\delta,t)\in\tilde{\Upsilon}_{C_1,n}}|\tilde{R}_n(\delta,t)|$ and hence is $o_{\pr^*}(1)$, which finishes the proof.
\end{proof}

\begin{lemma} \label{lem:tildeX_checkX_singular}
    Under Assumptions \ref{assm:sing_inf}, \ref{assm:sing_Zn}, \ref{assm:closed}, and \ref{assm:cone}, $\sup_{t \in \mathcal{T}_\epsilon} |\tilde{X}_n(t) - \check{X}_n(t)| = o_{\pr^*}(1)$.
\end{lemma}

\begin{proof}[Proof of Lemma \ref{lem:tildeX_checkX_singular}]
    By the same arguments as in the proof of Lemma \ref{lem:set_dist}, replacing $\mathcal{T}$ with $\mathcal{T}_\epsilon$ and using Assumption \ref{assm:sing_inf} in place of Assumption \ref{assm:inf_bound}, we have $\lim_{n\to\infty} \sup_{t\in\mathcal{T}_\epsilon} d(\Delta_n(t),\Delta(t)) = 0$. Consequently, for any $C_1 > 0$,
    \begin{equation*}
        \lim_{n\to\infty}\sup_{t\in\mathcal{T}_\epsilon,\Vert z\Vert_2\leq C_1} \left| \Vert z - P_{\Delta_n(t)}z\Vert_2^2 - \Vert z - P_{\Delta(t)} z \Vert_2^2 \right| = 0.
    \end{equation*}
    As in the proof of Lemma \ref{lem:tildeX_checkX}, we have $\tilde{X}_n(t) = \Vert Z_n(t)\Vert_2^2 - \Vert Z_n(t) - P_{\Delta_n(t)}Z_n(t)\Vert_2^2$ and $\check{X}_n(t) = \Vert Z_n(t)\Vert_2^2 - \Vert Z_n(t) - P_{\Delta(t)}Z_n(t)\Vert_2^2$. Thus, for outcomes where $\Vert Z_n\Vert_{\mathcal{T}_\epsilon} \leq C_1$,
    \begin{align*}
        \sup_{t \in \mathcal{T}_\epsilon}|\tilde{X}_n(t) - \check{X}_n(t)| &= \sup_{t\in \mathcal{T}_\epsilon}\left|\Vert Z_n(t) - P_{\Delta(t)}Z_n(t)\Vert_2^2 - \Vert Z_n(t) - P_{\Delta_n(t)}Z_n(t)\Vert_2^2\right| \\
        & \leq \sup_{t\in\mathcal{T}_\epsilon, \Vert z\Vert_2 \leq C_1} \left| \Vert z - P_{\Delta_n(t)}z\Vert_2^2 - \Vert z - P_{\Delta(t)}z\Vert_2^2 \right|,
    \end{align*}
    which tends to zero. Since $\Vert Z_n\Vert_{\mathcal{T}_\epsilon} = O_{\pr^*}(1)$ by Assumption \ref{assm:sing_Zn} (noting $\mathcal{T}_\epsilon \subseteq \mathcal{T}_{ns}$), the outer probability of $\Vert Z_n\Vert_{\mathcal{T}_\epsilon} > C_1$ can be made arbitrarily small by choosing $C_1$ large enough, completing the proof.
\end{proof}

\begin{lemma}\lb{lem:checkX_X_singular}
    Under Assumptions \ref{assm:sing_Zn}, \ref{assm:closed}, and \ref{assm:cone}, $\check{X}_n \Lto X$ in $F(\mathcal{T}_{ns})$, where $X(t) = \sup_{\delta \in \Delta(t)} (2\delta^\tsp Z(t) - \Vert \delta\Vert_2^2)$.
\end{lemma}

\begin{proof}[Proof of Lemma \ref{lem:checkX_X_singular}]
    Define the map $g: F^p(\mathcal{T}_{ns}) \to F(\mathcal{T}_{ns})$ by $g(z)(t) = \sup_{\delta \in \Delta(t)} (2\delta^\tsp z(t) - \Vert \delta\Vert_2^2)$. For any $z_1, z_2 \in F^p(\mathcal{T}_{ns})$, let $\hat{\delta}_i(t)$ be the projection of $z_i(t)$ onto $\Delta(t)$ for $i=1,2$. Since $2\delta^\tsp z - \Vert \delta\Vert_2^2 = \Vert z \Vert_2^2 - \Vert z - \delta\Vert_2^2$, maximizing over $\delta \in \Delta(t)$ is the same as minimizing $\Vert z_i(t) - \delta\Vert_2$, so the supremum defining $g(z_i)(t)$ is attained at $\hat{\delta}_i(t)$. Moreover, since $0 \in \Delta(t)$, $\Vert \hat{\delta}_i(t)\Vert_2 \leq \Vert z_i(t)\Vert_2 \leq \Vert z_i\Vert_{\mathcal{T}_{ns}}$. Because $\hat{\delta}_2(t)$ is optimal for $z_2(t)$, replacing it by $\hat{\delta}_1(t)$ can only decrease $g(z_2)(t)$, so
    \begin{align*}
        g(z_1)(t) - g(z_2)(t) &\leq \{2\hat{\delta}_1(t)^\tsp z_1(t) - \Vert \hat{\delta}_1(t)\Vert_2^2\} - \{2\hat{\delta}_1(t)^\tsp z_2(t) - \Vert \hat{\delta}_1(t)\Vert_2^2\} \\
        &= 2\hat{\delta}_1(t)^\tsp \{z_1(t) - z_2(t)\},
    \end{align*}
    and the same bound holds with the roles of $z_1$ and $z_2$ interchanged. Thus, by the Cauchy--Schwarz inequality,
    \begin{align*}
        |g(z_1)(t) - g(z_2)(t)| &\leq 2\max(\Vert \hat{\delta}_1(t)\Vert_2, \Vert \hat{\delta}_2(t)\Vert_2) \Vert z_1(t) - z_2(t)\Vert_2 \\
        &\leq 2\max(\Vert z_1\Vert_{\mathcal{T}_{ns}}, \Vert z_2\Vert_{\mathcal{T}_{ns}}) \Vert z_1 - z_2\Vert_{\mathcal{T}_{ns}}.
    \end{align*}
    Taking the supremum over $t \in \mathcal{T}_{ns}$, we see that $g$ is Lipschitz on bounded subsets of $F^p(\mathcal{T}_{ns})$, and hence continuous. Since $Z$ is tight, it concentrates on a separable subspace, and the continuous mapping theorem applies. Thus, $\check{X}_n = g(Z_n) \Lto g(Z) = X$ in $F(\mathcal{T}_{ns})$.
\end{proof}

\begin{proof}[Proof of Corollary \ref{cor:main_singular_fullT}]
For a fixed $n$, define
\[
    g_n(\xi,t) = 2\{\ell_n(\xi,t) - \ell_n(\xi_0,t)\},\qquad X_n(t) = \sup_{\xi\in\Xi(t)} g_n(\xi,t).
\]
We first show that $t\mapsto X_n(t)$ is lower semicontinuous on $\mathcal{T}$.
Let $t_m\to t$ in $\mathcal{T}$ and fix $\eta>0$. Choose $\xi\in\Xi(t)$ such
that $g_n(\xi,t) \geq X_n(t) - \eta$. By inner semicontinuity of $\Xi(\cdot)$,
there exist $\xi_m\in\Xi(t_m)$ such that $\xi_m\to\xi$. Since $\ell_n$ is lower
semicontinuous, $\liminf_{m\to\infty} \ell_n(\xi_m,t_m) \geq \ell_n(\xi,t)$,
and therefore
\[
    \liminf_{m\to\infty} X_n(t_m)
    \geq \liminf_{m\to\infty} g_n(\xi_m,t_m)
    \geq g_n(\xi,t)
    \geq X_n(t) - \eta.
\]
Letting $\eta\downarrow 0$ proves lower semicontinuity of $X_n$.

By Assumption \ref{assm:T_metric}, $\mathcal{T}_{ns}$ is dense in $\mathcal{T}$.
For any lower semicontinuous function $f$ on a metric space, density implies
$\sup_{t\in\mathcal{T}} f(t) = \sup_{t\in\mathcal{T}_{ns}} f(t)$, because for
any $t\in\mathcal{T}$ one can pick $t_m\in\mathcal{T}_{ns}$ with $t_m\to t$ and
then $f(t) \leq \liminf_m f(t_m) \leq \sup_{\mathcal{T}_{ns}} f$.
Applying this with $f=X_n$ yields
\[
    \Lambda_n = \sup_{t\in\mathcal{T}} X_n(t) = \sup_{t\in\mathcal{T}_{ns}} X_n(t) = \Lambda_n^{(ns)}.
\]

Next, define the extension $X$ as in the corollary: extend $Z$ to be almost
surely continuous on $\mathcal{T}$, and for $t\in\mathcal{T}_s$ define
$\Delta(t)$ as the inner limit of $\Delta(t')$ as $t'\to t$. Then $X(t)$ is given
by
\[
    X(t) = \sup_{\delta\in\Delta(t)} \{2\delta^\tsp Z(t) - \|\delta\|_2^2\}.
\]
The same argument as above (with $\xi$ replaced by $\delta$ and using inner
semicontinuity of $\Delta(\cdot)$ and continuity of $Z$) shows that $X$ is lower
semicontinuous on $\mathcal{T}$. Hence, again by density of $\mathcal{T}_{ns}$,
\[
    \Lambda = \sup_{t\in\mathcal{T}} X(t) = \sup_{t\in\mathcal{T}_{ns}} X(t) = \Lambda^{(ns)}.
\]

Finally, Theorem \ref{thm:main_singular} gives $\Lambda_n^{(ns)}\Lto\Lambda^{(ns)}$.
Since $\Lambda_n=\Lambda_n^{(ns)}$ and $\Lambda=\Lambda^{(ns)}$, we conclude
$\Lambda_n\Lto\Lambda$.
\end{proof}

\begin{proof}[Proof of Proposition~\ref{prop:assm_2_sufficient}]
    For a multi-index $k = (k_1, \dots, k_p)$ of non-negative integers, write
    $|k| = k_1 + \cdots + k_p$, $k! = k_1! \cdots k_p!$,
    $\partial_\xi^k = \partial^{|k|}/\partial\xi_1^{k_1}\cdots\partial\xi_p^{k_p}$,
    and $(\xi-\xi_0)^k = (\xi_1-\xi_{0,1})^{k_1}\cdots(\xi_p-\xi_{0,p})^{k_p}$.
    Condition~(iii) of Proposition~\ref{prop:assm_2_sufficient} means formally that
    there exists a measurable $M:\mathcal{Y}\to\R$ with $\E\{M(Y_1)\}<\infty$ and,
    for some $c_1>0$,
    \[
        \sup_{\theta\in\Theta:\|\xi-\xi_0\|_2\leq c_1} |\partial_\xi^k \log f(y;\theta)|
        \leq M(y)
    \]
    for all $k$ with $|k|=3$ and $\nu$-almost every $y\in\mathcal{Y}$.
    Recall \eqref{eq:ln_expand} and note, by Taylor's theorem, for any fixed $t
    \in \mathcal{T}$, there exists some $\tilde{\xi}$ between $\xi_0$ and $\xi$
    such that
    \begin{align*}
         n\Vert I(t)^{1/2}(\xi - \xi_0)\Vert_2^2 R_n(t, \xi) &= \frac{1}{3!} \sum_{k: |k| = 3} \frac{3!}{k!} \partial_\xi^k \ell_n(\tilde{\xi}, t)(\xi - \xi_0)^k + \frac{n}{2}(\xi - \xi_0)^\tsp \{I(t) - \hat{I}_n(t)\}(\xi - \xi_0).
    \end{align*}
    By the triangle inequality and condition~(iii) of Proposition~\ref{prop:assm_2_sufficient}, the absolute value of the first term is upper bounded by
    \begin{align*}
         \frac{1}{3!} \sum_{k: |k| = 3} \frac{3!}{k!}  \sum_{i = 1}^n M(Y_i) |\xi - \xi_0|^k &= \frac{1}{3!}  \sum_{i = 1}^n M(Y_i) (|\xi - \xi_0|^\tsp 1_p)^3 \leq \frac{1}{3!} p^{3/2} \Vert \xi - \xi_0\Vert_2^3 \sum_{i = 1}^n M(Y_i).
    \end{align*}
    The absolute value of the second term is bounded by $(n/2)\Vert \xi -
    \xi_0\Vert_2^2 \Vert I - \hat{I}_n\Vert_{\mathcal{T}}$. Combining the two bounds and using $\Vert I(t)^{1/2}(\xi-\xi_0)\Vert_2^2 \geq \underline{\kappa}\Vert\xi-\xi_0\Vert_2^2$ from Assumption~\ref{assm:inf_bound} gives
    \[
        |R_n(t, \xi)| \leq \underline{\kappa}^{-1}\left\{\frac{1}{3!} p^{3/2} \Vert \xi - \xi_0\Vert_2 n^{-1} \sum_{i = 1}^n M(Y_i) + \frac{1}{2} \Vert I - \hat{I}_n\Vert_{\mathcal{T}}\right\}.
    \]
    The first term is asymptotically uniformly equicontinuous almost surely, and
    hence in outer probability, since $n^{-1}\sum_{i = 1}^n M(Y_i) \to
    \E\{M(Y_1)\}$ almost surely by the strong law of large numbers. The second
    term does not depend on $\xi$ or $t$ and is $o_{p^*}(1)$ by
    condition~(ii), so the sum is asymptotically uniformly
    equicontinuous in outer probability, completing the proof.
\end{proof}

 \begin{proof}[Proof of Theorem \ref{thm:link_quad}]
 With a slight abuse of notation, write $\ell_n(\xi,t)$ for the linkage log
 likelihood with $t$ and $\xi$ as in \re{tDef}, so that
 $\xi_0=(0,0)$ corresponds to the null. Let $\va=(\va_1,\va_2)$ be such that
 $\xi=(\va_1^2,\va_2^2)$ with $\va_1,\va_2\ge 0$, and write $\ell_n(t,\va)$ when
 we view the same log-likelihood as a function of $\va$.

 Following \citet{rotnitzky2000likelihoodbased}, expand $\ell_n(t,\va)$ around $\va=0$ up to
 order $4$:
 \begin{equation}
 \ell_n(t,\va) = \ell_n(t,0) + \sum_{k:\,1\le |k|\le 4} \ell_n^{(k)}(t,0)\frac{\va^k}{k!}
 + \bar{R}_n(t;\va),
 \lb{lneps}
 \end{equation}
 where $\ell_n^{(k)}(t,\va)=\partial^k \ell_n(t,\va)/\partial \va^k$,
 $k=(k_1,k_2)$, and $\va^k=\va_1^{k_1}\va_2^{k_2}$. Moreover,
 $\ell_n^{(k)}(t,0)=\sum_{i=1}^n \psi_k(X_i;t)$, where
 \begin{equation}
 \psi_k(x;t) = \frac{\partial^k \log f_{(t,\va)}(x)}{\partial \va^k}
 = \frac{\partial^k \log \pr_{\mbox{\scr p},\ga}(v(\tau)|\Phi)}{\partial \va^k}.
 \lb{psik}
 \end{equation}
 Using results from \citet{hossjer2005conditional} and repeated differentiation with respect
 to $\va$ at $\va=0$, one obtains in particular that the only nonzero
 contributions up to order $4$ come from the even derivatives, with
 \[\psi_{20}(x;t)=2S_1(v(\tau)),\qquad \psi_{02}(x;t)=2S_2(v(\tau)),\qquad \psi_{11}(x;t)=0,\]
 where $S=(S_1,S_2)$ is as in \eqref{S1S2}, and
 \[\psi_{40}(x;t)=-3I_{11}(t,t)+r_{40}(x;t),\quad
     \psi_{22}(x;t)=-I_{12}(t,t)+r_{22}(x;t),\quad
     \psi_{04}(x;t)=-3I_{22}(t,t)+r_{04}(x;t),\]
 with $I(t,t)=E_0\{S(v(\tau))^\tsp S(v(\tau))\}$ as in \re{IDefLink} and
 $E_0\{r_k(X;t)\}=0$ for $k\in\{(4,0),(2,2),(0,4)\}$.

 Substituting $\xi=(\va_1^2,\va_2^2)$ into \re{lneps} and collecting terms, we
 arrive at the quadratic expansion in the notation of Section~2:
 \begin{equation}
 \ell_n(\xi,t)-\ell_n(\xi_0)
 = \sqrt{n}(\xi-\xi_0)^\tsp U_n(t)
 -\frac{n}{2}(\xi-\xi_0)^\tsp I(t)(\xi-\xi_0)
 + n\|I(t)^{1/2}(\xi-\xi_0)\|_2^2 R_n(\xi,t),
 \lb{lnxi}
 \end{equation}
 where
 \[
 U_n(t)=n^{-1/2}\nabla_\xi\ell_n(\xi_0,t),\qquad \nabla_\xi\ell(\xi_0,t)=S(v(\tau);K)\ \text{as in \eqref{S1S2}},
 \]
 and $I(t)=I(t,t)$ from \re{IDefLink}. This identifies the $U_n(t)$ and $I(t)$
 appearing in \eqref{eq:ln_expand}, and yields $Z_n(t)=A(t)^{-1}U_n(t)$ as in
 Section~2.

 It remains to check that the remainder can be written in the form
 $n\|I(t)^{1/2}(\xi-\xi_0)\|_2^2 R_n(\xi,t)$ with $R_n$ satisfying Assumption
 \ref{assm:equicont}. By Taylor's theorem, the remainder $\bar{R}_n(t;\va)$ in
 \re{lneps} can be expressed using derivatives of order $5$ evaluated at points
 between $0$ and $\va$. Under a regularity condition analogous to the ones used
 in Section~\ref{sec:asy_dist} (existence of an integrable envelope that bounds
 the order-$5$ derivatives uniformly over $t\in\cT$ and $\va$ in a neighborhood
 of $0$), we obtain the uniform bound
 \[
 |\bar{R}_n(t;\va)| \le \Bigl(\sum_{i=1}^n M(X_i)\Bigr)\,\|\va\|_2^5
 \qquad \text{for all $t\in\cT$ and $\|\va\|_2$ small},
 \]
 for some $M$ with $E_0\{M(X_1)\}<\infty$. Since $\|\va\|_2^4=\|\xi-\xi_0\|_2^2$ and, by Assumption~\ref{assm:inf_bound}, $\|I(t)^{1/2}(\xi-\xi_0)\|_2^2 \geq \underline{\kappa}\|\xi-\xi_0\|_2^2$,
 this yields
 \[
 \sup_{t\in\cT}\sup_{\xi\in \bar{B}_{c_3}(\xi_0;t)}
 \frac{|\bar{R}_n(t;\va)|}{n\|I(t)^{1/2}(\xi-\xi_0)\|_2^2}
 \le \underline{\kappa}^{-1}\Bigl(\frac{1}{n}\sum_{i=1}^n M(X_i)\Bigr)\,c_3^{1/2},
 \]
 and the right-hand side is $O_{p}(1)$ by the law of large numbers.
 Choosing $c_3$ small enough gives Assumption \ref{assm:equicont}, completing
 the proof.

 \end{proof}

\subsection{Details for Section~\ref{sec:linkage}}\label{sec:link_details}

\paragraph{Derivation of $\pr_{\scr{p},\ga}(v(\tau)|\Phi)$ in \re{fDefLink}.}
Under $H_1$, the phenotype vector $\Phi= (\Phi_k,\, k\in\cP)$ of pedigree $\cP$ depends on the inheritance vector $v(\tau)$ of the family at the disease locus $\tau \in [0,T]$, and marker data $Y=\{v(s); \, 0\le s \le T\}$ of the family is conditionally independent of $\Phi$ given $v(\tau)$. Since the null distribution of $v(\tau)$ is uniform \re{PUnif}, by Bayes' rule,
\begin{equation}
\pr_{\scr{\rmp},\ga}(v(\tau)|\Phi) \propto \pr_{\scr{\rmp},\ga}(\Phi|v(\tau)).
\lb{PxiDef}
\end{equation}
Let $\cF$ be the set of founders of $\cP$ and $F=|\cF|$ the number of founders. Then $|\cP|=N = F + m/2$ is the total number of pedigree
members, and we let $G = (G_1,\ldots,G_N)$ refer to the genotypes at $\tau$, where
$G_k = (a_{2k-1},a_{2k})$ contains the two alleles of individual $k$. For a
biallelic gene with $a_k\in\{0,1\}$ and probability $\pr(1)=\rmp$ of the disease allele 1, the founder alleles
$a = (a_1,\ldots,a_{2F})$ have distribution
$\pr_{\scr{\rmp}}(a) = \rmp^{|a|}\rmq^{2F-|a|}$, with $|a|=\sum a_k$ and $\rmq = 1-\rmp$.
Since genotypes of nonfounders are determined by $a$ and $v=v(\tau)$ via
$G = G(a,v)$, conditioning on $a$ gives
\begin{equation}
\pr_{\scr{\rmp},\ga}(\Phi|v) = \sum_a \rmp^{|a|}\rmq^{2F-|a|}\,\pr_\ga(\Phi|G(a,v)).
\lb{Ppizet}
\end{equation}
For a monogenic binary phenotype ($\Phi_k\in\{0,1,?\}$) with conditionally
independent phenotypes given genotypes (a typical assumption of a monogenic disease),
\begin{equation}
\pr_\ga(\Phi|G) = \prod_{k=1}^N \pr_\ga(\Phi_k|G_k),\quad
\pr_\ga(\Phi_k|G_k) = \ga_{|G_k|}^{\{\Phi_k=1\}}(1-\ga_{|G_k|})^{\{\Phi_k=0\}},
\lb{Ppi}
\end{equation}
and $|G_k|=a_{2k-1}+a_{2k}$. 

\paragraph{Phenotype-based weights for pairs of individuals, for score function \re{S1S2}.}

It can be shown \citep{hossjer2005conditional} that the weight assigned to a pair $(k,l)$ of individuals, for the score vector components $S_1$ and $S_2$ in \re{S1S2}, is given by
\begin{equation}
\om_{kl} = \frac{\partial^2 \pr(\Phi|G)/\partial\ga_{|G_k|}\partial\ga_{|G_l|}
\big|_{\ga=(K,K,K)}}{\pr(\Phi)},
\lb{omDef}
\end{equation}
with $\pr(\Phi|G)$ the conditional probability defined in \re{Ppi}. 

\paragraph{Reparametrization of penetrance parameters.}
For fixed $\rmp$, introduce the inner product $(x,y) = \rmq^2 x_0 y_0 +
2\rmq\rmp x_1 y_1 + \rmp^2 x_2 y_2$ on $\R^3$ and the orthonormal basis
\[
e_0 = (1,1,1),\quad
e_1 = \frac{(-2\rmp,\,\rmq-\rmp,\,2\rmq)}{\sqrt{2\rmp\rmq}},\quad
e_2 = \bigl(\rmq^{-1}-1,\,-1,\,\rmp^{-1}-1\bigr).
\]
Writing $\ga = Ke_0 + \va_1 e_1 + \va_2 e_2$ as in \eqref{zetaExp}, one checks
that $K = \E(\Phi_k)$ is the disease prevalence and $\va_1^2 + \va_2^2 =
\Var(\E(\Phi_k|G_k))$ is the total genetic variance split into additive
($\xi_1=\va_1^2$) and dominance ($\xi_2=\va_2^2$) components. 

\noindent {\bf Derivation of \re{rhos}.} We generalize Theorem 2 of \citet{hossjer2005spectral} from one-
to two-dimensional score functions. We introduce the space $\cA$ of
mappings $\{0,1\}^m \to \R$, so that both components of the score function $S=(S_1,S_2)^\tsp$ are elements
of $\cA$. Endow $\cA$ with the scalar product
$$
\langle S_1,S_2 \rangle = 2^{-m}\sum_w S_1(w)S_2(w).
$$
Given any $w\in\{0,1\}^m$, let $S_w(u) = (-1)^{w\cdot u}$,
where $w\cdot u = \sum_{l=1}^m w_l u_l$ is the vector dot product of $w$ and $u$.
Then $\{S_w\}$ is an orthonormal basis on $\cA$ and $S$ can be expanded as
$$
S = \sum_{w} R_S(w)S_w,
$$
where $R_S(w)=(R_{S_1}(w),R_{S_2}(w))^\tsp$ and $R_{S_k}(w) = \langle S_k,S_w \rangle$. Notice
that $R_S(0)=0$ since both $S_1$ and $S_2$ are standardized to have mean zero,
under the null hypothesis of no linkage, i.e.\ $E_{\xi_0}(S_k)=\langle S_k,S_0 \rangle=0$. It follows
from \re{IDefLink} and the proof of Theorem 2 in \citet{hossjer2005spectral} that
\begin{equation}
I(t,\tpr) = \sum_{w\ne 0} R_S(w)R_S(w)^\tsp\exp(-2|w||\tau^\prime-\tau|),
\lb{CovExp}
\end{equation}
where $|w|=\sum_{l=1}^m w_l$ is the number of one-components of $w$. Combining
\re{CovExp} with \re{rho} and \re{rhostat} we arrive at \re{rhos}, with
\beq
\begin{array}{rcl}
\ka_l &=& A(t)^{-1}\sum_{w;|w|=l} R_S(w)R_S(w)^\tsp A(t)^{-\tsp}\\
&=& B^{-1}\sum_{w;|w|=l} R_{S_0}(w)R_{S_0}(w)^\tsp B^{-\tsp}.
\end{array}
\lb{kappal}
\eeq
where $S_0=K^2 S = (\Spairs-E_{\xi_0}(\Spairs),\Sgprs-E_{\xi_0}(\Sgprs))^\tsp$ and $B=K^2 A(t)$ are both
independent of $t$.
\slut

\paragraph{MLS score for affected sib pairs.}
For affected sib pairs (pedigree type 1 in Figure~\ref{Pedigrees}), the MLS
parametrization uses parameter vector $\theta=(\xi,t)$, with
\begin{equation}
t = \tau, \quad \xi = (z_0,z_1)^\tsp,
\lb{xiMLS}
\end{equation}
where $z_i$ is the probability that the sib pair shares $i$ alleles IBD at
$\tau$, with a simple null hypothesis value $\xi_0 = (0.25,0.5)^\tsp$ and parameter space for the structural parameter $\xi$ the possible triangle
\citep{holmans1993asymptotic}:
\begin{equation}
\Xi = \{\xi:\,z_0\ge 0,\,z_1\le 0.5,\,2z_0\le z_1\}.
\lb{XiMLS}
\end{equation}
An affected sib pair has $m=4$ meioses and marker data $Y=\{v(s)=(v_1(s),\ldots,v_4(s)); \, 0\le s \le T\}$. Number the four family members so that the parents have numbers 1,2 and the siblings numbers 3,4. The density of marker data $Y$ conditionally on the phenotype vector $\Phi=(?,?,1,1)$ is
\begin{equation}
f(y;\theta)=\pr_\xi(\IBD(\tau)|\Phi)\,\pr(v(\tau)|\IBD(\tau))\,\pr(y|v(\tau)),
\lb{fMLS}
\end{equation}
where $\pr(y|v(\tau))=\exp(-mT)=\exp(-4T)$ is a constant, $\IBD(\tau)$ is the number of alleles shared IBD by the sib pair at the disease locus $\tau$, $\pr_\xi(\IBD|\Phi) = z_0^{\{\scr{\IBD}=0\}}z_1^{\{\scr{\IBD}=1\}}(1-z_0-z_1)^{\{\scr{\IBD}=2\}}$, and $\pr(v(\tau)|\IBD(\tau)) = 1/8$ if $\IBD(\tau)=1$ and $1/4$ otherwise. The per-family score vector is
\begin{equation}
\begin{array}{rcl}
\nabla_\xi \log f(y;\xi,t) &=& \left(4\cdot\mathbf{1}_{\{\IBD(\tau)=0\}} - 4\cdot\mathbf{1}_{\{\IBD(\tau)=2\}},\;
2\cdot\mathbf{1}_{\{\IBD(\tau)=0\}} - 4\cdot\mathbf{1}_{\{\IBD(\tau)=2\}}\right)^\tsp\\
&=& S(v(\tau)),
\end{array}
\lb{psiMLS}
\end{equation}
where $S(v(\tau))$ is as in \eqref{Saff} for pedigree type 1, so that the
information matrices agree:
\begin{equation}
I(t) = \left(\begin{array}{cc} 8 & 4\\ 4 & 6 \end{array}\right).
\lb{I0MLS}
\end{equation}
From \eqref{XiMLS}, the tangent cone at $\xi_0$ of $\Xi - \xi_0$ is
$C = \{(c_1,c_2)^\tsp:\,c_2\le 0,\,2c_1\le c_2\}$, which in matrix form \re{Ct2} corresponds to $C=C_V$, with
\begin{equation}
V = \left(\begin{array}{cc} -2 & 1\\ 0 & -1\end{array}\right).
\lb{VMLS}
\end{equation}
Inserting \eqref{I0MLS} and \eqref{VMLS} into $\De(t) = A(t)^\tsp C = \De$
and the definition of $w_2$ in \re{w2DefLink}, and writing $\De = C_U$ with rows of $U$
normalized to unit length, one obtains
$$
U = \left(\begin{array}{cc} -0.5774 & 0.8165\\ 0 & -1.0000\end{array}\right),
\quad w_2 = 0.0980,
$$
in agreement with the affected sib pair entry of Table~\ref{Tab1}.

\newpage

\begin{figure}[htbp]
\begin{center}
\centerline{\resizebox{120mm}{!}{\includegraphics{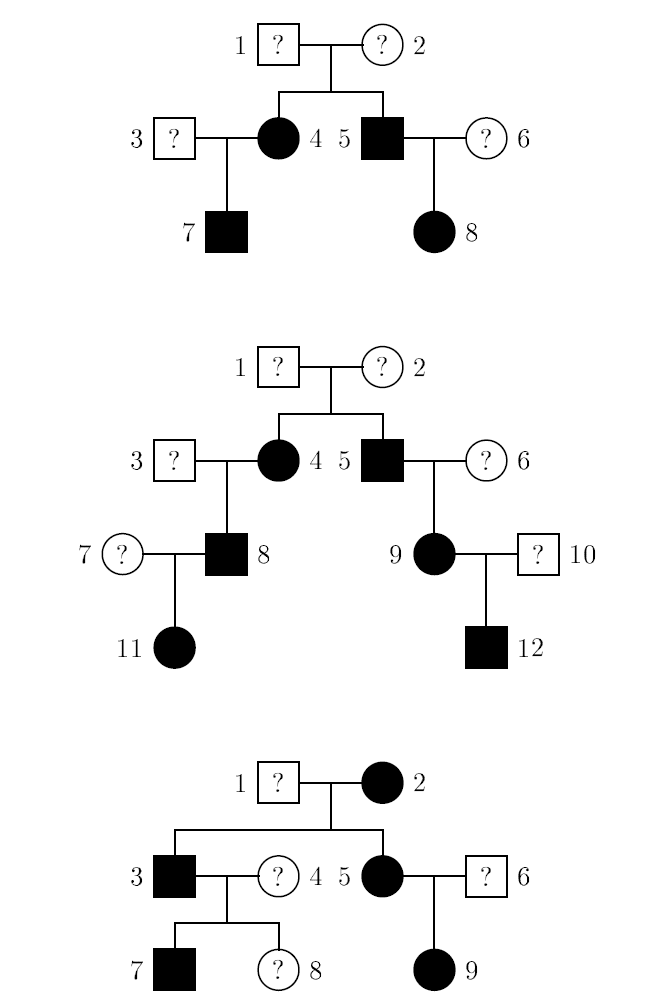}}}
\caption{Pedigree types $(\cP_j,\Phi_j)$ used in Tables \ref{Tab2}-\ref{Tab1}. For $j=1,2,3,4$, 
$\cP_j$ consists of two parents with unknown phenotypes and $k+1$ affected offspring.
(In particular, $(\cP_1,\Phi_1)$ is an affected sib pair.) $\cP_5$ (upper), $\cP_6$ (middle) and
$\cP_7$ (lower) are shown above with individual numbers. 
Males and females correspond to squares and circles, affected individuals have black and unaffected ones have
white symbols. Individuals with unknown phenotypes have question marks.}
\lb{Pedigrees}
\end{center}
\end{figure}

\newpage

\begin{table}
\caption{Matrices $\ka_l=(\ka_{l,uv})_{u,v=1}^2$ in \eqref{rhos} for the pedigree types
$(\cP_j,\Phi_j)$ of Figure~\ref{Pedigrees}. The upper triangular square root $A(t)^\tsp$ of $I(t)$ is used in \re{kappal}.}
\lb{Tab2}
\begin{center}
\begin{tabular}{|c|c|c|c|}
\hline
$j$ & $\ka_1$ & $\ka_2$ & $\ka_3$\\
\hline
\hline
1-4 & $\left(\begin{array}{cc} 0 & 0 \\ 0 & 0 \end{array}\right)$ &
$\left(\begin{array}{cc} 1 & 0 \\ 0 & 0 \end{array}\right)$ &
$\left(\begin{array}{cc} 0 & 0 \\ 0 & 0 \end{array}\right)$\\
\hline
5 &
$\left(\begin{array}{cc} 0 & 0 \\ 0 & 0 \end{array}\right)$ &
$\left(\begin{array}{cc} 0.8137 & 0.1976\\ 0.1976 & 0.0821\end{array}\right)$ &
$\left(\begin{array}{cc} 0.1765 & -0.1872\\-0.1872 & 0.1985\end{array}\right)$ \\
\hline
6 &
$\left(\begin{array}{cc} 0.1914 & -0.1375\\-0.1375 & 0.0988\end{array}\right)$ &
$\left(\begin{array}{cc} 0.5455 & 0.3265\\0.3265 & 0.2601\end{array}\right)$ &
$\left(\begin{array}{cc} 0.2161 & -0.1552\\-0.1552 & 0.1115\end{array}\right)$ \\
\hline
7 &
$\left(\begin{array}{cc} 0.1356 & -0.1246\\-0.1246 & 0.1144\end{array}\right)$ &
$\left(\begin{array}{cc} 0.7034 & 0.2725\\0.2725 & 0.1352\end{array}\right)$ &
$\left(\begin{array}{cc} 0.1525 & -0.1401\\-0.1401 & 0.1287\end{array}\right)$ \\
\hline
\hline
$j$ & $\ka_4$ & $\ka_5$ & $\ka_6$ \\
\hline
1-4 &
$\left(\begin{array}{cc} 0 & 0 \\ 0 & 1 \end{array}\right)$ &
$\left(\begin{array}{cc} 0 & 0 \\ 0 & 0 \end{array}\right)$ &
$\left(\begin{array}{cc} 0 & 0 \\ 0 & 0 \end{array}\right)$ \\
\hline
5 & $\left(\begin{array}{cc} 0.0098 & -0.0104\\-0.0104 & 0.7194\end{array}\right)$ &
$\left(\begin{array}{cc} 0 & 0 \\ 0 & 0 \end{array}\right)$ &
$\left(\begin{array}{cc} 0 & 0 \\ 0 & 0 \end{array}\right)$\\
\hline
6 &
$\left(\begin{array}{cc} 0.0425 & -0.0306\\-0.0306 & 0.5273\end{array}\right)$ &
$\left(\begin{array}{cc} 0.0043 & -0.0031\\-0.0031 & 0.0022\end{array}\right)$ &
$\left(\begin{array}{cc} 0.0002 & -0.0002\\ -0.0002 & 0.0001\end{array}\right)$ \\
\hline
7 &
$\left(\begin{array}{cc} 0.0085 & -0.0078\\-0.0078 & 0.6217\end{array}\right)$ &
$\left(\begin{array}{cc} 0 & 0 \\ 0 & 0 \end{array}\right)$ &
$\left(\begin{array}{cc} 0 & 0 \\ 0 & 0 \end{array}\right)$ \\
\hline
\end{tabular}
\end{center}
\end{table}

\newpage

\begin{center}
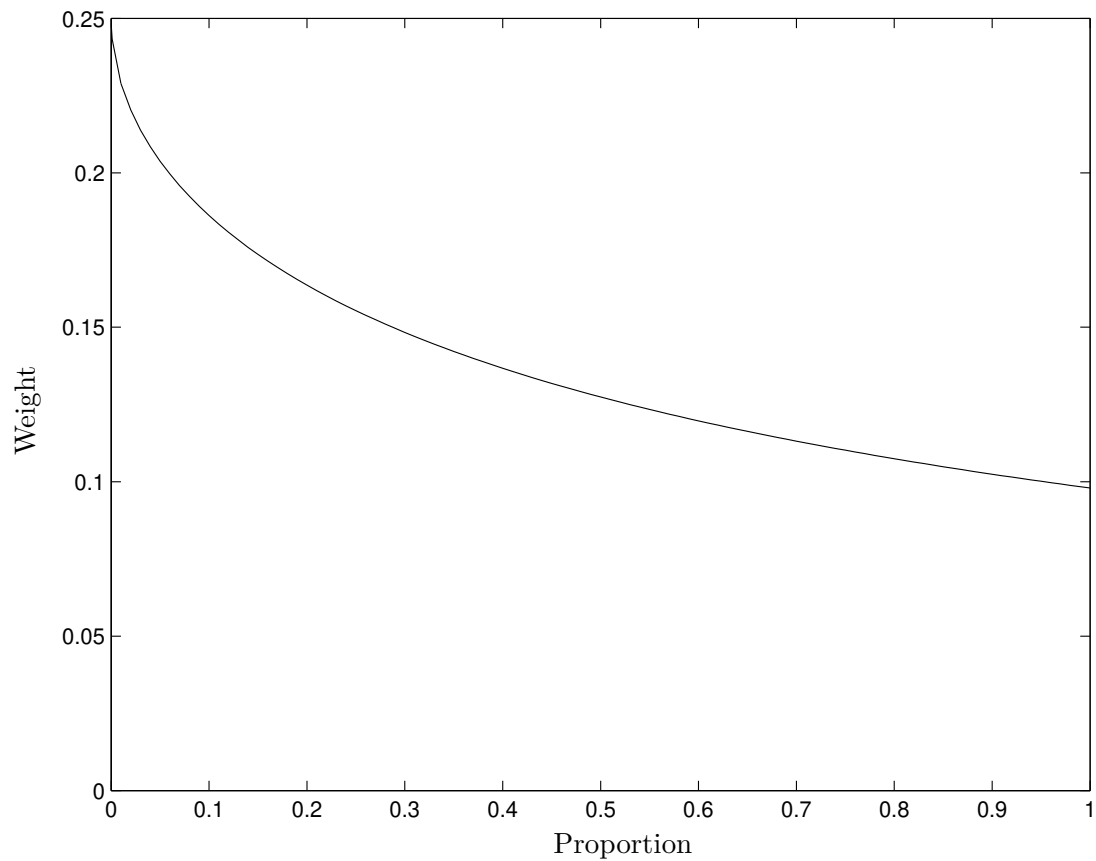
\begin{figure}[htbp]
\centerline{%
\begin{overpic}[width=150mm]{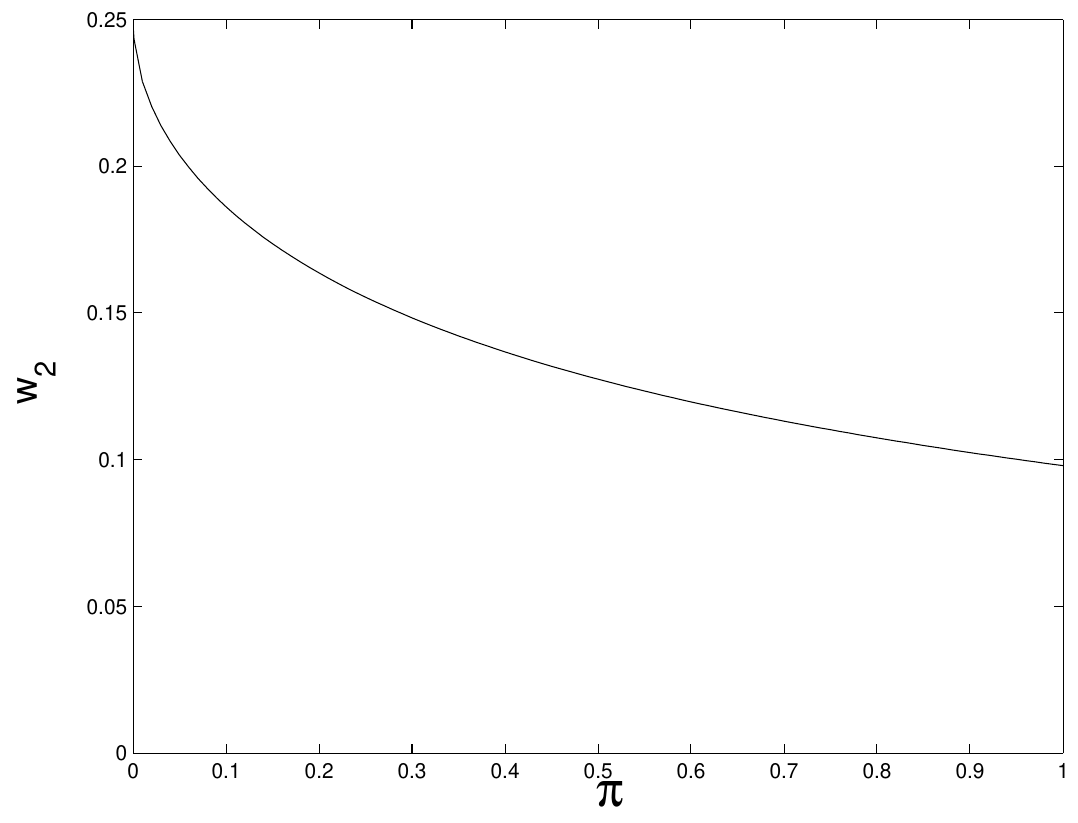}
  \put(55,0.6){\textcolor{white}{\rule{8mm}{5mm}}}
  \put(57.5,1.6){\makebox(0,0){\small Proportion}}
  \put(1,31){\textcolor{white}{\rule{8mm}{22mm}}}
  \put(5,40){\rotatebox{90}{\makebox(0,0){\small Weight}}}
\end{overpic}}
\caption{Plot of $\chi^2_2$-weight $w_2$ as function of the proportion $\beta$ of
affected sib pairs for a mixture of affected sib and first cousin pairs.}
\lb{Fig2}
\end{figure}
\end{center}

\newpage

\begin{table}
\caption{Values of Fisher information matrix $I(t)$, boundary region $\Delta=C_U$ and $\chi^2_2$-weight $w_2$ (see \re{w2DefLink}) for the pedigree types of Figure \ref{Pedigrees}. The upper triangular square root $A(t)^\tsp$ of $I(t)$ is used for calculating $U(t)=A(t)^{-\tsp}$. After normalizing the two rows of $U(t)$ to have unit length, the resulting matrix $U$ does not depend on $t$.}
\begin{center}
\begin{tabular}{|c|c|c|c|}
\hline
Pedigree type & $I(t)$ & $U$ & $w_2$ \\
\hline
\hline
$(\cP_1,\Phi_1)$ & $K^{-4}\cdot\left( \begin{array}{cc} 0.1250 & 0.1250\\ 0.1250 & 0.1875 \end{array}\right)$ 
& $\left( \begin{array}{cc} 0.5774 & -0.8165\\ 0 & 1.0000 \end{array} \right)$ & 0.0980 \\
\hline
$(\cP_2,\Phi_2)$ & $K^{-4}\cdot\left( \begin{array}{cc} 0.3750 & 0.3750\\0.3750 & 0.5625\end{array}\right)$
& $\left( \begin{array}{cc} 0.5774 & -0.8165\\ 0 & 1.0000 \end{array} \right)$
& 0.0980\\
\hline
$(\cP_3,\Phi_3)$ & $K^{-4}\cdot\left( \begin{array}{cc}0.7500 & 0.7500\\0.7500 & 1.1250 \end{array} \right)$
& $\left( \begin{array}{cc} 0.5774 & -0.8165\\ 0 & 1.0000 \end{array} \right)$
& 0.0980\\
\hline
$(\cP_4,\Phi_4)$ & $K^{-4}\cdot\left( \begin{array}{cc} 1.2500 & 1.2500\\1.2500 & 1.8750\end{array} \right)$
& $\left( \begin{array}{cc} 0.5774 & -0.8165\\ 0 & 1.0000 \end{array} \right)$
& 0.0980\\
\hline
$(\cP_5,\Phi_5)$ & $K^{-4}\cdot\left( \begin{array}{cc} 0.7969 & 0.2813\\ 0.2813 & 0.1875 \end{array} \right)$
& $\left( \begin{array}{cc} 0.6860 & -0.7276 \\0 & 1.0000\end{array} \right)$
& 0.1203\\
\hline
$(\cP_6,\Phi_6)$ & $K^{-4}\cdot\left( \begin{array}{cc} 2.2959 & 0.3828\\0.3828 & 0.1875\end{array}\right)$
& $\left( \begin{array}{cc} 0.8121 & -0.5835\\0 & 1.0000\end{array} \right)$
& 0.1508\\
\hline
$(\cP_7,\Phi_7)$ & $K^{-4}\cdot\left( \begin{array}{cc} 0.9219 & 0.2813\\0.2813 & 0.1875\end{array} \right)$
& $\left( \begin{array}{cc} 0.7365 & -0.6765\\0 & 1.0000\end{array} \right)$
& 0.1318\\
\hline
\end{tabular}
\end{center}
\lb{Tab1}
\end{table}
\end{appendix}

\end{document}